\documentclass[11pt,reqno]{amsart}

\usepackage[T1]{fontenc}
\usepackage[utf8]{inputenc}

\usepackage{hyperref}

\allowdisplaybreaks[4]
\usepackage{amsmath}
\usepackage{amssymb}
\usepackage{mathrsfs}
\usepackage{amsthm}
\usepackage{bm}
\usepackage{graphicx}
\usepackage{tikz}
\usepackage{booktabs}

\usepackage[caption=false]{subfig}
\usepackage{caption}

\usepackage{float}

\newfloat{longfigure}{htbp}{lof}
\floatname{longfigure}{Figure}
\newsubfloat{longfigure}

\usepackage{geometry}
\usepackage{color}

\usepackage{color}
\theoremstyle{plain}
\newtheorem{lemma}{Lemma}[section]
\newtheorem{theorem}[lemma]{Theorem}
\newtheorem{definition}[lemma]{Definition}

\newtheorem{remark}[lemma]{Remark}
\newtheorem{proposition}[lemma]{Proposition}
\newtheorem{assumption}{Assumption}
\newtheorem{example}{Example}

\begin{document}
\title[Schemes for HJE on Wasserstein space on graphs]{Finite difference schemes for Hamilton--Jacobi equation on Wasserstein space on graphs}
\author{Jianbo Cui, Tonghe Dang, Chenchen Mou}
\address{Department of Applied Mathematics, The Hong Kong Polytechnic
University, Hung Hom, Kowloon, Hong Kong; Department of Mathematics, City University of Hong Kong, Hong Kong, SAR, China.}
\email{jianbo.cui@polyu.edu.hk;tonghe.dang@polyu.edu.hk(Corresponding author);chencmou@cityu.edu.hk}
\thanks{This research is partially supported by MOST National Key R\&D Program No. 2024YFA1015900, the Hong Kong Research Grant Council GRF grant 15302823, NSFC grant 12301526, NSFC/RGC Joint Research Scheme N PolyU5141/24, and the CAS AMSS-PolyU Joint Laboratory of Applied Mathematics.}
\begin{abstract}
This work proposes and studies numerical schemes for initial value problems of Hamilton--Jacobi equations (HJEs) with a graph individual noise on the Wasserstein space on graphs. Numerically solving such equations is particularly challenging due to the structural complexity caused by discrete geometric derivatives and logarithmic geometry. Our numerical schemes are constructed using finite difference approximations that are adapted to both the discrete geometry of graphs and the differential structure of Wasserstein spaces. To ensure numerical stability and accuracy of numerical behavior, we use extrapolation-type techniques to simulate the numerical solution on the boundary of density space. By analyzing approximation error of Wasserstein gradient of the viscosity solution, we prove the uniform convergence of the schemes to the original initial value problem, and establish an $L^{\infty}_{\mathrm{loc}}$-error estimate of order one-half. Several numerical experiments are presented to illustrate our theoretical findings and to study the effect of individual noise and Hamiltonians on graphs. To the best of our knowledge, this is the first result on numerical schemes for HJEs on the Wasserstein space with a graph structure.  
\end{abstract}
%MSC code. {49L25, 65M15, 35R02.} 
%35D40(2010–now)Viscosity solutions to PDEs
%35F21(2010–now)Hamilton-Jacobi equations For calculus of variations and optimal control, see 49Lxx; for mechanics of particles and systems, see 70H20
%49L25(1991–now)Viscosity solutions to Hamilton-Jacobi equations in optimal control and differential games%65M12(1991–now)Stability and convergence of numerical methods for initial value and initial-boundary value problems involving PDEs
%65M15(1973–now)Error bounds for initial value and initial-boundary value problems involving PDEs
%35R02(2010–now)PDEs on graphs and networks (ramified or polygonal spaces)
\maketitle
\section{Introduction} 

Hamilton–Jacobi equations (HJEs) originated in classical mechanics as a reformulation of Hamiltonian dynamics and has since become central in fields such as dynamic programming, geometric optics, and optimal control. In recent years, motivated by applications in
mean field control and potential mean field game problems \cite{CardaliaguetPorretta,CarmonaDelarue1,CarmonaDelarue2,MFG_Caines,LasryLions}, the study of HJEs  has been extended to the Wasserstein space, which provides a natural setting for analyzing
dynamics over probability measures. Significant progress has been made about HJEs on Wasserstein space in the continuum setting in recent years, including the development of viscosity solution theories based on sub- and super- differentials \cite{GNT}, as well as L-viscosity solution framework that lifts the equation to a Hilbert space of $L^2$ random variables \cite{CarmonaDelarue1,CarmonaDelarue2,GT}. We refer to \cite{BCP,BFY,CDJS,CGKPR,Daudin,Daudin2,approximation_HJB,approximation_HJB2,Zhang} for further developments of the HJE on Wasserstein space in the continuum setting. In parallel, the increasing importance of graph-structured data in modern applications has spurred interest in discrete partial differential equations on graphs. Although discrete optimal transport, gradient flows, and Hamiltonian systems on graphs have been actively explored \cite{Chow2,Cui_MC22,Cui_Hamiltonian,Maas,Mielke}, the study of HJEs on the Wasserstein space on graphs remains in its early stages.

In this paper, we focus on the numerical analysis of the following HJE \cite{MCC} on the Wasserstein space on a finite graph $G=(V,E,\omega)$ 
 \begin{align}\label{HJeq1}
\partial_tu(t,\xi)+\mathcal H(\xi,\nabla_{\mathcal W}u(t,\xi))+\mathcal F(\xi)=\Delta_{\mathrm{ind}}u(t,\xi),\quad u(0,\xi)=\mathcal U_0(\xi),
\end{align}
where $\mathcal U_0$ is the initial value,  $\mathcal F$ is the potential function, and $\mathcal H$ is the Hamiltonian function. 
Here, $t\in(0,T)$ with some $T>0,$ and the state variable $\xi$ belongs to the probability simplex of the graph, i.e., $\xi\in\mathcal P(G)$, 
where the graph $G$ is weighted, undirected and connected. 
In \eqref{HJeq1}, $\Delta_{\mathrm{ind}}u(t,\xi):=-( \nabla_{\mathcal W}u(t,\xi),\nabla_G\log \xi)_{\xi}$ is the graph individual noise. 
Here, $\nabla_{\mathcal W}$ is the Wasserstein gradient operator on the Wasserstein space on graphs, $(\cdot,\cdot)_{\xi}$ is an inner product acting on the set of skew-symmetric matrices $\mathbb S^{d\times d},$ and $\nabla_G$ is the graph gradient operator, whose adjoint operator for the inner product $(\cdot,\cdot)_{\xi}$ is $-\mathrm{div}_{\xi}$.  This noise term can be associated with a continuous time discrete state Markov chain on the vertex set $V$ (see Section \ref{sec_3}). 
We refer to subsections \ref{notation} and  \ref{eq_assp} for more details on notations, graphs, Hamiltonians and the noise. 
The HJE \eqref{HJeq1} arises from the following mean field control problem on graphs: 
\begin{align*}
\mathcal U(t,\xi)=\inf_{(\sigma,\bar\upsilon)}\Big\{\mathcal U_0(\sigma_0)+\int_0^t(\mathcal L(\sigma_s,\bar\upsilon_s)-\mathcal F(\sigma_s))\mathrm ds:\sigma_t=\xi\Big\},
\end{align*} 
where $\mathcal L(\xi,m)=\sup_{p\in\mathbb S^{d\times d}}\{(p,m)_{\xi}-\mathcal H(\xi,p)\}$ is the Legendre transform of the Hamiltonian $\mathcal H$ in \eqref{HJeq1} (see Assumption \ref{ass_H}), 
and the infimum is performed over admissible pairs $(\sigma,\bar\upsilon)\in W^{1,1}([0,T];\mathcal P(G))\times L^1([0,T];\mathbb S^{d\times d})$ such that control $\bar\upsilon$ satisfies the continuity equation 
$$
\dot{\sigma}+\mathrm{div}_{\sigma}(\bar\upsilon+\nabla_G\log\sigma)=0. 
$$

The HJE \eqref{HJeq1} on Wasserstein space on graphs exhibits distinctive features compared to the continuum counterparts in Euclidean space. 
 First, the interplay between graph geometry and probability simplex constraints introduces structural complexities absent in the classical setting. For instance, a commonly used Hamiltonian function in the Euclidean case may encounter ``blind spots''  on Wasserstein space on graphs when the energy degenerates. In particular, non-zero inputs may yield no effect due to that $\|p\|^2_{\xi}=0$ can hold for $p\neq 0$. 
This complicates the monotonicity analysis of Hamiltonians near the boundary. 
  Second, the graph individual noise, formulated via discrete geometric derivatives and logarithmic geometry,
 alters the structure of the equation. It makes the solution of HJE exclusively defined on the interior of the set of probability measures. Unlike conventional second-order diffusive terms in the classical setting, this graph individual noise acts as a first-order interaction that is deeply intertwined with the discrete graph structure. 
Furthermore, these intrinsic characteristics bring numerical challenges that are more difficult than 
the classical Euclidean case. In the  Euclidean setting, the boundary condition can be naturally defined and used to determine the behavior in the interior domain. While in the graph setting, 
the absence of predefined boundary behavior requires numerical schemes to approximate solutions only based on interior dynamics. 
Therefore, it is desirable to study  
numerical methods that reconcile discrete graph geometry with probability space constraints to ensure stability and accuracy (see Section \ref{sec_3} for details). 

To this end, we develop a discretization framework by transforming the probability constraints into the equivalent Euclidean subspace constraints, to study numerical schemes 
of the HJE on the Wasserstein space on graphs. A prerequisite is establishing the connection between Wasserstein derivatives on graphs and partial derivatives in Euclidean spaces, which is helpful for understanding the relation of difference matrices in different geometries. 
To address the lack of predefined boundary behavior, we use extrapolation-type techniques to simulate the numerical solution on the boundary of density space. Then we present the skew-symmetric difference matrices 
to approximate the Wasserstein gradient of viscosity solution. These matrices are constructed using finite difference methods that are adapted to both the discrete
geometry of graphs and the differential structure of Wasserstein spaces. This allows us to propose two classes of fully discrete numerical schemes on the discretized probability simplex that capture the geometric features of probability spaces. One involves a monotone scheme, which is stable and consistent under a usual CFL condition on the meshsize and time stepsize. The other is an implicit scheme, which is unconditionally stable and free from stepsize restrictions for longtime simulations. 
 Surprisingly, we find that for both schemes it is necessary to treat the noise term and the Hamiltonian jointly when constructing the discrete Hamiltonians on graphs, 
 which is different from the classical setting \cite{Crandall_boun, Soug_boun, boundary1} due to the structural role of the graph individual noise. 
  
Furthermore, we utilize the coordinate transformation to analyze the error of finite difference matrices that approximate the Wasserstein gradient of the solution, as the smoothness property of the Wasserstein gradient on graphs remains unestablished \cite{MCC}. 
With a doubling of variables techniques, we prove that the proposed schemes converge uniformly to the original initial value problem \eqref{HJeq1} in the interior of probability simplex $\mathcal P^{\circ}(G)$. To quantify the approximation error of proposed schemes, we need a more delicate tool to deal with the logarithmic singularity of the graph individual noise, since the viscosity solution is only defined on the interior. In particular, we introduce a new auxiliary barrier function to prevent the grid points of the numerical solution from reaching the boundary. 
Then we propose a boundary assumption on a compact subset $\mathcal P_{\epsilon}(G)$ of the probability simplex such that the monotonicity of the Hamiltonian near the boundary is preserved by the proposed schemes. 
By establishing viscosity inequalities on this compact subset, we obtain the error estimate of a one-half order in $L^{\infty}_{\mathrm{loc}}$. 
Compared to Euclidean counterparts (see e.g. \cite{Crandall_boun, Soug_boun, boundary1}), we find that the noise term is crucial in establishing the monotonicity properties of Hamiltonians near boundaries of $\mathcal P_{\epsilon}(G)$. 
Finally,  several numerical experiments are presented to illustrate the convergence order, the effect of graph individual noise and Hamiltonians on the numerical dynamics.

We would like to mention the connection between the HJE \eqref{HJeq1}, mean field control and potential mean field games, where the value function describes the evolution of population of agents. 
In mean field control, all agents act cooperatively, and the problem evolves into an optimal control formulation on the Wasserstein space. In contrast, mean field games adopt a competitive framework, leading to a fixed-point problem where no individual agent has an incentive
to alter their strategy given the strategies of all others. 
In practice, populations in many models, such as Stag--Hunt game and evolutionary game theory \cite{HJSK}, are often settled on discrete spatial domains like a simple finite graph. 
The graph can be connected to a reversible Markov
chain, 
whose edge is weighted by the transition probability and invariant measure of the process  \cite{GLL}. By introducing the graph structure, the dynamics of discrete-state population distributions can be characterized by HJE \eqref{HJeq1}.

This paper is organized as follows. In Section \ref{sec_3}, we give an intuitive comparison between the HJEs on  Euclidean space and Wassertein space on graphs.  In Section \ref{sec_2}, we present some preliminaries about notations and the well-posedness result for the HJE on the Wasserstein space on graphs. 
In Section \ref{sec_4.1}, we introduce two numerical schemes for the considered HJE. In Section 
\ref{sec_4},  we give 
our main results on the monotonicity, stability, and convergence of numerical solutions. Section \ref{sec_5} is devoted to the convergence analysis of the proposed schemes.  
In Section \ref{sec_6} and Appendix \ref{sec_app2}, we present several numerical experiments to verify our theoretical results.

\section{Intuitive comparison of HJEs on Euclidean space and the Wasserstein space on graphs} \label{sec_3}

In this section, we present three distinctions between HJEs on Euclidean spaces and on the Wasserstein space on graphs. 

 Recall that for classical viscosity solutions, well-posedness can be ensured by explicitly imposing boundary conditions such as Dirichlet, Neumann, or periodic conditions (see  e.g. \cite{boun_sys, boun_lions, boun_06, boun_08, periodic, boundary1}). In contrast, HJE \eqref{HJeq1} on the Wasserstein space on graphs is intrinsically defined within the interior of the probability simplex. Indeed, imposing artificial boundary conditions, such as Dirichlet boundary conditions, on \eqref{HJeq1} can cause solution layers, implying that the solution may not be Lipschitz, as shown in Figure \ref{pic0} (A). Hence, the solution’s behavior at the boundary should be inferred from interior dynamics. 
This indicates us to employ extra interpolation techniques at the boundary to attain a stable and consistent approximation (see Figure \ref{pic0} (B) plotted by monotone method \eqref{explicit1} with linear extrapolation \eqref{boundary}). 

Another distinction comes from the technical treatment of the boundary in the numerical analysis: if the equation holds at interior points in the viscosity sense, it can automatically extend to boundary points. This typically requires monotonicity condition on the Hamiltonian near the boundary (see, e.g.  \cite{Crandall_boun,Soug_boun}). 
To see the difference, we consider a commonly used example in the Euclidean space case:  for $x\in\mathbb R^d,$ \begin{align*}
\mathbb R^d\ni p\mapsto H(x,p)\text{ is radially symmetric with a unique minimum at }p=0;
\end{align*} 
see e.g. \cite{boundary1}.  However, for the Wasserstein space on graphs, 
a key observation is that for fixed $\xi\in\mathcal P(G),$
\begin{align*}
\|p\|^2_{\xi}=\frac12\sum_{(i,j)\in E}p_{i,j}^2g_{i,j}(\xi)=0\text{ may hold for }p\neq 0, 
\end{align*} 
due to the graph structure and the metric tensor of the probability simplex. This complicates the analysis of behaviors of viscosity solutions near the boundary. To deal with this issue, we analyze properties of the graph individual noise and propose a boundary assumption (see Assumption \ref{ass_H2}). Roughly speaking, it requires that as long as $``(p+\nabla_G\log\xi)\cdot \mathbf n<0"$ for the inward normal vector
$\mathbf n$, then $p\to \mathcal H(\xi,p)-\mathcal O_{\xi}(p)$ is nondecreasing in the direction of $-\mathbf n$.

\begin{figure}[H]
\centering

   \subfloat[Graph individual noise, Dirichlet]{
\begin{minipage}[t]{0.35\linewidth}
\centering
\includegraphics[height=4.4cm,width=5.8cm]{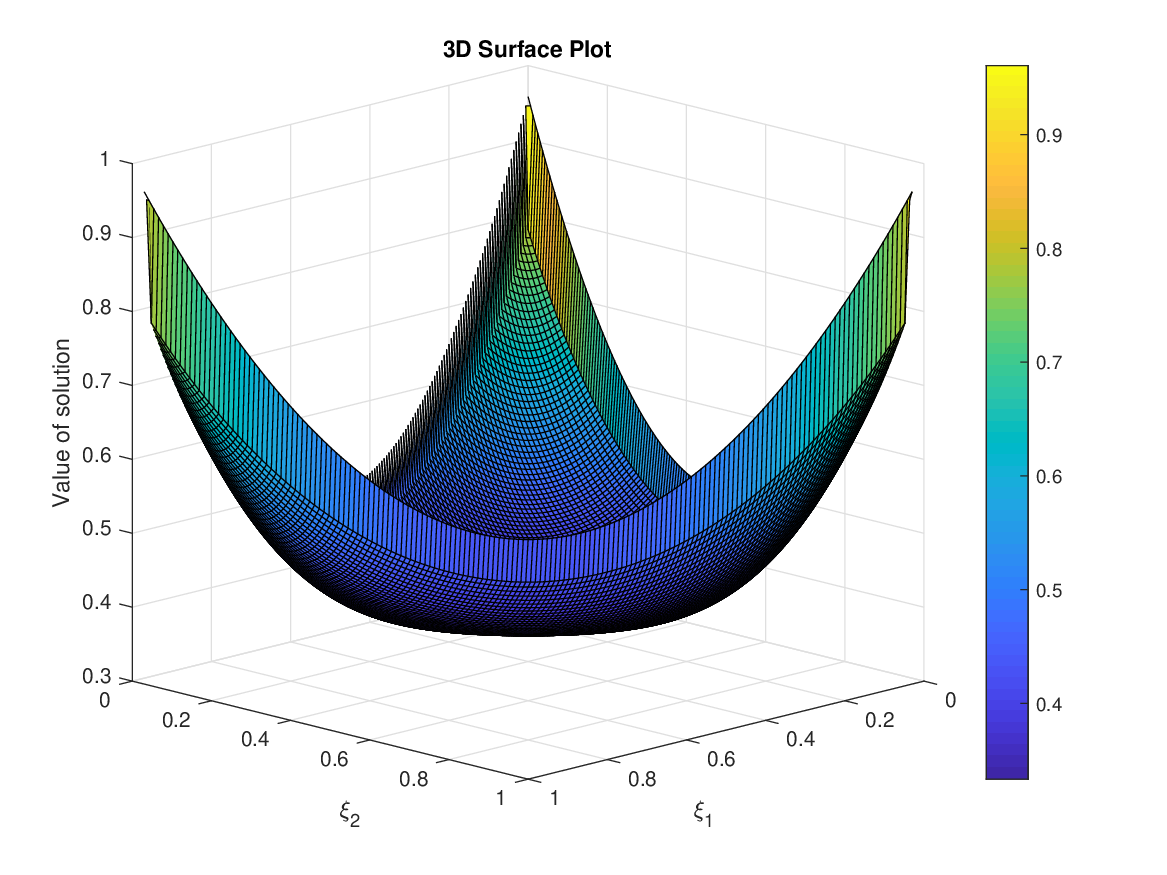}
\end{minipage}
}
\subfloat[Graph individual noise, Extrapolation]{
\begin{minipage}[t]{0.35\linewidth}
\centering
\includegraphics[height=4.4cm,width=5.8cm]{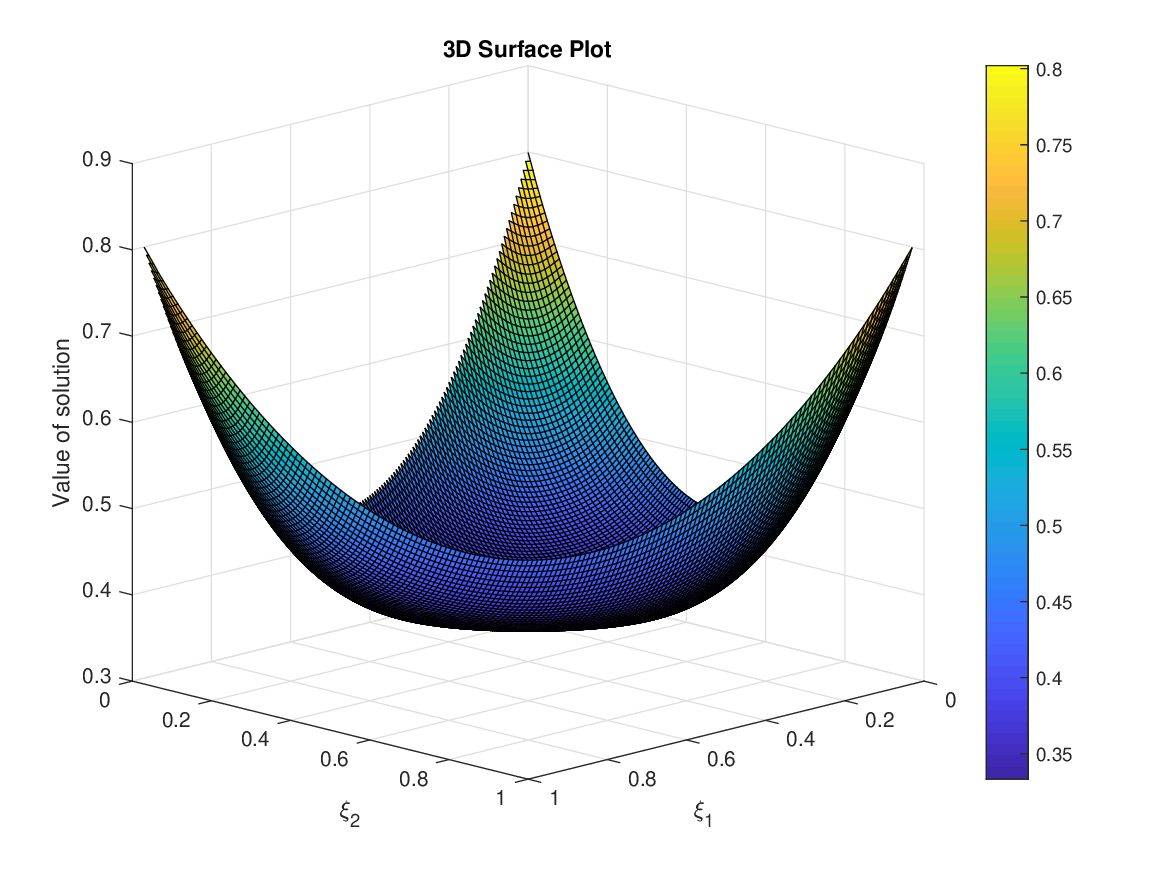}
\end{minipage}
}  

\centering
\caption{Numerical solution for $\mathcal H(\xi,p)=(\sum_{i=1}^3\xi_i^{-1})^{-2}\|p\|_{\xi}^2$, $\frac{\tau}{h}=0.1,h=0.1\times 2^{-4},\epsilon=0.01,$ initial value $\mathcal U_0(\xi)=\|\xi\|^2_{l^2}$.}\label{pic0}
\end{figure}

There are also differences caused by the individual noise operator. 
 We will illustrate this by connecting the operators to the corresponding 
PDE. We take the metric tensor $g$ as  Example \ref{ex_g} (\romannumeral2). It is known from \cite[Section 3.4]{MCC} that the solution of the linear equation 
\begin{align*}
\partial _t\mathcal U=\Delta_{\mathrm{ind}}\mathcal U\;\text{ on }(0,T)\times \mathcal P^{\circ}(G),\quad \mathcal U(0,\cdot)=\mathcal U_0
\end{align*} 
is given by 
$$\mathcal U(t,\xi)=\mathcal U_0(e^{tA}\xi).$$
 Here,  the matrix $A$ is defined as \begin{align*}
A_{i,j}=\omega_{i,j}\text{ if }j\in N(i); \quad A_{i,j}=0\text{ if }j\notin N(i),\;j\neq i;\quad A_{i,j}=-\sum_{k\in N(i)}\omega_{i,k}\text{ if }j=i,
\end{align*}
where $N(i):=\{j\in V:\omega_{i,j}>0\}.$ The matrix $A$ is a $Q$-matrix, and  $e^{tA}$ is a transition matrix. 
According to \cite[Section 2.5]{Markov}, there exists a probability space $(\Omega,\mathcal F,\mathbb P)$ such that for any $\xi\in\mathcal P(G),$ we can find a Markov chain $S:[0,T]\times \Omega\to V$ such that $\mathbb P(S(0,\cdot)=i)=\xi_i$ and 
\begin{align*}
\mathbb P(S(t+h,\cdot)=i|S(t,\cdot)=j)=(e^{h A})_{j,i},\quad t,h\ge 0,\; i,j\in V.
\end{align*}
Set $\sigma_i(t)=\mathbb P(S(t,\cdot)=i)$ for $i\in V.$ It follows from \cite[Section 3.4]{MCC} that 
\begin{align*}
\frac{\mathrm d}{\mathrm dt}\sigma_i(t)=\sum_{j\neq i}A_{j,i}(\sigma_j(t)-\sigma_i(t)),
\end{align*}
which is equivalent to $\frac{\mathrm d}{\mathrm dt}\sigma (t)=\mathrm{div}_{\sigma(t)}(\nabla_G(\log\sigma(t))) $ with $\sigma(t):=(\sigma_1(t),\ldots,\sigma_d(t))\in\mathcal P(G)$ (see Section \ref{notation} for definitions of operators $\mathrm{div}_{\sigma},\nabla_G$).
By the definition of matrix $A$,  the unique solution is given by $\sigma(t)=e^{tA}\xi.$ Since $e^{tA}$ is a transition matrix, there is a function $C_t>0$ such that $(e^{tA}\mu)_i=\sum_{j=1}^n(e^{tA})_{i,j}\xi_j\ge C_t\epsilon,$ if $\xi_i\ge \epsilon$ for all $i\in V.$ Therefore, if the initial value $\xi\in\mathcal P^{\circ}(G),$ we have $\sigma(t)\in\mathcal P^{\circ}(G).$ Hence,  we obtain 
$$\tau_G:=\inf\{t>0:\sigma(t)\in\partial \mathcal P(G),\sigma(0)\in\mathcal P^{\circ}(G)\}=+\infty.$$

In the classical setting, it is well-known that the heat equation $\partial_tu=\frac12\Delta u$ can be linked to Brownian motion $\{B_t\}_{t\ge 0}$ via the Feymann--Kac formula $u(t,x)=\mathbb E\big[u_0(B_t)\mathbf 1_{\{t<\tau_D\}}|B_0=x\big]=e^{-\frac{t\Delta}{2} }u_0(x).
$ Here $D$ is a bounded domain, $\tau_D$ is the first exit time of $B_t$, and $u_0$ is the initial value. The classical isoperimetric
inequality for the exit time of Brownian motion has the form of (see e.g. \cite{exit}) $$\lim_{t\to\infty}\frac 1t\log\mathbb P(\tau_D>t)=-\frac{\lambda_1(D)}{2},$$ where $\lambda_1(D)$ is the first Dirichlet eigenvalue for the Laplacian in $D$. This implies that for any small $\gamma>0,$ there is a constant $T_0>0,$ such that when $t>T_0$, 
$
\mathbb P(\tau_D\leq t)\ge 1-e^{-t(\frac{\lambda_1(D)}{2}-\gamma)}.
$ Hence when $t$ is large, there is a positive probability 
for the heat particles undergoing Brownian motion to reach the boundary of $D$ over a finite time.

\section{preliminaries}\label{sec_2}
In this section, we present the preliminaries for the Wasserstein space of probability measures on a finite graph. Then we present the assumptions and useful properties of \eqref{HJeq1} with a graph individual noise term on the Wasserstein space on graphs.

\subsection{Notations}\label{notation}
 Consider an undirected connected graph $G=(V,E,\omega),$ with no self-loops or multiple edges, where $V=\{1,\ldots,d\}$ with $d\in\mathbb N_+$ is the set of vertices and $E\subset V\times V$ is the set of edges. The weight $\omega=(\omega_{i,j})_{1\le i,j\le d}$ is a $d\times d$ symmetric matrix with nonnegative entries such that $\omega_{i,j}>0$ if $(i,j)\in E.$ 
We denote by $\mathcal P(G)$ the probability simplex 
\begin{align*}
\mathcal P(G)=\Big\{\xi=(\xi_1,\ldots,\xi_d)\in[0,1]^d:\sum_{i=1}^d\xi_i=1\Big\}.
\end{align*}
For fixed $\epsilon\in(0,\frac1d),$
let $\mathcal P_{\epsilon}(G):=\mathcal P(G)\cap [\epsilon,1)^d$. By $\partial A$ and $A^{\circ}$ we denote the boundary and interior of a Borel  set $A$. 
 Then $\mathcal P^{\circ}(G)$ is the interior of $\mathcal P(G),$ $\mathcal P^{\circ}_{\epsilon}(G)$ is the interior of $\mathcal P_{\epsilon}(G)$ for $\epsilon\in(0,\frac1d),$ and  we have $\partial \mathcal P(G)=\mathcal P(G)\backslash \mathcal P^{\circ}(G)$ and $\partial\mathcal P_{\epsilon}(G)=\mathcal P_{\epsilon}(G) \backslash\mathcal P^{\circ}_{\epsilon}(G).$ 

Denote $\mathbb S^{d\times d}$ the set of $d\times d$ skew-symmetric matrices. We introduce a symmetric function $g:[0,\infty)^2 \to[0,\infty)$, i.e., $g(t,r)=g(r,t)$ for $t,r\in[0,\infty)$. 
 For $\xi\in\mathcal P(G),$ we say that $\upsilon,\tilde\upsilon\in\mathbb S^{d\times d}$ are $\xi$-equivalent if $(\upsilon_{i,j}-\tilde\upsilon_{i,j})g_{i,j}(\xi)=0$ for all $(i,j)\in E,$ where $g_{i,j}(\xi):=g(\xi_i,\xi_j).$ This equivalence relation can induce a quotient space on $\mathbb S^{d\times d}$ denoted by $\mathbb H_{\xi}.$ Under certain conditions specified later (see Assumption \ref{assumption_g}), we use $g$ to define a metric tensor on $\mathcal P(G)$ and endow $\mathbb H_{\xi}$ with the inner product and discrete norm 
\begin{align*}
(\upsilon,\tilde\upsilon)_{\xi}:=\frac12\sum_{(i,j)\in E}\upsilon_{i,j}\tilde\upsilon_{i,j}g_{i,j}(\xi),\quad \|\upsilon\|_{\xi}:=\sqrt{(\upsilon,\upsilon)_{\xi}},\;\upsilon,\tilde\upsilon\in\mathbb S^{d\times d}.
\end{align*}
Here the coefficient $\frac12$ accounts for the fact that whenever $(i,j)\in E$ then $(j,i)\in E.$

For the mapping $\phi=(\phi_1,\ldots,\phi_d):V\to\mathbb R^d,$  define its graph gradient as $\nabla_G\phi:=\sqrt{\omega_{i,j}}(\phi_i-\phi_j)_{(i,j)\in E}.$ The adjoint of $\nabla_G$ for the $(\cdot,\cdot)_{\xi}$ inner product is divergence operator $-\mathrm{div}_{\xi}:\mathbb H_{\xi}\to\mathbb R^d$ given by
\begin{align*}
\mathrm{div}_{\xi}(\upsilon):=\Big(\sum_{j=1}^d\sqrt{\omega_{i,j}}\upsilon_{j,i}g_{i,j}(\xi)\Big)_{i=1}^d,\;\upsilon\in\mathbb S^{d\times d}
\end{align*}
such that the integration by parts formula holds, i.e., $(\nabla_G\phi,\upsilon)_{\xi}=-(\phi,\mathrm{div}_{\xi}(\upsilon))$. Here $(\upsilon,\tilde\upsilon):=\frac12 \sum_{(i,j)\in E}\upsilon_{i,j}\tilde\upsilon_{i,j}$ for $\upsilon,\tilde\upsilon\in\mathbb S^{d\times d}$. We use $\|\cdot\|_{\infty}$ to denote the supremum norm of a matrix, and $\|\cdot\|_{l^2}$ to denote the usual Frobenius norm of a matrix.

For $\rho^0,\rho^1\in\mathcal P(G),$ define the $L^2$-Monge--Kantorovich metric 
\begin{align}\label{MK}
\mathcal W(\rho^0,\rho^1):=\Big(\inf_{(\sigma,\upsilon)}\Big\{\int_0^1(\upsilon,\upsilon)_{\sigma}\mathrm dt:\dot{\sigma}+\mathrm {div}_{\sigma}(\upsilon)=0,\;\sigma(0)=\rho^0,\,\sigma(1)=\rho^1\Big\}\Big)^{\frac12},
\end{align}
where the infimum is taken over the set of pairs $(\sigma,\upsilon)$ such that $\sigma\in H^1(0,1;\mathcal P(G))$ and $\upsilon :[0,1]\to\mathbb S^{d\times d}$ is measurable. The probability simplex $\mathcal P(G)$ endowed with the metric $\mathcal W$ is the \textit{Wasserstein space on graphs}, which is denoted by $(\mathcal P(G),\mathcal W)$; see e.g. \cite{Liwuchen,MCC,Maas} for more introduction about this space.  
Throughout this paper, we let the symmetric function $g$ satisfy the following assumption. 
\begin{assumption}\label{assumption_g}
\begin{itemize}
\item[(g-\romannumeral1)] $g$ is continuous on $[0,\infty)^2$ and is smooth on $(0,\infty)^2$;
\item[(g-\romannumeral2)] $t\wedge r\leq g(t,r)\leq t\vee r$ for any $t,r\in[0,\infty);$
\item[(g-\romannumeral3)] $g(\lambda t,\lambda r)=\lambda g(t,r)$ for any $\lambda,t,r\in[0,\infty);$
\item[(g-\romannumeral4)] $g$ is concave;   
\item[(g-\romannumeral5)] $\int_0^1\frac{1}{\sqrt{g(r,1-r)}}\mathrm dr<\infty.$
\end{itemize}
\end{assumption}
 From \cite[Proposition 3.7]{Liwuchen}, we know that (g-\romannumeral5) ensures that $\mathcal W(\rho^0,\rho^1)<\infty$ for any $\rho^0,\rho^1\in\mathcal P(G).$ 
Examples of symmetric function $g$ satisfying  (g-\romannumeral1)--(g-\romannumeral5) are given as follows. 
\begin{example}\label{ex_g}
(\romannumeral1) The average probability weight \cite{Chow1}: $g_1(t,r)=\frac{t+r}{2};$ 

(\romannumeral2) The logarithmic probability weight \cite{Chow2}: 
 $$g_2(t,r)=\frac{t-r}{\log t-\log r},\text{ if }t\neq r;\;g_2(t,r)=0,\text{ if }t=0\text{ or }r=0;\;g_2(t,r)=t,\text{ if }t=r;
$$ (\romannumeral3) The harmonic probability weight \cite{Maas}: 
$$g_3(t,r)=0,\text{ if }t=0\text{ or }r=0;\;g_3(t,r)=\frac{2}{\frac1t+\frac1r},\text{ otherwise}.$$
One can generate more examples by taking convex combinations of $g_l,l=1,2,3.$
\end{example}

With the metric tensor $g$,  one can define the  \textit{$\mathcal W$-differentiability} and the \textit{Wasserstein gradient} $\nabla_{\mathcal W}$ on the Wasserstein space $(\mathcal P(G),\mathcal W)$  \cite[Definition 3.9]{MCC}. For $f:\mathcal P(G)\to\mathbb R$ being $\mathcal W$-differentiable at $\xi\in\mathcal P(G),$ we use the notation $\nabla_{\mathcal W}f(\xi)$ to denote the Wasserstein gradient of $f$ at $\xi\in\mathcal P(G),$ which is a skew-symmetric matrix. For any $(i,j)\in E$ such that $1\leq i<j\leq d,$ we define $e_{i,j}\in\mathbb R^d$ whose $i$th entry is $1$, $j$th entry is $-1$ and other entries are zero. For a function $f:\mathcal P(G)\to\mathbb R,$  define the limit when it exists: 
\begin{align*}
\nabla^{e_{i,j}}f(\xi):=\lim_{t\to0}\frac{f(\xi+te_{i,j})-f(\xi)}{t},\quad \xi\in\mathcal P^{\circ}(G).
\end{align*} Denote by $\frac{\delta f}{\delta \xi}$ the Fr\'echet derivative of $f$. If the Fr\'echet derivative of $f$ exists and is continuous at $\xi\in\mathcal P^{\circ}(G)$, then we have $\nabla_{\mathcal W}f(\xi)=\nabla_{G}\frac{\delta f}{\delta \xi}(\xi),$ and the entries of $\nabla_{\mathcal W}f(\xi)$ are: $\sqrt{\omega_{i,j}}\nabla^{e_{i,j}}f(\xi),1\leq i<j\leq d$, $-\sqrt{\omega_{i,j}}\nabla^{e_{j,i}}f(\xi)$ for $1\leq j<i\leq d.$ 

Throughout this paper, we use $C$ to denote a generic positive constant which may take different values at different appearances. 
Sometimes we write $C(a,b)$ or $C_{a,b}$ to emphasize the dependence on
the parameters $a,b.$ 
 \subsection{Hamilton--Jacobi equation on $(\mathcal P(G),\mathcal W)$}\label{eq_assp} Consider the HJE on the Wasserstein space on  graph:
\begin{align}\label{HJeq}
\partial_tu(t,\xi)+\mathcal H(\xi,\nabla_{\mathcal W}u(t,\xi))+\mathcal F(\xi)=\mathcal O_{\xi}(\nabla_{\mathcal W}u(t,\xi)),\quad u(0,\xi)=\mathcal U_0(\xi),
\end{align}
where $t\in(0,T)$ with some $T>0,$ $\xi\in\mathcal P^{\circ}(G),$ the initial value $\mathcal U_0$ and the potential function $\mathcal F$ are  continuous on $(\mathcal P(G),\mathcal W)$, and $\mathcal O_{\xi}(p)=-( p,\nabla_G\log \xi)_{\xi},\,p\in\mathbb S^{d\times d}$ is the graph individual noise term. The integration by parts formula implies that $\mathcal O_{\xi}(\nabla_{\mathcal W}u):=(\mathrm {div}_{\xi}(\nabla_{\mathcal W}u),\log \xi)=\Delta_{\mathrm{ind}}u(\xi),$ where $\Delta_{\mathrm{ind}}$ is called the graph individual noise operator. 
We note that different from the continuum setting on the Wasserstein space, 
this graph individual noise operator is essentially a first-order differential operator, formulated as a specific combination of Wasserstein derivatives. We now introduce the assumptions on 
the Hamiltonian function $\mathcal H$ and the graph individual noise. \begin{assumption}\label{ass_H}
Fix a constant $\kappa>1$ and assume that there exist a constant $t_{*}>1$ and nonnegative functions $\gamma,\bar{\gamma}\in\mathcal C([0,\infty))$ such that for any $\xi,\eta\in\mathcal P^{\circ}(G)$ and $p\in\mathbb S^{d\times d},$ the following hold: 
\begin{itemize} 
\item[(H-\romannumeral1)] $\mathcal H\in\mathcal C(\mathcal P^{\circ}(G)\times \mathbb S^{d\times d})$ and $\mathcal H(\xi,\cdot)$ is convex. 
\item[(H-\romannumeral2)] $\lim_{t\to1^+}\bar{\gamma}(t)=1,\;\gamma(t)>1$ for any $t\in(1,t_*)$, and we have 
\begin{align*}
t\gamma(t)\mathcal H(\xi,p)\leq \mathcal H(\xi,tp)\leq \bar{\gamma}(t)\mathcal H(\xi,p),\quad \forall \,t>0. 
\end{align*}
\item[(H-\romannumeral3)] For every $\epsilon\in(0,1),$ there exists $\theta_{\epsilon}>0$ such that 
\begin{align*}
\theta_{\epsilon}\|p\|^{\kappa}_{\xi}\leq \mathcal H(\xi,p),\quad\forall\,\xi\in\mathcal P_{\epsilon}(G).
\end{align*} 
\item[(H-\romannumeral4)] Let $\mathcal H(\xi,0)=0$. For fixed $\epsilon\in(0,1)$, there exist a modulus $\mathfrak m_{\epsilon}$
and a constant $C_{\epsilon}>0$ such that 
\begin{align*}
\mathcal H(\xi,p)-\mathcal H(\eta,p)\ge -\mathfrak m_{\epsilon}(\|\xi-\eta\|_{l^2})\|p\|^{\kappa}_{\xi}-C_{\epsilon}\big|\|p\|_{\xi}-\|p\|_{\eta}\big|\big(\|p\|^{\kappa-1}_{\xi}+\|p\|^{\kappa-1}_{\eta}\big),\quad \forall\,\xi,\eta\in\mathcal P_{\epsilon}(G).
\end{align*}
\item[(H-\romannumeral5)] Denote $\mathcal I(\xi):=\sum_{i=1}^d\frac{1}{\xi_i}.$ There exists $C_H>0$ such that  
\begin{align*}
|\mathcal H(\xi,p)|\leq C_H\|p\|^{\kappa}_{\xi}\mathcal I^{-\kappa}(\xi),\quad \forall\,(\xi,p)\in\mathcal P^{\circ}(G)\times \mathbb S^{d\times d}.
\end{align*}
\item[($\mathcal O$-\romannumeral1)] There exists $C_{\mathcal O}\ge0$ such that $|\mathcal O_{\xi}(p)|\leq C_{\mathcal O}\|p\|_{l^2},\quad \forall\,(\xi,p)\in\mathcal P^{\circ}(G)\times\mathbb S^{d\times d}.$  
\end{itemize}\end{assumption}

Recall that a modulus, also known as modulus of continuity, is a function $\mathfrak m:[0,\infty)\to[0,\infty)$ that is increasing such that $\mathfrak m(r)\downarrow  0$ when $r\downarrow  0$. The modulus $\mathfrak m_{\epsilon}$ in (H-\romannumeral4) 
plays an important role in the comparison principle for the HJE; see e.g. \cite{book2011}. 
We remark that for the metric tensor $g$ being functions $ g_2,g_3$ given in Example  \ref{ex_g},  \begin{align}\label{extension}(0,1)^2\ni (t,s)\mapsto \frac{\log t-\log s}{t-s}g(t,s) \text{\; has a continuous extension to }[0,1]^2,
\end{align} which ensures that condition ($\mathcal O$-\romannumeral1) is satisfied for the graph individual noise. 
Below we give an example of the Hamiltonian that satisfies  (H-\romannumeral1)--(H-\romannumeral5). 
\begin{example}\label{exam1} 
Take $\mathcal H(\xi,p):=\mathfrak a(\xi)\|p\|^{\kappa}_{\xi},\;\xi\in\mathcal P^{\circ}(G),p\in\mathbb S^{n\times n}$, where coefficient $\mathfrak a=\mathcal I^{-\kappa}$ and $\kappa>1$. In this case, $\theta_{\epsilon}=C\epsilon^{\kappa}d^{-\kappa},\, C_{\epsilon}=C\kappa d^{-\kappa}, \,m_{\epsilon}(r)=Cr$ for $r\ge 0$ due to the $l^2$-Lipschitz continuity of $\mathcal I^{-\kappa}$, where $C>0$ is a generic constant which may be different from each appearance. 
We refer to \cite[Example 5.1]{MCC} for more details. 
\end{example}

Under Assumption \ref{ass_H}, \eqref{HJeq} is well-posed in the sense of viscosity solution. In fact, the viscosity solution is shown to be bounded and Lipschitz continuous.

\begin{proposition}\label{prop_exact}
\cite[Proposition 6.4, Section 6]{MCC} Let Assumptions \ref{assumption_g}--\ref{ass_H} hold. In addition assume that $\mathcal U_0$ is $l^2$-Lipschitz continuous and $\mathcal F\in\mathcal C(\mathcal P(G)).$ Then there exists a unique  bounded continuous viscosity solution $u$ 
of \eqref{HJeq} on $[0,T)\times \mathcal P^{\circ}(G)$ satisfying that 
\begin{itemize}
\item[(\romannumeral1)] There exists $L_1>0$ such that $|u(t,\xi)-u(r,\xi)|\leq L_1|t-r|$ for all $\xi\in\mathcal P^{\circ}(G), t,r\in[0,T)$;
\item[(\romannumeral2)] For every $\epsilon\in(0,1)$, there exists $L_2:=L_2(\epsilon)>0$ such that $|u(t,\xi)-u(t,\eta)|\leq L_2\|\xi-\eta\|_{l^2}$ for all $t\in[0,T),\xi,\eta\in\mathcal P_{\epsilon}(G).$
\end{itemize}
\end{proposition}

\section{Numerical schemes of HJE on Wasserstein space on graphs}\label{sec_4.1}
 
 In this section, 
we first introduce the spatial meshspace on the Wasserstein space on graphs and present difference matrices  to approximate the Wasserstein gradient of the viscosity solution. Then we introduce the discrete Hamiltonian and two types of numerical schemes on Wasserstein space on graphs. 
\subsection{Meshspace and difference matrix} 
For $\xi\in\mathcal P(G),$ we introduce a transformation as follows:
\begin{align*}
s^0=0,\;s^1=\xi_1,\;s^2=\xi_1+\xi_2,\;\ldots,\;s^d=\sum_{i=1}^d\xi_i=1.
\end{align*} 
This yields the alternative presentation of $\mathcal P(G)$ by nondecreasing $d$-tuples between $0$ and $1$:
\begin{align*}
\widetilde{\mathcal P}(G):=\{s=(s^1,\ldots,s^{d-1})\in\mathbb R^{d-1}:0=s^0\leq s^1\leq\cdots\leq s^{d-1}\leq s^d=1\}.
\end{align*} 
Similarly, for $\epsilon\in(0,\frac1d)$, define $$\widetilde{\mathcal P}_{\epsilon}(G):=\{(s^1,\ldots,s^{d-1})\in\mathbb R^{d-1}:0\leq s^1\leq \cdots\leq s^{d-1}\leq 1-d\epsilon\}.$$ We claim that there exists a bijective mapping $\Pi$ such that $\widetilde {\mathcal P}_{\epsilon}(G)=\Pi \mathcal P_{\epsilon}(G)$ for $\epsilon\in[0,1)$, where we use the convention $\mathcal P_0(G)=\mathcal P(G).$ 
In fact, note that for each $s\in\widetilde{\mathcal P}_{\epsilon}(G),$ it corresponds to a unique point $\Pi^{-1}(s):=(s^1+\epsilon,s^2-s^1+\epsilon,\ldots,s^{d-1}-s^{d-2}+\epsilon,1-s^{d-1}-(d-1)\epsilon)\in\mathcal P_{\epsilon}(G).$ We also use the notation $\xi^s=\Pi^{-1}(s)$ for simplicity.  Conversely, for each given point $\xi\in\mathcal P_{\epsilon}(G)$, we have $\Pi(\xi):=(\xi_1-\epsilon,\xi_2+\xi_1-2\epsilon,\xi_3+\xi_2+\xi_1-3\epsilon,\ldots,\sum_{i=1}^{d-1}\xi_i-(d-1)\epsilon)\in\widetilde{\mathcal P}_{\epsilon}(G)$, where we have used $\xi_i\ge \epsilon,$ $\sum_{i=1}^{d-1}\xi_i-(d-1)\epsilon=1-\xi_d-(d-1)\epsilon\leq 1-d\epsilon$ and $\sum_{i=1}^{j-1}\xi_i-(j-1)\epsilon-(\sum_{i=1}^{j}\xi_i-j\epsilon)=\epsilon-\xi_j\leq  0.$

With the help of the above transformation $\Pi$, we have the following connection between Wasserstein derivatives on graphs and partial derivatives in Euclidean spaces. It  
 plays an important role in the convergence analysis for numerical schemes on  Wasserstein spaces on graphs.

\begin{lemma}\label{lemma1}
For $f\in\mathcal C^1(\mathcal P^{\circ}(G))$, denoting $\tilde f(x)=f(\xi)$ with $x=\Pi(\xi)$, we have that $\nabla_{\mathcal W}f(\xi)$ can be expressed explicitly by the partial derivatives of $\tilde f$ at point $\Pi(\xi)$, and the entries are 
\begin{align*}
\sqrt{\omega_{j,k}}\nabla^{e_{j,k}}f(\xi)=\sqrt{\omega_{j,k}}(\partial_{x_j}+\cdots+\partial_{x_{k-1}})\tilde f(\Pi(\xi)),\quad 1\leq j<k\leq d,\;\;\xi\in\mathcal P^{\circ}(G).
\end{align*}
\end{lemma}
 For completeness, we put its proof in Appendix \ref{app_vari}. Now we are in a position introduce the discrete probability space and the difference matrix corresponding to the Wasserstein gradient.

Given the space meshsize $h\in(0,\frac13),$  without loss of generality we assume that $N_{h,\epsilon}:=\frac{1-d\epsilon}{h}$ for fixed small $\epsilon>0$ is an integer. 
Let $\vec i=(i_1,\ldots,i_{d-1})\in\{0,1,\ldots,N_{h,\epsilon}\}^{d-1}$ denote the multi-index. 
We introduce the  meshspace with a uniform meshsize as follows:  
\begin{align*}
\widetilde{\mathcal P}^{h}_{\epsilon}(G):=&\Big\{s_{\vec i}=(s^1_{i_1},s^2_{i_2},\ldots,s^{d-1}_{i_{d-1}})\in\mathbb R^{d-1}:\\
&\;\;\;0\leq s^1_{i_1}\leq \cdots\leq s^{d-1}_{i_{d-1}}\leq 1-d\epsilon,\;s^j_{i_j}=i_jh,\,i_j=0,1,\ldots,N_{h,\epsilon},\,j=1,\ldots,d-1\Big\}. 
\end{align*} 
We call $\widetilde{\mathcal P}_{\epsilon}^h(G)$ the discrete probability space under the Euclidean coordinates. Denote the discrete probability space under the Wasserstein coordinates by 
\begin{align*}
\mathcal P_{\epsilon}^h(G):=\big\{\xi\in\mathcal P_{\epsilon}(G): \Pi(\xi)\in\widetilde{\mathcal P}^h_{\epsilon}(G)\big\}.
\end{align*}  
Define the sets of indices 
\begin{align}\label{index_set}&\mathcal N_{\epsilon}:=\{\vec i=(i_1,\ldots,i_{d-1}):\;0\leq i_1\leq \cdots\leq i_{d-1}\leq N_{h,\epsilon},\,s_{\vec i}\in\widetilde{\mathcal P}^h_{\epsilon}(G)\},\notag\\
&\partial\mathcal N_{\epsilon}=\{\vec i\in\mathcal N_{\epsilon}:\xi^{s_{\vec i}}\in {\mathcal P}^h_{\epsilon}(G)\cap\partial{\mathcal P}_{\epsilon}(G),\quad \mathcal N^{\circ}_{\epsilon}:=\mathcal N_{\epsilon}\backslash\partial\mathcal N_{\epsilon}. 
\end{align} 
Namely, we use the index $\vec i\in\mathcal N_{\epsilon}$ to represent the $\vec i$th 
grid point in the meshspace $\widetilde{\mathcal P}^h_{\epsilon}(G),$ and $\vec i\in\partial\mathcal N_{\epsilon}$ (resp. $\vec i\in\mathcal N^{\circ}_{\epsilon}$) to denote the $\vec i$th grid point in the boundary (resp. interior).  It is clear that $|\mathcal N_{\epsilon}|=|\widetilde{\mathcal P}^h_{\epsilon}(G)|=|\mathcal P^h_{\epsilon}(G)|,$ where $|\cdot|$ denotes the cardinal number of a set.  
 When $\epsilon=0,$ we can define $N_{h}=N_{h,0}$ to simplify the notations, similarly, 
 we  can define the meshspaces and index sets $\widetilde{\mathcal P}^h(G),{\mathcal P}^h(G),\mathcal N,\partial\mathcal N,\mathcal N^{\circ}$ in a consistent way.

To distinguish the coordinates in meshspaces, we use $\tilde u(t,\cdot)$ with initial value $\widetilde{\mathcal U}_0:=\mathcal U_0\circ\Pi^{-1}$ to denote the exact viscosity solution on $\widetilde{\mathcal P}_{\epsilon}(G)$, which satisfies that $\tilde u(t,x)=u(t,\Pi^{-1}(x))=u(t,\xi)$ with $x=\Pi(\xi)$. The numerical approximations of $u$ (resp. $\tilde u$) are denoted by  $U$  (resp. $\widetilde U$), respectively. Namely, we have $\widetilde U(t_n,s_{\vec i})=\widetilde U(t_n,\Pi(\xi^{s_{\vec i}}))=U(t_n,\xi^{s_{\vec i}}),\,s_{\vec i}\in\widetilde {\mathcal P}^h_{\epsilon}(G),t_n=n\tau$ with $\tau$ being the time stepsize introduced later. For simplicity we also use notations $\widetilde U^n_{\vec i},U^n_{\vec i}$ to denote them when there is no confusion.

 To  approximate the Wasserstein gradient, we first define two $d\times d$ skew-symmetric difference matrices 
 on the meshspace ${\mathcal P}^h_{\epsilon}(G)$: for $\vec i\in \mathcal N^{\circ}_{\epsilon},$
\begin{align}\label{differ_matrix}
&[D^{\pm} U_{\vec i}]=(0,b_{1,2},b_{1,3},\ldots,b_{1,d};-b_{1,2},0,b_{2,3},\ldots,b_{2,d};\ldots;-b_{1,d},\ldots,-b_{d-1,d},0)\notag\\
&\text{ with entries }
 b_{j,k}=\frac{\sqrt{\omega_{j,k}}}{h}D^{\pm}_{e_{j,k}} U_{\vec i} \text{ for } 1\leq j<k\leq d.
 \end{align}  
 Here,  the forward and backward differences are respectively defined as 
\begin{align}\label{D_notation}
D^{+}_{e_{j,k}} U_{\vec i}:=U(\xi^{s_{\vec i}}+he_{j,k})-U(\xi^{s_{\vec i}}),\quad D^{-}_{e_{j,k}} U_{\vec i}:=U(\xi^{s_{\vec i}})-U(\xi^{s_{\vec i}}-he_{j,k})
\end{align}
when $\xi^{s_{\vec i}}\pm he_{j,k}\notin \mathcal P^h_{\epsilon}(G)\cap \partial \mathcal P_{\epsilon}(G).$ Otherwise, we use the extra interpolation method to define the values of the numerical solution at the boundary. To be specific, for $y\in  \mathcal P^h_{\epsilon}(G)\cap \partial \mathcal P_{\epsilon}(G),$ we use the constant extrapolation \begin{align}\label{boundary1}
U(y):=U(\xi^{s_{\vec i}}) \text{  with }
\xi^{s_{\vec i}}=\arg\min\limits_{\substack{z\in \mathcal P^h_{\epsilon}(G)\backslash\partial \mathcal P_{\epsilon}(G)\text{ or}\\ z\in \partial \mathcal P_{\epsilon}(G), \,U(z)\text{ is known}}}\|y-z\|_{l^2};\end{align}
or use the linear extrapolation  
\begin{align}\label{boundary}
U(y):=2U(\xi^{s_{\vec i}})-U(2\xi^{s_{\vec i}}-y), 
\end{align} 
where $\xi^{s_{\vec i}},2\xi^{s_{\vec i}}-y$ are two nearest interior points or boundary points where $U(\xi^{s_{\vec i}}),U(2\xi^{s_{\vec i}}-y)$ have been defined. We note that the arbitrary selection of defined two points does not affect the convergence analysis. The definition of points near the corner may involve the use of values at the boundary.

To introduce the difference matrices in coordinates of $\widetilde{\mathcal P}^h_{\epsilon}(G),$ for fixed $j,k$ with $1\leq j<k\leq d,$ we introduce the multi-index $\vec{m}_{j,k}:=(m_1,\ldots,m_{d-1})$  such that  $m_l=1$ for $l\in\{j,\ldots,k-1\}$ and $m_l=0$ for $l\notin\{j,\ldots,k-1\}$. 
By virtue of the relation $\Pi^{-1}(s_{\vec i})\pm he_{j,k}=\Pi^{-1}(s_{\vec i}\pm h\vec{m}_{j,k}),$ one has that  
\begin{align}\label{def_DtildeU}
D^+_{e_{j,k}} U_{\vec i}&=\widetilde U(s_{\vec i}+h\vec m_{j,k})-\widetilde U(s_{\vec i})=:D^+_{\vec m_{j,k}}\widetilde U_{\vec i},\notag\\
D^-_{e_{j,k}} U_{\vec i}&=\widetilde U(s_{\vec i})-\widetilde U(s_{\vec i}-h\vec m_{j,k} )=:D^-_{\vec m_{j,k} }\widetilde U_{\vec i}.
\end{align}
Then the equivalent difference matrices of \eqref{differ_matrix} in the coordinate of $\widetilde{\mathcal P}^h_{\epsilon}(G)$ are denoted by $[D^+\widetilde U_{\vec i}], [D^-\widetilde U_{\vec i}],$ whose entries are given by $b_{j,k}=\frac{\sqrt{\omega_{j,k}}}{h}D^{\pm}_{\vec m_{j,k}}\widetilde U_{\vec i}$ for $1\leq j<k\leq d.$

\subsection{Numerical schemes}
In this subsection, we propose two types of numerical schemes for the HJE \eqref{HJeq} on the Wasserstein space on graphs. 

\textbf{Monotone scheme on  Wasserstein space on graphs:} 
We propose spatial-temporal fully discrete monotone methods.  
 Denote the skew-symmetric matrices \begin{align*}&P:=(0,p_{1,2},p_{1,3},\ldots,p_{1,d};-p_{1,2},0,p_{2,3},\ldots,p_{2,d};\ldots;-p_{1,d},\ldots,-p_{d-1,d},0),\\
 &Q:=(0,q_{1,2},q_{1,3},\ldots,q_{1,d};-q_{1,2},0,q_{2,3},\ldots,q_{2,d};\ldots;-q_{1,d},\ldots,-q_{d-1,d},0).
 \end{align*}  
We rewrite the Hamiltonian as $\mathcal H(\xi,P)=\mathcal H(\xi,p_{1,2},p_{1,3},\ldots,p_{d-1,d}):\mathcal P^{\circ}(G)\times \mathbb R^{(d^2-d)/2}\to\mathbb R.$
Define the discrete Hamiltonian on Wasserstein space on graphs $$\mathcal G(\xi,P,Q)=\mathcal G(\xi,p_{1,2},q_{1,2};p_{1,3},q_{1,3};\ldots;p_{d-1,d},q_{d-1,d}):\mathcal P^{\circ}(G)\times \mathbb R^{d^2-d}\to\mathbb R.$$ We impose the following  assumption on $\mathcal G$. 

\begin{assumption}\label{ass_G}
\begin{itemize}
\item[(\romannumeral1)] (Monotonicity) $\mathcal G$ is non-increasing with respect to $p_{k,l}$ and non-decreasing with respect to $q_{k,l},1\leq k<l\leq d.$
\item[(\romannumeral2)] (Consistency) $\mathcal G(\xi,P,P)=\mathcal H(\xi,P)-\mathcal O_{\xi}(P),\;\xi\in\mathcal P^{\circ}(G).$
\item[(\romannumeral3)] (Local Lipschitz property) $\mathcal G$ is locally Lipschitz with respect to  arguments $p_{k,l},q_{k,l},1\leq k<l\leq d$, i.e., for  each fixed $R>0,$ when $\|P\|_{l^2}\vee\|Q\|_{l^2}\vee\|\bar P\|_{l^2}\vee\|\bar Q\|_{l^2}\leq R$, we have 
\begin{align}\label{locally}
|\mathcal G(\xi,P,Q)-\mathcal G(\xi,\bar P,\bar Q)|\leq C_{R}(\|P-\bar P\|_{l^2}+\|Q-\bar Q\|_{l^2}),\quad \xi\in \mathcal P^{\circ}(G)
\end{align}
for some $C_{R}>0$ independent of $\xi$. 
\end{itemize}
\end{assumption} 
Inspired by the Osher--Sethian Hamiltonian in \cite{Osher} and Lax--Friedrichs Hamiltonian (see e.g. \cite{CrandallLions}) in the classical setting, we present two examples of  $\mathcal G$ satisfying Assumption \ref{ass_G} with the Hamiltonian $\mathcal H$ given by Example \ref{exam1}.

\begin{example}\label{exam2}
 Define the function $\mathcal G$ by 
 \begin{align*}
 \mathcal G(\xi,P,Q)=\mathcal G_{\mathcal H}(\xi,P,Q)+\mathcal G_{\mathcal O}(\xi,P,Q),
 \end{align*}
 where 
 \begin{align*}
 &\quad \,\mathcal G_{\mathcal H}(\xi,p_{1,2},q_{1,2};p_{1,3},q_{1,3};\ldots;p_{d-1,d},q_{d-1,d})\\
 &:=\mathcal I^{-\kappa}(\xi)\Big(g_{1,2}\big((p_{1,2}^-)^2+(q_{1,2}^+)^2\big)+g_{1,3}\big((p_{1,3}^-)^2+(q_{1,3}^+)^2\big)+\cdots+g_{d-1,d}\big((p_{d-1,d}^-)^2+(q_{d-1,d}^+)^2\big)\Big)^{\frac{\kappa}{2}},
 \end{align*} 
  and 
 \begin{align*}
 &\quad\,\mathcal G_{\mathcal O}(\xi,p_{1,2},q_{1,2};p_{1,3},q_{1,3};\ldots;p_{d-1,d},q_{d-1,d})\\
 & :=\sqrt{\omega_{1,2}}g_{1,2}(\log\xi_1-\log\xi_2)\big(p_{1,2}\mathbf 1_{\{\xi_1\leq \xi_2\}}+q_{1,2}\mathbf 1_{\{\xi_1> \xi_2\}}\big)\\
 &\quad +\sqrt{\omega_{1,3}}g_{1,3}(\log\xi_1-\log\xi_3)\big(p_{1,3}\mathbf 1_{\{\xi_1\leq \xi_3\}}+q_{1,3}\mathbf 1_{\{\xi_1> \xi_3\}}\big)\\
 &\quad +\cdots+\sqrt{\omega_{d-1,d}}g_{d-1,d}(\log\xi_{d-1}-\log\xi_d)\big(p_{d-1,d}\mathbf 1_{\{\xi_{d-1}\leq \xi_d\}}+q_{d-1,d}\mathbf 1_{\{\xi_{d-1}> \xi_d\}}\big).
  \end{align*} 
  Here, $p_{k,l}^-:=\max\{0,-p_{k,l}\}$, $q_{k,l}^+:=\max\{0,q_{k,l}\}$ for $1\leq k<l\leq d,$ and we call that $\kappa>1$ is given in Example \ref{exam1}, $g$ is the metric tensor and $\omega=(\omega_{i,j})_{i,j}$ is the weight of the graph.  We refer to Appendix \ref{app_vari} for the verification that Assumption \ref{ass_G} is satisfied. 
 \end{example}

 \begin{example}\label{exam3}
 Define the function $\mathcal G_{\mathcal H}$ by 
\begin{align*}
 \mathcal G_{\mathcal H}(\xi,P,Q):=\mathcal H(\xi,\frac{P+Q}{2})-\sum_{1\leq i<j\leq d}\gamma_{i,j}(p_{i,j}-q_{i,j}),
 \end{align*} 
where $\gamma_{i,j}>0$ is such that $\gamma_{i,j}\leq \frac12\sup_{\{\|P\|_{\infty}\leq R\}}|\partial _{p_{ij}} \mathcal H(P)|$ for some fixed $R>0.$ 
Then $\mathcal G:=\mathcal G_{\mathcal H}+\mathcal G_{\mathcal O}$ satisfies Assumption \ref{ass_G}, where $\mathcal G_{\mathcal O}$ is given in Example \ref{exam2}. 
\end{example}  

Let $\tau\in(0,1)$ be the uniform temporal stepsize, and we assume without loss of generality that $N_T:=\frac{T}{\tau}$ is an integer. Set $t_n:=n\tau,n=0,\ldots,N_T.$ 
Introduce the following explicit scheme of \eqref{HJeq},
\begin{align}\label{explicit1}
U^{n+1}_{\vec i}=U_{\vec i}^n-\tau\mathcal F(\xi^{s_{\vec i}})-\tau\mathcal G(\xi^{s_{\vec i}},[D^+U^n_{\vec i}],[D^-U^n_{\vec i}]),  \;\;0\leq n\leq N_{T}-1,\quad  U^0_{\vec i}=\mathcal U_0(\xi^{s_{\vec i}}),
\end{align}
where $\vec i\in\mathcal N^{\circ}$ 
and $\xi^{s_{\vec i}}=\Pi^{-1}(s_{\vec i})\in\mathcal P^h(G)$ with 
 $s_{\vec i}
\in\widetilde{\mathcal P}^h(G).$ For $\vec i\in\partial\mathcal N,$ we use the constant extrapolation \eqref{boundary1} or the linear extrapolation \eqref{boundary} to define the value of $U^{n+1}_{\vec i}.$  
The equivalent formulation of \eqref{explicit1} in  $\widetilde{\mathcal P}^h(G)$ reads as 
\begin{align}\label{explicit}
\widetilde U^{n+1}_{\vec i}=\widetilde U^{n}_{\vec i}-\tau\widetilde{\mathcal F}(s_{\vec i})-\tau \widetilde{\mathcal G}\big({s_{\vec i}}, [D^+\widetilde U^n_{\vec i}],[D^-\widetilde U^n_{\vec i}]\big),\;\;0\leq n\leq N_{T}-1,\quad \widetilde U^0_{\vec i}=\widetilde {\mathcal U}_0(s_{\vec i}),  
\end{align} 
where  $[D^{\pm}\widetilde U^n_{\vec i}]$ are difference matrices whose entrices are given in \eqref{def_DtildeU}, and $\widetilde{\mathcal G}(x,\cdot)=\mathcal G(\xi,\cdot)$ with $x=\Pi(\xi)$. For simplicity, we also denote $ U^{n+1}=\vec{\mathcal G}( U^n)$ and $\widetilde U^{n+1}=\vec{\mathcal G}( \widetilde U^n)$ without distinction when the arguments are not emphasized.

 Note that $U^n(\xi^{s_{\vec i}}+rhe_{j,k})=\widetilde U^n(\Pi^{-1}(\xi^{s_{\vec i}}+rhe_{j,k}))=\widetilde U^n(s_{\vec i}+rh\vec m_{j,k})$.  The monotonicity of the numerical method \eqref{explicit} (see Appendix \ref{sec_app} for the definition of monotone method) can be similarly defined by considering the corresponding numerical solution  
coordinated in $\widetilde{\mathcal P}^h(G)$. From this point onward, we will not distinguish between the monotonicity of \eqref{explicit1} and \eqref{explicit}.

\textbf{Fully-implicit scheme on Wasserstein space on graphs:} 
It is known that explicit methods often require the CFL-type condition (see \eqref{cond_CLF}) on the meshsize and time stepsize to ensure stability and accuracy. To remove this type of condition to allow for large time steps in a long-time simulation \cite{Luo}, we introduce an fully-implicit method for the HJE \eqref{HJeq} as follows  
 \begin{align} \label{implicit1}
 U^{n+1}_{\vec i}= U^{n}_{\vec i}-\tau{\mathcal F}(\xi^{s_{\vec i}})-\tau\mathcal G\big(\xi^{s_{\vec i}}, [D^+ U^{n+1}_{\vec i}],[D^- U^{n+1}_{\vec i}]\big),\; 0\leq n\leq N_{T}-1,\quad  U^0_{\vec i}= \mathcal U_0(\xi^{s_{\vec i}}),
\end{align} 
where $\vec i\in\mathcal N^{\circ}$ and $s_{\vec i}\in\widetilde{\mathcal P}^h(G).$  
For $\vec i\in\partial\mathcal N,$ we use the constant extrapolation \eqref{boundary1} or the linear extrapolation \eqref{boundary} to define the value of $U^{n+1}_{\vec i}.$ 
The equivalent formulation in  $\widetilde{\mathcal P}^h(G)$ is \begin{align}\label{implicit}
\widetilde U^{n+1}_{\vec i}=\widetilde U^{n}_{\vec i}-\tau\widetilde{\mathcal F}(s_{\vec i})-\tau\widetilde{\mathcal G}\big({s_{\vec i}}, [D^+\widetilde U^{n+1}_{\vec i}],[D^-\widetilde U^{n+1}_{\vec i}]\big),\; 0\leq n\leq N_{T}-1,\quad \widetilde U^0_{\vec i}=\widetilde{\mathcal U}_0 (s_{\vec i}).
\end{align}

\section{Main Results}\label{sec_4}

In this section, we first show that the proposed explicit method \eqref{explicit1} is monotone and stable. 
The uniform boundedness of the numerical solution is also proved. Furthermore, we present the convergence and error estimates of the solutions of numerical methods \eqref{explicit1} and \eqref{implicit1}. 

\subsection{Monotonicity and stability}
\label{apriori} Define the set 
\begin{align}\label{setM}
\mathcal M_{0}:=&\Big\{\widetilde U\in L^{\infty}(\widetilde{\mathcal P}^{h}(G)\backslash\partial\widetilde{\mathcal P}(G)): |D^{\pm}_{\vec m_{j,k}}\widetilde U_{\vec i}|\leq Rh,\;1\leq j<k\leq d,\vec i\in\mathcal N^{\circ}\Big\},
\end{align}
 where $D^{\pm}_{\vec m_{j,k}}\widetilde U_{\vec i}$ is given by \eqref{def_DtildeU}.
Due to the $l^2$-Lipschitz continuity of $\mathcal U_0,$ we derive that 
$|D^+_{e_{j,k}}\mathcal U_0|\leq h\mathrm{Lip}(\mathcal U_0)\|e_{j,k}\|_{l^2}\leq Rh$ when $R>1+\sqrt{2}\mathrm{Lip}(\mathcal U_0).$  For the one-step mapping $\vec{\mathcal G}:\mathbb R^{|\mathcal N^{\circ}|}\to \mathbb R^{|\mathcal N^{\circ}|}, \widetilde U^n\mapsto\widetilde U^{n+1},$ from \eqref{explicit} we have that the $\vec i$th entry  is 
 $(\vec{\mathcal G}(\widetilde U^n))_{\vec i}:=\widetilde U^{n}_{\vec i}-\tau\widetilde{\mathcal F}(s_{\vec i})-\tau\widetilde{\mathcal G}\big({s_{\vec i}}, [D^+\widetilde U^{n}_{\vec i}],[D^-\widetilde U^{n}_{\vec i}]\big),\vec i\in\mathcal N^{\circ}$. For $\widetilde U,\widetilde V\in L^{\infty}(\widetilde{\mathcal P}^{h}(G)\backslash\partial\widetilde{\mathcal P}(G) ),$ we use $\widetilde U\leq \widetilde V$ to mean $\widetilde U_{\vec i}\leq \widetilde V_{\vec i}$ for all $\vec i\in\mathcal N^{\circ}.$

 The following lemma gives the monotonicity,  contractivity, and stability properties of the numerical solution. Compared with the Euclidean setting \cite{CrandallLions}, we only need the monotonicity at the interior of probability simplex for  \eqref{explicit1}. We present the detailed proof to Appendix \ref{app_vari} to keep the article self-contained.   

\begin{lemma}\label{lem1} Let Assumptions \ref{assumption_g}--\ref{ass_G} hold, $\mathcal F\in\mathcal C(\mathcal P(G)),$  $\mathcal U_0$ be $l^2$-Lipschitz continuous,  and
let  
 $\frac{\tau}{h}=:\lambda_0$ satisfy  
\begin{align}\label{cond_CLF}
\lambda_0\leq (2C_{dR}(d^2-d)\|\sqrt{\omega}\|_{\infty})^{-1}.  
\end{align} 
with some $R>0$. Then the following properties hold. 
\begin{itemize} 
\item[(\romannumeral1)] Scheme \eqref{explicit1} is monotone on $[-R,R]$; 
\item[(\romannumeral2)] If $\widetilde U,\widetilde V\in\mathcal M_{0}$, then $\|\vec{\mathcal G}(\widetilde U)-\vec{\mathcal G}(\widetilde V)\|_{\infty}\leq \|\widetilde U-\widetilde V\|_{\infty}$;
\item[(\romannumeral3)] $\vec{\mathcal G}(\mathcal M_{0})\subset \mathcal M_{0}$;
\item[(\romannumeral4)] For $n,j\ge 0$ and $\widetilde U\in\mathcal M_{0}$, we have $\|\vec{\mathcal G}^{(n+j)}(\widetilde U)-\vec{\mathcal G}^{(n)}(\widetilde U)\|_{\infty}\leq Kj\tau$, where $K:=\sup_{\{\xi\in\mathcal P^{\circ}(G),\|P\|_{\infty}\leq R,\|Q\|_{\infty}\leq R\}}|\mathcal G(\xi,P,Q)|+\|\mathcal F\|_{L^{\infty}(\mathcal P(G))}.$
\end{itemize}
\end{lemma}
 Based on Lemma \ref{lem1}, we have the following uniform boundedness of the solution of \eqref{explicit1}.
\begin{proposition}\label{mono1}
Let Assumptions \ref{assumption_g}--\ref{ass_G} hold, $\mathcal F\in\mathcal C(\mathcal P(G)),$ $\mathcal U_0$ be $l^2$-Lipschitz continuous, and  
\eqref{cond_CLF} hold with some $R>1+\sqrt 2\mathrm{Lip}(\mathcal U_0)$. Then there exists some positive constant $C:=C(\|\mathcal F\|_{L^{\infty}(\mathcal P(G))},\|\mathcal U_0\|_{L^{\infty}(\mathcal P(G))})$ such that the solution of method \eqref{explicit1} satisfies $$\sup\limits_{\vec i\in\mathcal N^{\circ},\xi^{s_{\vec i}}\in\mathcal P^h(G)}|U^n_{\vec i}|\leq C(t_n+1).$$ 
\end{proposition}
\begin{proof}
It suffices to prove the boundedness of the solution $\widetilde U^n.$ 
By Lemma \ref{lem1} (\romannumeral2), we have 
\begin{align*}
\|\vec{\mathcal G}^{(n)}(\widetilde {\mathcal U}_0)\|_{\infty}&\leq \|\vec{\mathcal G}^{(n)}(\widetilde  {\mathcal U}_0)-\vec{\mathcal G}^{(n)}(0)\|_{\infty}+\|\vec{\mathcal G}^{(n)}(0)\|_{\infty}\leq \|\widetilde {\mathcal U}_0\|_{\infty}+
\|\vec{\mathcal G}^{(n)}(0)\|_{\infty}.
\end{align*}
Noting that the potential function $\mathcal F$ is independent of $\widetilde U,$ it follows from \eqref{explicit}, the consistency of $\mathcal G$, and $\mathcal H(\xi,0)=0$ (see Assumption \ref{ass_H} (H-\romannumeral4)), we have  
$
\vec{\mathcal G}^{(2)}(0)=\vec{\mathcal G}(-\tau \mathcal F(\cdot))=-2\tau\mathcal F(\cdot)\leq 2\tau\|\mathcal F\|_{L^{\infty}(\mathcal P(G))}.
$
By iterations, we derive $
\|\vec{\mathcal G}^{(n)}(0)\|_{\infty} \leq n\tau \|\mathcal F\|_{L^{\infty}(\mathcal P(G))}.$ 
This finishes the proof. 
\end{proof}
\begin{remark}\label{remark1}
By implementing the extrapolation-type  boundary 
condition (see \eqref{boundary1} or  \eqref{boundary}), one can also derive the uniform boundedness for the numerical solution on the whole probability simplex: 
\begin{align}\label{whole_bound}\sup_{\vec i\in\mathcal N,\xi^{s_{\vec i}}\in\mathcal P^h(G)}|U^n_{\vec i}|\leq C(t_n+1).\end{align} 
Furthermore, when  replacing  $\mathcal P^h(G)$ ($\widetilde{\mathcal P}^h(G)$)  by $\mathcal P^h_{\epsilon}(G)$ ($\widetilde{\mathcal P}^h_{\epsilon}(G)$) for some fixed $\epsilon\in(0,\frac1d),$  we can analogously define the set 
$\mathcal M_{\epsilon}$ in a manner consistent with \eqref{setM}. 
Similarly, 
Lemma \ref{lem1} and Proposition \ref{mono1} also holds when replacing $\mathcal P^h(G)$ ($\widetilde{\mathcal P}^h(G)$)  by $\mathcal P^h_{\epsilon}(G)$ ($\widetilde{\mathcal P}^h_{\epsilon}(G)$). Moreover, similar to \eqref{whole_bound}, we also have 
\begin{align}\label{whole_bound2}
\sup_{\vec i\in\mathcal N_{\epsilon},\xi^{s_{\vec i}}\in\mathcal P^h_{\epsilon}(G)}|U^n_{\vec i}|\leq C_{\epsilon}(t_n+1). 
\end{align} In the compact subset case, the constant $C_R$ in the local Lipschitz condition in Assumption \ref{ass_G} (\romannumeral3) can be replaced by $C_{R,\epsilon}$. This allows for more robustness on the choice of metric tensor $g$, as the extension condition \eqref{extension} is not needed.  
We mention that in convergence analysis, the  
boundedness established over $\mathcal P^h(G)$ (i.e., \eqref{whole_bound})  
is employed to prove the convergence of numerical methods, and that over the compact subset $\mathcal P^h_{\epsilon}(G)$ is utilized to derive the error estimates of the methods. 
\end{remark}

\subsection{Convergence} This subsection presents the convergence of the numerical solutions. We first show that the numerical solution converges uniformly in the interior of probability simplex to the original initial value problem.   
Then we can obtain the error estimate of order $\frac12$ for the proposed  schemes, when restricted to the compact subset $\mathcal P_{\epsilon}(G)$ with some fixed $\epsilon\in(0,\frac1d)$.   The proofs are postponed to Section \ref{sec_5}. 
When presenting the error estimate of the numerical schemes on $\mathcal P_{\epsilon}(G)$, we require 
the following monotonicity of the Hamiltonian near the boundary of $\mathcal P_{\epsilon}(G)$, which helps to establish the viscosity inequalities on $\partial \mathcal P_{\epsilon}(G).$ 

\begin{assumption}\label{ass_H2} Fix $\epsilon\in(0,\frac1d).$
There exsits a non-negative function $\mathfrak{b}\in\mathcal C^1(\mathcal P_{\epsilon}(G))\cap \mathcal C(\partial\mathcal P_{\epsilon}(G))$ such that $\mathfrak b(\xi)=0$ if and only if $\xi\in\partial \mathcal P_{\epsilon}(G).$   For any $\xi_0\in\partial\mathcal P_{\epsilon}(G),k>0,$ 
positive function $\mu:\mathcal P^{\circ}_{\epsilon}(G)\to\mathbb R^+,$ 
and $\eta\in\mathcal P_{\epsilon}(G),$ it holds that 
\begin{align*}
&\liminf_{\xi\to\xi_0}\Big(\mathcal H(\xi,k\nabla_G(\xi-\eta)-\mu(\xi)\nabla_{\mathcal W}\mathfrak b(\xi))-\mathcal O_{\xi}(k\nabla_G(\xi-\eta)-\mu(\xi)\nabla_{\mathcal W}\mathfrak b(\xi))\Big)\\
&\ge \mathcal H(\xi_0,k\nabla_G(\xi_0-\eta))-\mathcal O_{\xi_0}(k\nabla_G(\xi_0-\eta)).
\end{align*}
\end{assumption}

\begin{example}\label{ex_boundH}
Consider the Hamiltonian function in Example \ref{exam1} $\mathcal H(\xi,p)=\mathcal I(\xi)^{-2}\|p\|_{\xi}^2$ with $d=3$. 
Define $\mathfrak b(\xi):=(\xi_1-\epsilon)(\xi_2-\epsilon)(\xi_3-\epsilon)$. One can calculate that 
\begin{align*}
& %\nabla_{\mathcal W}\mathfrak b^{-1}(\xi)=-\mathfrak b^{-2}(\xi)
-\nabla_{\mathcal W}\mathfrak b(\xi)
=
%\mathfrak b^{-2}(\xi)
\left(\begin{array}{ccc}
0& \sqrt{\omega_{1,2}}(\xi_3-\epsilon) e_{1,2}^{\top}(\xi-\epsilon)&\sqrt{\omega_{1,3}}(\xi_2-\epsilon) e_{1,3}^{\top}(\xi-\epsilon)\\
-\sqrt{\omega_{1,2}}(\xi_3-\epsilon) e_{1,2}^{\top}(\xi-\epsilon)&0&\sqrt{\omega_{2,3}}(\xi_1-\epsilon) e_{2,3}^{\top}(\xi-\epsilon)\\
-\sqrt{\omega_{1,3}}(\xi_2-\epsilon) e_{1,3}^{\top}(\xi-\epsilon)& -\sqrt{\omega_{2,3}}(\xi_1-\epsilon) e_{2,3}^{\top}(\xi-\epsilon)&0
\end{array}
\right).
\end{align*} Assume that $\omega_{1,2}=\omega_{1,3}=\omega_{2,3}$, and the metric tensor $g(t,\epsilon)$ is non-decreasing for $t\in[\epsilon,1-2\epsilon)$, which can be satisfied by $g_1,g_2,g_3$ given in Example \ref{ex_g}. We refer to Appendix \ref{app_vari} for the verification that Assumption \ref{ass_H2} is satisfied. 
\end{example}

\begin{theorem}\label{thm1}
Let Assumptions \ref{assumption_g}--\ref{ass_G} hold, $h\in(0,\frac13)$, $\mathcal F\in\mathcal C(\mathcal P(G)),$ and  
$\mathcal U_0$ be $l^2$-Lipschitz continuous with the Lipschitz constant $\mathrm{Lip}(\mathcal U_0).$

(\romannumeral1) There exist a small constant $a\in(0,1)$ and a constant $R>1+\max\{\sqrt{2}\mathrm{Lip}(\mathcal U_0),C(d)(L_2(a)+a^{-1})\}$ such that \eqref{cond_CLF} holds.  Moreover,  
$$\lim_{h\to0}\sup_{n\in\{0,\ldots,N_T-1\},\vec i\in\mathcal N^{\circ},\xi^{s_{\vec i}}\in\mathcal P^h(G)}|U^n_{\vec i}-u(t_n,\xi^{s_{\vec i}})|=0.$$ 

(\romannumeral2) Let Assumption \ref{ass_H2} hold with some fixed $\epsilon\in(0,\frac1d).$ We in addition assume that modulus $\mathfrak m_{\epsilon}$ is of linear growth (i.e., $ \mathfrak m_{\epsilon}(r)\leq C({\epsilon})r,\,r\ge 0$), and that potential $\mathcal F$ is $l^2$-Lipschitz continuous on $\mathcal P_{\epsilon}(G)$. Then when \eqref{cond_CLF} is satisfied with $R>1+\max\{\sqrt{2}\mathrm{Lip}(\mathcal U_0),C(d)L_2(\epsilon)\}$, there exists
  some constant $C:=C(\epsilon,T)>0$, \begin{align*}
\sup_{n\in\{0,\ldots,N_{T} \},\vec i\in\mathcal N_{\epsilon},\xi^{s_{\vec i}}\in\mathcal P_{\epsilon}^h(G)}|U^n_{\vec i}-u(t_n,\xi^{s_{\vec i}})|\leq C(\tau^{\frac12}+h^{\frac12}).
\end{align*}
\end{theorem}

The following theorem states that when $\mathcal G$ is the numerical Hamiltonian function of \eqref{HJeq}, the implicit method is  well-defined
and convergent. 
Note that there is no CFL-type condition like \eqref{cond_CLF} for the implicit method \eqref{implicit1}, compared with the monotone method \eqref{explicit1}. The analysis is similar to that of the explicit method and is thus put in Appendix \ref{app_vari}. 
\begin{theorem}\label{thm_implicit}
Let Assumptions \ref{assumption_g}--\ref{ass_G} hold,  $h\in(0,\frac13)$, $\mathcal F\in\mathcal C(\mathcal P(G)),$ and $\mathcal U_0$ be $l^2$-Lipschitz continuous.  Then there exists a unique  numerical solution of the implicit method \eqref{implicit1} satisfying    
$\sup\limits_{n\in\{0,\ldots,N_T-1\},\vec i\in\mathcal N^{\circ},\xi^{s_{\vec i}}\in\mathcal P^h(G)}|U^n_{\vec i}|<\infty$, and 
$$\lim_{\tau,h\to0}\sup_{n\in\{0,\ldots,N_T-1\},\vec i\in\mathcal N^{\circ},\xi\in\mathcal P^h_0(G)}|U^n_{\vec i}-u(t_n,\xi)|=0.$$ 
 Moreover, for each fixed $\epsilon\in(0,\frac1d),$ if we in addition let Assumptiom \ref{ass_H2} hold, modulus $\mathfrak m_{\epsilon}$ be of linear growth, and  potential $\mathcal F$ be $l^2$-Lipschitz continuous on $\mathcal P_{\epsilon}(G)$, then there exists some constant $C:=C(\epsilon,T)>0,$ \begin{align*}
\sup_{n\in\{0,\ldots,N_{T}\},\vec i\in\mathcal N_{\epsilon},\xi^{s_{\vec i}}\in\mathcal P_{\epsilon}^h(G)}|U^n_{\vec i}-u(t_n,\xi^{s_{\vec i}})|\leq C(\tau^{\frac12}+h^{\frac12}). 
\end{align*} 
\end{theorem}

\section{Proofs of Main Results}
\label{sec_5}

This subsection is devoted to the proof of Theorem \ref{thm1}. We first give the proof of the convergence for the monotone method \eqref{explicit1}.
 The proof is based on a doubling of variables argument (see e.g. \cite{CrandallLions}), with carefully chosen penalty terms for the auxiliary function to prevent the maximizer from being attained at the boundary.

\begin{proof}[Proof of Theorem \ref{thm1} (\romannumeral1)] 
For simplicity, we only give the proof for $d=3$ since the proof can be also generalized to the other case $d>3$. Denote $\mathbb Q^{\tau}_{h,0}:=\{(t_n,(j,k,l)h)\in[0,T)\times {\mathcal P}^h(G): (j, k,l)\in\mathcal N^{\circ};n=0,\ldots,N_T-1\}$.
Without loss of generality, we assume that 
$\sigma:=\sup_{(t_n,\xi)\in\mathbb Q^{\tau}_{h,0}}(u(t_n,\xi)-U(t_n,\xi))>0.$ The case that $\inf _{(t_n,\xi)\in\mathbb Q^{\tau}_{h,0}}(u(t_n,\xi)-U(t_n,\xi))<0$ can be proved in a similar way.  Denote $\mathcal J(\xi,\eta):=\frac12\|\xi-\eta\|^2_{l^2}.$

\textbf{Penalization procedure.} For $\lambda\in(\frac12,1)$ and $\varrho,a,\delta\in(0,1),$ define the auxiliary function 
\begin{align}\label{def_Phi}
\Phi_{\varrho,a,\delta}(t,\xi,s,\eta)&:=\lambda u(t,\xi)-U(s,\eta)-\frac{\sigma}{4T}(t+s)-\frac{\varrho}{T-t}-\frac{\varrho}{T-s}\notag\\
&\quad -\frac{\mathcal J(\xi,\eta)}{\delta^2}-\frac{(t-s)^2}{2\delta^2}-a(\mathcal I(\xi)+\mathcal I(\eta)),
\end{align}
for $t\in[0,T),s\in\{t_0,\ldots,t_{N_T-1}\}$ and $\xi\in\mathcal P^{\circ}(G),\eta\in\mathcal P^h(G)\backslash \partial \mathcal P(G).$ 
Note that $\Phi_{\varrho,a,\delta}$ tends to $-\infty$ when $t,s\to T^-,$ or $\xi,\eta\to\partial\mathcal P(G),$
and it has upper bound on the set $[0,T)\times\mathcal P^{\circ}(G)\times \mathbb Q^{\tau}_{h,0}.$ Hence for every $\varrho,a,$ and $\delta,$ there exists a maximizer on $[0,T)\times\mathcal P^{\circ}(G)\times\mathbb Q^{\tau}_{h,0}$ for $\Phi_{\varrho,a,\delta}.$ 
We denote a  maximizer of $\Phi_{\varrho,a,\delta}$ by $(\bar t,\bar\xi,\bar s,\bar \eta)$ without displaying the dependence on $\varrho,a,\delta,h,\tau.$ Recalling \eqref{whole_bound}, we let $M>0$ satisfy that $\sup_{[0,T)\times \mathcal P^{\circ}(G)}|u(t,\xi)|\vee \sup_{\{n\in\{0,\ldots,N_T\},\vec i\in\mathcal N,\xi^{s_{\vec i}}\in\mathcal P^h(G)\}}|U^n_{\vec i}|\leq M.$ 
Denote  
\begin{align*}
&M_{\varrho,a,\delta}:=\sup_{(t,\xi)\in[0,T)\times \mathcal P^{\circ}(G),(s,\eta)\in\mathbb Q^{\tau}_{h,0}}\Phi_{\varrho,a,\delta}(t,\xi,s,\eta),\quad M_{\varrho,a}:=\sup_{(s,\eta)\in\mathbb Q^{\tau}_{h,0}}\Phi_{\varrho,a,\delta}(s,\eta,s,\eta),\\
&M_{\varrho}:=\sup_{(s,\eta)\in\mathbb Q^{\tau}_{h,0}}\Big(\lambda u(s,\eta)-U(s,\eta)-\frac{\sigma s}{2T}-\frac{2\varrho}{T-s}\Big),\quad \Phi_0(s,\eta)=\lambda u(s,\eta)-U(s,\eta)-\frac{\sigma s}{2T}.
\end{align*}    
We give two claims:

\textit{Claim $1$:}  There exist small constants $c_1,c_2$ independent of $h,\tau$ such that $\bar\xi,\bar\eta\in\mathcal P^{\circ}_{ac_1}(G)$ and $\bar t,\bar s< T-\varrho c_2.$

\textit{Claim $2$:} It holds that 
  \begin{align}&\lim_{\delta\to0}\frac{(\bar t-\bar s)^2+\|\bar\xi-\bar\eta\|^2_{l^2}}{\delta^2}=0\quad \forall \,\varrho,a\in(0,1),\label{smalldelta}\\
&\lim_{a\to0}\limsup_{\delta\to0}a(\mathcal I(\bar \xi)+\mathcal I(\bar\eta))=0\quad \forall \,\varrho\in(0,1),\label{small_I}\\
&\lim_{\delta\to0}M_{\varrho,a,\delta}=M_{\varrho,a},\quad  \lim\limits_{a\to0}M_{\varrho,a}=M_{\varrho},\quad \lim_{\varrho\to0}M_{\varrho}=\sup_{(s,\eta)\in\mathbb Q^{\tau}_{h,0}}\Phi_0(s,\eta).\label{limit1}
 \end{align}

\textit{Proof of Claim $1$}: On the one hand, we have 
\begin{align}\label{rela_<}
\Phi_{\varrho, a,\delta}(\bar t,\bar\xi,\bar s,\bar\eta)&\leq  2M-\frac{\varrho}{T-\bar t}-\frac{\varrho}{T-\bar s}-a\big(\max_{l=1,\ldots,d}\frac{1}{\bar\xi_l}+\max_{l=1,\ldots,d}\frac{1}{\bar\eta_l}\big)\notag\\
&\leq 2M-\max\Big\{\frac{\varrho}{T-\bar t},\frac{\varrho}{T-\bar s},a\max_{l=1,\ldots,d}\frac{1}{\bar\xi_l},a\max_{l=1,\ldots,d}\frac{1}{\bar\eta_l}\Big\},
\end{align}
where $d=3.$ Let $\lfloor \cdot\rfloor $ and $\lceil \cdot\rceil$ represent the floor and ceiling functions respectively. 
On the other hand, 
taking $t=s=\lfloor \frac{N_T}{2}\rfloor \tau\leq \frac T2,$ and $\xi=\eta$ with $\xi_1=\xi_2=\eta_1=\eta_2=\lceil\frac{N_h}{16}\rceil h,\xi_3=\eta_3=1-2h\lceil\frac{N_h}{16}\rceil,$ 
it follows from $\varrho,a<1$ that 
\begin{align}\label{rela_>}
\Phi_{\varrho,a,\delta}(\bar t,\bar\xi,\bar s,\bar\eta)> -3M-\frac 4T-2\times(32+\frac{1}{1-2h(\frac{N_h}{16}+1)})\ge -3M-\frac 4T-74
\end{align} 
when the spatial meshsize $h<\frac13.$  Inequalities  \eqref{rela_<} and \eqref{rela_>} lead to that $\min\{T-\bar t,T-\bar s\}> \frac{\varrho}{5M+\frac 4T+74}=:\varrho c_2$ and 
$\min\{\min\limits_{l=1,\ldots,d}\bar\xi_l,\min\limits_{l=1,\ldots,d}\bar\eta_l\}> \frac{a}{5M+\frac 4T+74}=:ac_1$ hold  simultaneously. Hence \textit{Claim $1$} is proved.

\textit{Proof of Claim $2$:} Observing that 
\begin{align*}
&\quad M_{\varrho,a,\delta}+\frac{(\bar t-\bar s)^2+2\mathcal J(\bar\xi,\bar\eta)}{4\delta^2}=\Phi_{\varrho,a,\delta}(\bar t,\bar\xi,\bar s,\bar\eta)+\frac{(\bar t-\bar s)^2+2\mathcal J(\bar\xi,\bar\eta)}{4\delta^2}\\
&=\Phi_{\varrho,a,\sqrt{2}\delta}(\bar t,\bar\xi,\bar s,\bar\eta)\leq M_{\varrho,a,\sqrt 2\delta},
\end{align*}
and that $\lim_{\delta\to 0^+}M_{\varrho,a,\delta}$ exists with a finite value due to the monotone convergence theorem, we obtain \eqref{smalldelta}. 
Moreover, using $\Phi_{\varrho,a,\delta}(\bar t,\bar\xi,\bar s,\bar\eta)=M_{\varrho,a,\delta}\ge \Phi_{\varrho,a,\delta}(s,\eta,s,\eta),$ and taking supremum with $(s,\eta)\in\mathbb Q^{\tau}_{h,0}$, we get $\Phi_{\varrho,a,\delta}(\bar t,\bar\xi,\bar s,\bar\eta)\ge M_{\varrho,a}.$ By \eqref{def_Phi}, it holds   
$\Phi_{\varrho,a,\delta}(\bar t,\bar \xi,\bar s,\bar\eta)\leq \lambda u(\bar t,\bar\xi)-U(\bar s,\bar\eta)-\frac{\sigma}{4T}(\bar t+\bar s)-\frac{\varrho}{T-\bar t}-\frac{\varrho}{T-\bar s}-a(\mathcal I(\bar\xi)+\mathcal I(\bar\eta)).$
 In virtue of  \eqref{smalldelta}, letting $\delta\to0$ gives $\lim_{\delta\to0}M_{\varrho,a,\delta}=M_{\varrho,a}.$

To prove \eqref{small_I}, we first note that \begin{align}\label{inequality11}
M_{\varrho,a,\delta}+\frac a2(\mathcal I(\bar\xi)+\mathcal I(\bar\eta))+\frac{(\bar t-\bar s)^2+2\mathcal J(\bar\xi,\bar\eta)}{4\delta^2}\leq M_{\varrho,\frac a2,\sqrt{2}\delta}.
\end{align}
Taking supremum limit with $\delta\to0$, and using \eqref{smalldelta} and the first equality in \eqref{limit1}, we have 
\begin{align*}
M_{\varrho,a}+\limsup_{\delta\to0}\frac{a}{2}(\mathcal I(\bar\xi)+\mathcal I(\bar\eta))\leq M_{\varrho,\frac{a}{2}}.
\end{align*}
Thus \eqref{small_I} holds if $\lim_{a\to0}M_{\varrho,a}$ exists and is finite. Indeed, we can prove that $\lim_{a\to0}M_{\varrho,a}=M_{\varrho}$, i.e., the second equality in \eqref{limit1} holds. On the one hand,   since $\Phi_{\varrho,a,\delta}(s,\eta,s,\eta)\leq M_{\varrho}$, taking supremum with $(s,\eta)\in\mathbb Q^{\tau}_{h,0}$ and taking  $a\to0$ yield 
$\lim\limits_{a\to0}M_{\varrho,a}\leq M_{\varrho}$. On the other hand, notice that $$M_{\varrho,a}\ge \Phi_{\varrho,a,\delta}(s,\eta,s,\eta)=\lambda u(s,\eta)-U(s,\eta)-\frac{\sigma s}{2T}-\frac{2\varrho}{T-s}-2a\mathcal I(\eta),\quad (s,\eta)\in\mathbb Q^{\tau}_{h,0}.$$ Letting $a\to0$ leads to $\lim_{a\to0}M_{\varrho,a}\ge \lambda u(s,\eta)-U(s,\eta)-\frac{\sigma s}{2T}-\frac{2\varrho}{T-s}.$ This finishes the proof of the second equality in \eqref{limit1} 
by taking the supremum with $(s,\eta)\in \mathbb Q^{\tau}_{h,0}$. 
The third equality in \eqref{limit1} can be proved in a similar way and thus is omitted. 
Hence \textit{Claims} $1$--$2$ are proved.

Next, we split the proof of the convergence into three cases: $\bar t>0,\bar s>0;$ $ \bar t\ge 0,\bar s=0;$ and $\bar t=0,\bar s>0.$ 

\textbf{Analysis for 
case  $\bar t>0,\bar s>0.$} 

\textit{Step $1$: Viscosity inequalities.} 
When $\bar s,\bar\eta$ are fixed, the mapping 
\begin{align*}
(t,\xi)\mapsto u(t,\xi)-\frac{\sigma t}{\lambda 4T}-\frac{\varrho}{\lambda(T-t)}-\frac{\mathcal J(\xi,\bar\eta)}{\lambda \delta^2}-\frac{(t-\bar s)^2}{2\lambda\delta^2}-\frac{a}{\lambda}\mathcal I(\xi)
\end{align*}
takes the maximum value at point $(\bar t,\bar \xi).$ By the definition of the viscosity subsolution (see Definition \ref{def_viscosity}), we infer   
\begin{align}\label{ineq_lambda}
\frac{\sigma}{ 4T}+\frac{\varrho}{(T-t)^2}+\frac{t-\bar s}{\delta^2}+\lambda\mathcal H(\bar\xi, \frac{\bar p}{\lambda})+\lambda\mathcal F(\bar\xi)\leq \lambda\mathcal O_{\bar\xi}\big(\frac{\bar p}{\lambda}\big),
\end{align}
where $\bar p:=\frac{\nabla_{\mathcal W}\mathcal J(\cdot,\bar\eta)(\bar\xi)}{\delta^2}+a\nabla_{\mathcal W}\mathcal I(\bar\xi)=:\bar p_1+\bar p_2$.
Utilizing the convexity of $\mathcal H(\bar\xi,\cdot)$ (see Assumption \ref{ass_H} (H-\romannumeral1)),  from $\lambda_1 \bar p_1=\lambda_1 (\bar p-\bar p_2)=\lambda_1 \bar p+(1-\lambda_1)\frac{\lambda_1}{\lambda_1-1}\bar p_2$ with $\lambda_1=\frac{1+\lambda}{2}\in(\lambda,1)\subset (0,1),$ we have 
$
\lambda \mathcal H(\bar\xi,\frac{\bar p}{\lambda})\ge \frac{\lambda}{\lambda_1}\mathcal H(\bar\xi,\frac{\lambda_1 \bar p_1}{\lambda})-\frac{\lambda(1-\lambda_1)}{\lambda_1}\mathcal H(\bar\xi,\frac{\lambda_1 \bar p_2}{(\lambda_1-1)\lambda}). 
$
It follows from \cite[Eq. (5.16)]{MCC} that $\|\nabla_{\mathcal W}\mathcal I(\bar\xi)\|_{\bar\xi}\leq C\sum_{l=1}^d\frac{1}{\bar\xi_l^2}\leq C(\mathcal I(\bar\xi))^2$. Then combining (H-\romannumeral5) we obtain that $
\lambda \mathcal H(\bar\xi,\frac{\bar p}{\lambda})\ge \frac{\lambda}{\lambda_1}\mathcal H(\bar\xi,\frac{\lambda_1 \bar p_1}{\lambda})-C
(a\mathcal I(\bar\xi))^{\kappa}
$
for $C>0$ independent of $a$ and $\delta.$ Inserting this inequality into \eqref{ineq_lambda} and using (H-\romannumeral2),  one deduces 
\begin{align}\label{rela1}
 \frac{\sigma}{4T}+\frac{\varrho}{T^2}+\frac{\bar t-\bar s}{\delta^2}+\gamma(\frac{\lambda_1}{\lambda})\mathcal H(\bar\xi,\bar p_1)+\mathcal F(\bar \xi)-\mathcal O_{\bar\xi}(\bar p)\leq (1-\lambda)\|\mathcal F\|_{\infty}+C(a\mathcal I(\bar\xi))^{\kappa}. 
 \end{align}

In addition,  since $\Phi_{\rho,a,\delta}$ takes maximum at $(\bar t,\bar\xi,\bar s,\bar\eta),$ we have that the mapping 
\begin{align*}
(s,\eta)\mapsto  U(s,\eta)+\frac{\sigma s}{4T}+\frac{\varrho}{T-s}+\frac{\mathcal J(\bar\xi,\eta)}{\delta^2}+\frac{(\bar t-s)^2}{2\delta^2}+a\mathcal I(\eta)
\end{align*}
with $(s,\eta)\in\mathbb Q^{\tau}_{h,0}$ takes the minimum value at $(\bar s,\bar\eta):=(n_0\tau,(j_0,k_0,l_0)h).$ This yields that 
\begin{align}\label{leadsto1}
 U^n_{j,k,l}&\ge  U^{n_0}_{j_0,k_0,l_0}+\frac{\sigma}{4T}(n_0-n)\tau+\frac{\varrho}{T-n_0\tau}-\frac{\varrho}{T-n\tau}+
 \frac{1}{\delta^2}\big(\mathcal J(\bar\xi,(j_0,k_0,l_0)h)-\mathcal J(\bar\xi,(j,k,l)h)\big)\notag
\\
& +\frac{1}{2\delta^2}\big((\bar t-n_0\tau)^2-(\bar t-n\tau)^2\big)+a\big(\mathcal I((j_0,k_0,l_0)h)-\mathcal I((j,k,l)h)\big),\;\; (j,k,l)\in\mathcal N^{\circ}.
\end{align}
Then by virtue of the monotonicity of the one-step mapping $\vec{\mathcal G}$ of \eqref{explicit1}, and \eqref{leadsto1} with $n=n_0-1,$ we obtain   
\begin{align}\label{U_eq1}
& U^{n_0}_{j_0,k_0,l_0}=\big(\vec{\mathcal G}( U^{n_0-1})\big)_{j_0,k_0,l_0}\notag\\
&\ge  U^{n_0}_{j_0,k_0,l_0}+\frac{\sigma\tau}{4T}+\frac{\varrho}{T-n_0\tau}-\frac{\varrho}{T-n_0\tau+\tau}+\frac{1}{2\delta^2}\big((\bar t-n_0\tau)^2-(\bar t-(n_0-1)\tau)^2\big)
-\tau{\mathcal F}(\bar\eta)\notag\\
&  -\tau\mathcal G\Big(\bar\eta,  \frac{\sqrt{\omega_{1,2}}}{h}(-\frac{1}{\delta^2}D^+_{e_{1,2}}\mathcal J(\bar\xi,\cdot)(\bar\eta)-aD^+_{e_{1,2}}\mathcal I(\bar\eta)),  \frac{\sqrt{\omega_{1,2}}}{h}(-\frac{1}{\delta^2}D^-_{e_{1,2}}\mathcal J(\bar\xi,\cdot)(\bar\eta)-aD^-_{e_{1,2}}\mathcal I(\bar\eta));\notag\\ &\qquad\;\; \frac{\sqrt{\omega_{1,3}}}{h} (-\frac{1}{\delta^2}D^+_{e_{1,3}}\mathcal J(\bar\xi,\cdot)(\bar\eta)-aD^+_{e_{1,3}}\mathcal I(\bar\eta)),  \frac{\sqrt{\omega_{1,3}}}{h}(-\frac{1}{\delta^2}D^-_{e_{1,3}}\mathcal J(\bar\xi,\cdot)(\bar\eta)-aD^-_{e_{1,3}}\mathcal I(\bar\eta));\notag\\&\qquad\;\;\frac{\sqrt{\omega_{2,3}}}{h} (-\frac{1}{\delta^2}D^+_{e_{2,3}}\mathcal J(\bar\xi,\cdot)(\bar\eta)-aD^+_{e_{2,3}}\mathcal I(\bar\eta)),  \frac{\sqrt{\omega_{2,3}}}{h}(-\frac{1}{\delta^2}D^-_{e_{2,3}}\mathcal J(\bar\xi,\cdot)(\bar\eta)-aD^-_{e_{2,3}}\mathcal I(\bar\eta))\Big), 
\end{align}  
where the definition of the difference operator $D^{\pm}_{e_{j,k}}$ is given by \eqref{D_notation}, and we use the monotonicity of the mapping $\vec{\mathcal G}$ and the inequality \eqref{leadsto1} to replace differences of $U$ in $\vec{\mathcal G}$ by those of $\mathcal J$ and $\mathcal I$.   By Lemma \ref{lemma1} with $f=\mathcal J$, we deduce that when $h$ is small,  
\begin{align}\label{nabla_relation}
 \frac{1}{h}D^+_{e_{1,2}}\mathcal J(\bar\xi,\cdot)(\bar\eta)&=\frac 1hD^+_{(1,0)}\widetilde{\mathcal J}(\bar x,\cdot)(\bar y)=\frac 1h(\widetilde{\mathcal J}(\bar x,(\bar y_1+h,\bar y_2))-\widetilde{\mathcal J}(\bar x,(\bar y_1,\bar y_2)))\notag\\
 &=\partial_{y_1}\widetilde{\mathcal J}(\bar x,\bar y)+O(h)=
 \nabla ^{e_{1,2}}\mathcal J(\bar\xi,\cdot)(\bar\eta)+O(h), 
\end{align}
where $\bar x=\Pi(\bar\xi),\bar y=(\bar y_1,\bar y_2)=\Pi(\bar\eta)$ and $O(h)$ denotes  the infinitesimal of the same order of $h$.  
Similarly, one can derive $\frac1hD_{e_{l,k}}^{\pm}\mathcal J(\bar\xi,\cdot)(\bar\eta)=\nabla^{e_{l,k}}\mathcal J(\bar\xi,\cdot)(\bar\eta)+O(h)$ for all $1\leq l<k\leq 3.$ 
 Estimates of subterms in the form of  $D^{\pm}_{e_{l,k}}\mathcal I(\bar\eta)$ can be obtained in a similar way. Inserting then these estimates into \eqref{U_eq1},  and utilizing the consistency and local Lipschitz property of  $\mathcal G,$  we derive from \eqref{U_eq1} that 
\begin{align}\label{rela2}
-\frac{\sigma}{4T}-\frac{\varrho}{(T-\bar s)^2}+\frac{\bar t-\bar s}{\delta^2}+O(\frac{\tau+h}{\delta^2})+{\mathcal F}(\bar\eta)
+\mathcal H(\bar\eta,\bar q)-\mathcal O_{\bar\eta}(\bar q)\ge 0,
\end{align}
where we set $\bar q=-\bar q_1-\bar q_2:=-\frac{1}{\delta^2}\nabla_{\mathcal W}\mathcal J(\bar\xi,\cdot)(\bar\eta)-a\nabla_{\mathcal W}\mathcal I(\bar\eta).$ Here, the local Lipschitz property of $\mathcal G$ can be used due to that there is some $R=1+C_0>0$  so that arguments of $\mathcal G(\bar\eta,\cdot)$ in \eqref{U_eq1} are  no larger than $R,$ where the constant $C_0$ will be specified later (see \eqref{R1}).   
We also remark that term $O(\frac{\tau+h}{\delta^2})$ in \eqref{rela2} may depend on parameter $a$ due to the local Lipschitz property of $\mathcal G$ and $\bar\eta\in\mathcal P_{ac_1}(G).$
By virtue of $-\bar q_1=\bar p_1,$ (H-\romannumeral1), (H-\romannumeral2), and (H-\romannumeral5),  one obtains that for $\lambda_2\in(0,1),$
\begin{align*}
\mathcal H(\bar\eta,\bar q)\leq \lambda_2\mathcal H(\bar\eta,\frac{\bar p_1}{\lambda_2})+(1-\lambda_2)\mathcal H(\bar\eta,-\frac{\bar q_2}{1-\lambda_2})\leq \lambda_2\bar{\gamma}(\frac{1}{\lambda_2})\mathcal H(\bar\eta,\bar p_1)+C(a\mathcal I(\bar\eta))^{\kappa}.
\end{align*}
This together with \eqref{rela2} yields  
\begin{align}\label{rela11}
-\frac{\sigma}{4T}-\frac{\varrho}{T^2}+\frac{\bar t-\bar s}{\delta^2}+O(\frac{\tau+h}{\delta^2})+\mathcal F(\bar\eta)+\lambda_2\bar\gamma(\frac{1}{\lambda_2})\mathcal H(\bar\eta,\bar p_1)-\mathcal O_{\bar\eta}(\bar q)+C(a\mathcal I(\bar\eta))^{\kappa}\ge 0.
\end{align}
Combining \eqref{rela11} with \eqref{rela1} concludes 
\begin{align}\label{error_esti1}
&\frac{\sigma}{2T}+\gamma(\frac{\lambda_1}{\lambda})\mathcal H(\bar\xi,\bar p_1)-\lambda_2\bar\gamma(\frac{1}{\lambda_2})\mathcal H(\bar\eta,\bar p_1)+\mathcal F(\bar\xi)-\mathcal F(\bar \eta)\notag\\
&\leq -\frac{2\varrho}{T^2}+(1-\lambda)\|\mathcal F\|_{\infty}+\mathcal O_{\bar\xi}(\bar p)-\mathcal O_{\bar\eta}(\bar q)+\omega_0(a,\delta)+O(\frac{\tau+h}{\delta^2}),
\end{align}
where  $\omega_0(a,\delta):=C((a\mathcal I(\bar\xi))^{\kappa}+(a\mathcal I(\bar\eta))^{\kappa})$. 

\textit{Step 2: Error estimate.} 
By (H-\romannumeral1), we take $\lambda_2$ close to 1 such that $z_0:=\gamma(\frac{\lambda_1}{\lambda})-\lambda_2\bar\gamma(\frac{1}{\lambda_2})>0.$ 
Then applying (H-\romannumeral3) to the term $\mathcal H(\bar\eta,\bar p_1)$ in \eqref{error_esti1}, and using the positivity of $z_0,\varrho$, we obtain \begin{align}\label{sigmaineq}
&\quad \frac{\sigma}{2T}+\lambda_2\bar{\gamma}(\frac{1}{\lambda_2})(\mathcal H(\bar\xi,\bar p_1)-\mathcal H(\bar\eta,\bar p_1))+\mathcal F(\bar\xi)-\mathcal F(\bar \eta)+z_0\theta_{ac_1}\|\bar p_1\|^{\kappa}_{\bar\xi}\notag\\
&\leq (1-\lambda)\|\mathcal F\|_{\infty}+\mathcal O_{\bar\xi}(\bar p)-\mathcal O_{\bar\eta}(\bar q)+\omega_0(a,\delta)+O(\frac{\tau+h}{\delta^2}). \end{align} 
By \cite[Example 5.3]{MCC}, we have 
\begin{align}\label{eq02051} \mathcal O_{\bar\xi}(\bar p_2)+\mathcal O_{\bar\eta}(\bar q_2)\leq 0,\quad \mathcal O_{\bar\xi}(\bar p_1)+\mathcal O_{\bar\eta}(\bar q_1)\leq C(ac_1)\delta^{-2}\|\bar\xi-\bar\eta\|^2_{l^2}
\end{align} with some $C(ac_1)>0.$  
Applying (H-\romannumeral4) to \eqref{sigmaineq} yields 
\begin{align}\label{good_esti}
\frac{\sigma}{2T}&\leq (1-\lambda)\|\mathcal F\|_{\infty}+\lambda_2\bar\gamma(\frac{1}{\lambda_2})\mathfrak m_{ac_1}(\|\bar\xi-\bar\eta\|_{l^2})\|\bar p_1\|^{\kappa}_{\bar\xi}+|\mathcal F(\bar\xi)-\mathcal F(\bar\eta)|+C(a)\delta^{-2}\|\bar\xi-\bar\eta\|^2_{l^2}\notag\\
&\quad +\omega_0(a,\delta)+O(\frac{\tau+h}{\delta^2})+C(a)|\|\bar p_1\|_{\bar \xi}-\|\bar p_1\|_{\bar\eta}|(\|\bar p_1\|^{\kappa-1}_{\bar\xi}+ \|\bar p_1\|^{\kappa-1}_{\bar\eta}), 
\end{align}
where the term $|\|\bar p_1\|_{\bar \xi}-\|\bar p_1\|_{\bar\eta}|$ can be estimated as 
\begin{align}\label{small_diffxi}
|\|\bar p_1\|_{\bar \xi}-\|\bar p_1\|_{\bar\eta}|\leq \frac12\sum_{(i,j)\in E}(\bar p_1)_{ij}^2|g_{i,j}(\bar\xi)-g_{i,j}(\bar\eta)|(\|\bar p_1\|_{\bar\xi}+\|\bar p_1\|_{\bar\eta})^{-1}\leq C(a)\|\bar\xi-\bar\eta\|_{l^2}
\end{align}
for some $C(a)>0,$ due to assumption (g-\romannumeral1) on $g$ (see Assumption \ref{assumption_g}). 

Recalling \eqref{smalldelta}, we have shown that $\|\bar \xi-\bar\eta \|_{l^2}\leq C\delta$  with some $C>0.$ We now aim to improve the smallness estimate of $\|\bar\xi-\bar\eta\|_{l^2}$. We  also need to give the upper bound estimates of $\|\bar p_1\|_{\bar\xi},\|\bar p_1\|_{\bar\eta}.$ Note that when $\bar t,\bar s,\bar\eta$ are fixed, the mapping $\xi\mapsto u(\bar t,\xi)-\frac{1}{\lambda \delta^2}\mathcal J(\xi,\bar\eta)-\frac{a}{\lambda}\mathcal I(\xi)$ is maximized at $\xi=\bar\xi\in\mathcal P_{ac_1}(G)$, due to that $\Psi_{\varrho,a,\delta}$ takes maximum at $(\bar t,\bar\xi,\bar s,\bar\eta).$ 
It follows from Proposition \ref{prop_exact} that $$\Big|-\frac{1}{\lambda\delta^2}\partial_{x_1}\widetilde{\mathcal J}(\cdot,\bar y)(\bar x)-\frac{a}{\lambda}\partial_{x_1}\widetilde{\mathcal I}(\bar x)\Big|\vee \Big|-\frac{1}{\lambda\delta^2}\partial_{x_2}\widetilde{\mathcal J}(\cdot,\bar y)(\bar x)-\frac{a}{\lambda}\partial_{x_2}\widetilde{\mathcal I}(\bar x)\Big|\leq \sqrt 2L_2(ac_1),$$ where $L_2(\cdot)$ is the Lipschitz constant of $u$ given in Proposition \ref{prop_exact}. 
 Hence by Lemma \ref{lemma1}, we deduce that   $\|-\frac{1}{\delta^2\lambda}\nabla_{\mathcal W} \mathcal J(\cdot,\bar\eta)(\bar\xi)-\frac{a}{\lambda} \nabla_{\mathcal W}  \mathcal I(\bar\xi)\|_{l^2}\leq C L_2(ac_1).$
  This, together with $\lambda<1$  yields 
\begin{align}
\label{L2ineq}
\|\frac{1}{\delta^2}\nabla_{\mathcal W}\mathcal J(\cdot,\bar\eta)(\bar\xi)\|_{l^2}=\|\frac{1}{\delta^2}\nabla_G(\bar\xi-\bar\eta)\|_{l^2}\leq C(L_2(ac_1)+a\|\nabla_{\mathcal W}\mathcal I(\bar\xi)\|_{l^2})\leq C(L_2(ac_1)+a^{-1})
\end{align} for some $C:=C(d)>0,$ 
where we use the triangle inequality and 
$\|\nabla_{\mathcal W}\mathcal I(\bar\xi)\|_{l^2}\leq C\sum_{l=1}^d\frac{1}{\bar\xi_l^2}\leq Ca^{-2}$ for $ \bar\xi\in\mathcal P_{ac_1}(G).
$ By virtue of the fact $\nabla_{\mathcal W}\mathcal J(\cdot,\bar\eta)(\bar\xi)=-\nabla_{\mathcal W}\mathcal J(\bar\xi,\cdot)(\bar\eta)$ and relation \eqref{L2ineq}, we obtain that for some $C:=C(d)>0,$
\begin{align}\label{R1}
&\quad \|\frac{1}{\delta^2}\nabla_{\mathcal W}\mathcal J(\bar\xi,\cdot)(\bar\eta)+a\nabla_{\mathcal W}\mathcal I(\bar\eta)\|_{\infty}\notag\\
&\leq C\|\frac{1}{\delta^2}\nabla_{\mathcal W}\mathcal J(\bar\xi,\cdot)(\bar\eta)+a\nabla_{\mathcal W}\mathcal I(\bar\eta)\|_{l^2}
\leq C(L_2(ac_1)+a^{-1})=:C_0. 
\end{align} 
In addition, from \eqref{L2ineq} and the fact that $\bar\xi,\bar\eta$ are discrete probability measures on graphs, 
we see that 
\begin{align}\label{delta^2}
\|\bar\xi-\bar\eta\|_{l^2}\leq C\|\nabla_G(\bar\xi-\bar\eta)\|_{l^2}\leq C(a)\delta^2, 
\end{align}  
and $\|\bar p_1\|_{\bar\xi}\vee \|\bar p_1\|_{\bar\eta}\leq C(a)\|\bar p_1\|_{l^2}\leq C(a)$ due to \eqref{L2ineq} for $\bar\xi,\bar\eta\in\mathcal P_{ac_1}(G)$.

 Set $\delta=(\tau+h)^{\frac14}.$ Using \eqref{small_I}, \eqref{delta^2}, the CFL-type cindition \eqref{cond_CLF}, and continuity of $\mathcal F$, and letting $\lambda\to1,h\to0$ and $a\to0$ in  \eqref{good_esti} yield the convergene.

\textbf{Analysis for cases $\bar t\ge 0,\bar s=0$ and $\bar t=0,\bar s>0.$} 

For the case 
 $\bar t\ge 0,\bar s=0,$ it follows from the triangle inequality, $\bar\xi,\bar\eta\in\mathcal P_{ac_1}(G)$, and Proposition \ref{prop_exact} that
\begin{align}\label{threelimit}
\Phi_{\varrho,a,\delta}(\bar t,\bar\xi,0,\bar \eta)&\leq \lambda |u(\bar t,\bar\xi)-u(\bar t,\bar\eta)|+\lambda |u(\bar t,\bar\eta)-u(0,\bar\eta)|+(1-\lambda)|u(0,\bar\eta)|\notag\\
&\leq \lambda L_2(ac_1)\|\bar\xi-\bar\eta\|_{l^2}+\lambda L_1\bar t+(1-\lambda)\|\mathcal U_0\|_{\infty}.
\end{align}
By \eqref{limit1}, we have 
\begin{align*}
\lim_{\varrho\to0}\lim_{a\to0}\lim_{\delta\to0}\Phi_{\varrho,a,\delta}(\bar t,\bar\xi,0,\bar\eta)=\sup_{(s,\eta)\in\mathbb Q^{\tau}_{h,0}}\Phi_0(s,\eta)\ge \sup_{(s,\eta)\in\mathbb Q^{\tau}_{h,0}}(\lambda u(s,\eta)-U(s,\eta))-\frac{\sigma}{2},
\end{align*}
which together with \eqref{threelimit} and \eqref{smalldelta} leads to 
$\sup_{(s,\eta)\in\mathbb Q^{\tau}_{h,0}}(\lambda u(s,\eta)-U(s,\eta))\leq (1-\lambda)\|\mathcal U_0\|_{\infty}+\frac{\sigma}{2}.$ Recalling the definition of $\sigma,$ we obtain \begin{align*}
\sigma=\sup_{(s,\eta)\in\mathbb Q^{\tau}_{h,0}}(u(s,\eta)-U(s,\eta))&\leq (1-\lambda)\sup_{(s,\eta)\in\mathbb Q^{\tau}_{h,0}}|u(s,\eta)|+\sup_{(s,\eta)\in\mathbb Q^{\tau}_{h,0}}(\lambda u(s,\eta)-U(s,\eta))\\
&\leq (1-\lambda)(M+\|\mathcal U_0\|_{\infty})+\frac{\sigma}{2}.
\end{align*}
Letting $\lambda\to1$ yields $\sigma=0$ for any $\tau,h.$

For the case $\bar t=0,\bar s>0,$ utilizing Lemma \ref{lem1} (\romannumeral4), we have 
\begin{align*}
\Phi_{\varrho,a,\delta}(0,\bar\xi,\bar s,\bar\eta)&\leq \lambda |u(0,\bar\xi)-u(0,\bar\eta)|+(1-\lambda)|u(0,\bar\eta)|+|u(0,\bar\eta)-U(\bar s,\bar\eta)|\\
&\leq L_2(ac_1)\lambda \|\bar\xi-\bar\eta\|_{l^2}+(1-\lambda)\|\mathcal U_0\|_{\infty}+K\bar s.
\end{align*}
 The remaining proof is similar to that of the case $\bar t\ge 0,\bar s=0$ and thus is omitted.  
\end{proof}
Moreover, from the above proof, using the $l^2$-Lipschitz condition on $\mathcal F$ and the linear growth property of $\mathfrak m_{ac_1}$, we have \begin{align*}
\sigma\leq C(a)\delta^2+O(\frac{\tau+h}{\delta^2})+\omega_0(a,\delta)\leq C(a)(\tau+h)^{\frac12}+\omega_0(a,h)
\end{align*}
for any fixed $a\in(0,1),$ where  $\omega_0(a,h)$ is such that $\lim\limits_{a\to0}\limsup\limits_{h\to0}\omega_0(a,h)=0.$
Thus we have obtained the convergence order $\frac12$ except the term $\omega_0(a,h),$ which arises because the viscosity solution is defined in the interior of the simplex. 
To quantify the convergence order of the proposed scheme, we take Assumption \ref{ass_H2} into consideration and consider the approximation error on $\mathcal P_{\epsilon}(G)$.

\begin{proof}[Proof of Theorem \ref{thm1} (\romannumeral2).] Denote  
$\mathbb Q^{\tau}_{h,\epsilon}:=\{(t_n,(j,k,l)h)\in[0,T]\times {\mathcal P}_{\epsilon}^h(G): (j, k,l)\in\mathcal  N_{\epsilon};n=0,\ldots,N_T\}$. 
Without loss of generality, we still assume that 
$\sigma:=\sup_{(t_n,\xi)\in{\mathbb Q}^{\tau}_{h,\epsilon}}( u(t_n,\xi)-U(t_n,\xi))>0.$ 

 For $(t,\xi)\in[0,T]\times {\mathcal P}_{\epsilon}(G)$ and $(s,\eta)\in{\mathbb Q}^{\tau}_{h,\epsilon},$ define a new  auxiliary function 
\begin{align}\label{new_aux}
\Psi_{\delta}(t,\xi,s,\eta)= u(t,\xi)- U(s,\eta)-\frac{\sigma}{4T}(t+s)-\frac{1}{2\delta^2}(|t-s|^2+2\mathcal J(\xi,\eta)).
\end{align}
Since  
 $\Psi_{\delta}$ is continuous with variable $(t,\xi)\in[0,T]\times {\mathcal P}_{\epsilon}(G)$ and takes finite value on  ${\mathbb Q}^{\tau}_{h,\epsilon}$,
  there exists a maximizer 
 $(\hat t,\hat \xi,\hat s,\hat\eta)\in [0,T]\times {\mathcal P}_{\epsilon}(G)\times {\mathbb Q}^{\tau}_{h,\epsilon}$ for $\Psi_{\delta}$ when $\delta$ is fixed. 
 Similar to the argument of \eqref{smalldelta}, we have $
 M_{\delta}+\frac{|\hat t-\hat s|^2+\|\hat\xi-\hat\eta\|^2_{l^2}}{4\delta^2}=\Psi_{\sqrt2\delta}(\hat t,\hat \xi,\hat s,\hat\eta)\leq M_{\sqrt2\delta},
$ where $M_{\delta}:=\sup_{[0,T]\times\mathcal P_{\epsilon}(G)\times\mathbb Q^{\tau}_{h,\epsilon}}\Psi_{\delta}( t,\xi, s,\eta)$, and thus 
 \begin{align}\label{smalldelta2}\lim_{\delta\to0}\frac{|\hat t-\hat s|^2+\|\hat\xi-\hat\eta\|^2_{l^2}}{\delta^2}=0,\quad \lim_{\delta\to0}M_{\delta}=\sup_{(s,\eta)\in\mathbb Q^{\tau}_{h,\epsilon}}(u(s,\eta)-U(s,\eta)-\frac{\sigma s}{2T}).
 \end{align}
Moreover, by the definition of $M_{\delta}$ and \eqref{smalldelta2}, it holds that  
\begin{align}\label{rela4}
\lim_{\delta\to0}M_{\delta}=\lim_{\delta\to0}\Psi_{\delta}(\hat t,\hat\xi,\hat s,\hat\eta)\ge \frac{\sigma}{2}. 
\end{align}
 Next we split the proof of the convergence order into three cases based on whether $\hat t$ or $\hat s$ equals zero.

\textbf{Case 1: $\hat t>0,\hat s>0.$} When $\hat s,\hat\eta$ are fixed, 
the mapping 
$
(t,\xi)\mapsto  u(t,\xi)-\frac{\sigma t}{4T}-\frac{1}{2\delta^2}(|t-\hat s|^2+2\mathcal J(\xi,\hat\eta))
$
takes the maximum value at point $(\hat t,\hat\xi).$  By the definition of the viscosity subsolution, we have that when $\hat\xi\in\mathcal P_{\epsilon}^{\circ}(G),$
\begin{align}\label{ineq1}
\frac{\sigma}{4T}+\frac{\hat t-\hat s}{\delta^2}+\mathcal H\big(\hat\xi, \frac{\nabla_G(\hat \xi-\hat\eta)}{\delta^2}\big)+{\mathcal F}(\hat\xi)\leq \mathcal O_{\hat\xi}(\frac{\nabla_G(\hat \xi-\hat\eta)}{\delta^2}).
\end{align}
When $\hat\xi\in\partial \mathcal P_{\epsilon}(G),$ for $a\in(0,1),$ we define  
\begin{align*}
\Psi_{a,\delta}(t,\xi)=u(t,\xi)-\frac{\sigma t}{4T}-\frac{1}{2\delta^2}(|t-\hat s|^2+2\mathcal J(\xi,\hat\eta))-\frac{a}{\mathfrak b(\xi)},
\end{align*}
where $\mathfrak b$ satisfies Assumption \ref{ass_H2}. 
For each $a\in(0,1),$ $\Psi_{a,\delta}$ has a local maximizer $(t_a,\xi_a)$ with $\xi_a\in\mathcal P_{\epsilon}^{\circ}(G)$ and \begin{align}\label{local_max}
(t_a,\xi_a)\to (\hat t,\hat\xi),\text{ as } a\to0; 
\end{align} see Appendix \ref{app_vari} for its proof.   
 Hence, according to the definition of viscosity subsolution, we have
\begin{align*}
&\frac{\sigma}{4T}+\frac{t_a-\hat s}{\delta^2}+\mathcal F(\xi_a)+\mathcal H\Big(\xi_a,\frac{\nabla_G (\xi_a-\hat\eta)}{\delta^2}-a\mathfrak b^{-2}(\xi_a)\nabla_{\mathcal W} \mathfrak b(\xi_a)\Big)\\
&\leq \mathcal O_{\xi_a}\Big( \frac{\nabla_G (\xi_a-\hat\eta)}{\delta^2}-a\mathfrak b^{-2}(\xi_a)\nabla_{\mathcal W}\mathfrak b(\xi_a)\Big).
\end{align*}
By Assumption \ref{ass_H2}, letting $a\to0$ also leads to \eqref{ineq1}.

 The mapping 
$
{\mathbb Q}^{\tau}_{h,\epsilon}\ni (s,\eta)\to  U(s,\eta)+\frac{\sigma s}{4T}+\frac{1}{2\delta^2}(|\hat t-s|^2+2\mathcal J(\hat\xi,\eta))$
 takes the minimum value at $(\hat s,\hat\eta):=(n_1\tau,(j_1,k_1,l_1)h).$ This leads to that 
 \begin{align}\label{leadsto}
 U^n_{j,k,l}\ge  U^{n_1}_{j_1,k_1,l_1}+\frac{\sigma}{4T}(n_1-n)\tau+\frac{1}{2\delta^{2}}\big(|\hat t-n_1\tau|^2-|\hat t-n\tau|^2 +2\mathcal J(\hat\xi,\hat\eta)-2\mathcal J(\hat\xi,\eta)\big).
\end{align}
 Recalling Remark \ref{remark1}, one can utilize the monotonicity of the numerical method \eqref{explicit1} and \eqref{leadsto} with $n=n_1-1$ to conclude that when $\hat\eta\in \mathcal P_{\epsilon}^h(G)\backslash \partial\mathcal P_{\epsilon}(G),$\begin{align}\label{tildeY_eq1}
& U^{n_1}_{j_1,k_1,l_1}=(\vec{\mathcal G}( U^{n_1-1}))_{j_1,k_1,l_1}\ge  U^{n_1}_{j_1,k_1,l_1}+\frac{\sigma\tau}{4T}+\frac{1}{2\delta^2}\big(|\hat t-\hat s|^2- |\hat t-\hat s+\tau|^2\big)\notag\\
& -\tau{\mathcal F}(\hat\eta) -\tau\mathcal G\Big(\hat\eta, -\frac{\sqrt{\omega_{1,2}}}{h\delta^2}D^+_{e_{1,2}}\mathcal J(\hat\xi,\cdot)(\hat\eta),  -\frac{\sqrt{\omega_{1,2}}}{h\delta^2}D^-_{e_{1,2}}\mathcal J(\hat\xi,\cdot)(\hat\eta); -\frac{\sqrt{\omega_{1,3}}}{h\delta^2}D^+_{e_{1,3}}\mathcal J(\hat\xi,\cdot)(\hat\eta),\notag\\ &-\frac{\sqrt{\omega_{1,3}}}{h\delta^2}D^-_{e_{1,3}} \mathcal J(\hat\xi,\cdot)(\hat\eta);-\frac{\sqrt{\omega_{2,3}}}{h\delta^2}D^+_{e_{2,3}}\mathcal J(\hat\xi,\cdot)(\hat\eta) ,-  \frac{\sqrt{\omega_{2,3}}}{h\delta^2}D^-_{e_{2,3}}\mathcal J(\hat\xi,\cdot)(\hat\eta) \Big).
\end{align}
Similar to the proof of \eqref{nabla_relation}, utilizing  Lemma \ref{lemma1},  
and then by means of the consistency and local Lipschitz condition of the numerical Hamiltonian function $\mathcal G,$  we arrive at
\begin{align}\label{rela22}
-\frac{\sigma}{4T}+\frac{\hat t-\hat s}{\delta^2}+O(\frac{\tau+h}{\delta^2})+{\mathcal F}(\hat\eta)+\mathcal H\big(\hat\eta,\frac{\nabla_G(\hat\xi-\hat\eta)}{\delta^2}\big)-\mathcal O_{\hat\eta}\big(\frac{\nabla_G(\hat\xi-\hat\eta)}{\delta^2}\big)\ge 0.
\end{align}
This, combining \eqref{ineq1} and \eqref{eq02051} implies that
\begin{align}\label{rela111}
&\frac{\sigma}{2T}+\mathcal H(\hat\xi,\frac{\nabla_G(\hat\xi-\hat\eta)}{\delta^2} )-\mathcal H(\hat\eta,\frac{\nabla_G(\hat\xi-\hat\eta)}{\delta^2} )+{\mathcal F}(\hat\xi)-{\mathcal F}(\hat\eta )\leq O(\frac{\tau+h}{\delta^2})+\frac{C_{\epsilon}\|\hat\xi-\hat\eta\|^2_{l^2}}{\delta^2}.
\end{align} 
When $\hat\eta\in\partial\mathcal P_{\epsilon}(G)$,  by \eqref{boundary1} or  \eqref{boundary}, repeating arguments of \eqref{tildeY_eq1}, we derive that 
$\frac{\sigma}{2T}+\mathcal H(\hat\xi,\frac{\nabla_G(\hat\xi-\hat\eta)}{\delta^2} )-\mathcal H(\hat\eta,\frac{\nabla_G(\hat\xi-\hat\eta)}{\delta^2} )+{\mathcal F}(\hat\xi)-{\mathcal F}(\hat\eta )\leq O(\frac{\tau+h}{\delta^2})+\frac{C_{\epsilon}\|\hat\xi-\hat\eta\|^2_{l^2}}{\delta^2}+C\frac{h}{\tau}\|\hat\xi-\hat\eta\|_{l^2}.$ 
By (H-\romannumeral4) we have that 
\begin{align*}
&\mathcal H(\hat\xi,\bar p)-\mathcal H(\hat \eta,\bar p)\ge -\mathfrak m_{\epsilon} (\|\hat\xi-\hat\eta\|_{l^2})\|\bar p\|^{\kappa}_{\hat\xi}-C_{\epsilon}\big|\|\bar p\|_{\hat\xi}-\|\bar p\|_{\hat\eta}\big|(\|\bar p\|^{\kappa-1}_{\hat \xi}+\|\bar p\|^{\kappa-1}_{\hat\eta }).
\end{align*}
Then combining \eqref{small_diffxi}, the $l^2$-Lipschitz condition on $\mathcal F$ and the linear growth condition on the modulus of $\mathfrak m_{\epsilon}$,
we deduce from \eqref{rela111} that for some $C$ depending on $\epsilon$,
\begin{align}\label{case1}
\sigma\leq O(\frac{\tau+h}{\delta^2})+C(\|\hat\xi-\hat\eta\|_{l^2}+\frac{\|\hat\xi-\hat\eta\|_{l^2}^2}{\delta^2}).
\end{align}

Let $\delta=(\tau+h)^{\frac14}.$ Recalling \eqref{smalldelta2}, we have shown that $|\hat t-\hat s|\vee \|\hat \xi-\hat\eta \|_{l^2}\leq C\delta.$  We now aim to improve the smallness estimates of $|\hat t-\hat s|, \|\hat \xi-\hat\eta \|_{l^2}$ to obtain the convergence order $\frac12$. Recalling that $\Psi_{\delta}$ is maximized at $(\hat t,\hat \xi,\hat s,\hat \eta),$ we have that when $\hat t,\hat s,\hat\xi$ are fixed, 
the mapping $\xi\to  u(\hat  t,\xi)-\frac{1}{2\delta^2}\| \xi-\hat\eta\|^2_{l^2}$ is maximized at $\xi=\hat\xi$.  Similar to the argument of \eqref{L2ineq}, 
we derive that when $\hat\xi\in\mathcal P^{\circ}_{\epsilon}(G),$ 
$
\|\hat \xi-\hat\eta \|_{l^2}\leq CL_2(\epsilon)\delta^2.
$ 
This yields \begin{align}\label{result1}
\sigma\leq C(\tau^{\frac12}+h^{\frac12})
\end{align} for $\hat\xi\in\mathcal P_{\epsilon}^{\circ}(G),\hat\eta\in \mathcal P_{\epsilon}^h(G)$. Moreover, by replacing the constant $a$ in  \eqref{L2ineq} by $\epsilon$ and using the Lipschitz cindition of $\mathcal U_0$, we derive that the arguments of $\mathcal G(\hat\eta,\cdot)$ in \eqref{tildeY_eq1} are no larger than some $R>1+\max\{\sqrt 2\mathrm{Lip}(\mathcal U_0),C(d)L_2(\epsilon)\}.$

\textbf{Case 2: $\hat t\ge 0,\hat s=0.$} When $\hat s,\hat\xi,\hat\eta$ are fixed, the mapping $t\to u(t,\hat\xi)-\frac{\sigma t}{4T}-\frac{1}{2\delta^2}|t-\hat s|^2$ is maximized at $t=\hat t>0.$ Then one can derive from Proposition \ref{prop_exact} that  
\begin{align}\label{esti2}
\hat t-\hat s\leq (L_1-\frac{\sigma}{4T})\delta^2, \quad \hat t=T,\text{ and }|\hat t-\hat s|\leq (L_1+\frac{\sigma}{4T})\delta^2, \quad \hat t\in(0,T).
\end{align}
From \eqref{rela4},  Proposition \ref{prop_exact} and \eqref{esti2}, it follows that 
\begin{align*}
\frac{\sigma}{2}&\leq \Psi_{\delta}(\hat t,\hat\xi,0,\hat\eta)\leq | u(\hat t,\hat\xi)- u(\hat t,\hat\eta)|+| u(\hat t,\hat\eta)- u(0,\hat\eta)|
\leq L_2\|\hat\xi-\hat\eta\|_{l^2}+L_1\hat t,
\end{align*}
which implies that \begin{align}\label{result2}
\sigma\leq C(\tau^{\frac12}+h^{\frac12}).
\end{align} 

\textbf{Case 3: $\hat t=0,\hat s>0.$} In this case,  using \eqref{rela4} and  Proposition \ref{prop_exact} again, combining  Lemma \ref{lem1} (\romannumeral4) and the Lipschitz condition on the initial value, we obtain
\begin{align*}
\frac{\sigma}{2}
&\leq u(0,\hat\xi)-u(0,\hat\eta)+U^0(\hat\eta)-U^{n_1}(\hat\eta)
\leq L_2\|\hat\xi-\hat\eta\|_{l^2}+K\hat s.
\end{align*}
From $\Psi_{\delta}(\hat t,\hat\xi,\hat s,\hat\eta )=\Psi_{\delta}(0,\hat \xi,\hat s,\hat\eta)\ge \Psi_{\delta}(0,\hat\xi,\hat s-\tau,\hat\eta)$ due to that $(\hat t,\hat\xi,\hat s,\hat\eta)$ is the maximizer,  we conclude that 
\begin{align*}
- U^{n_1}_{j_1,k_1,l_1}-\frac{\sigma \hat s}{4T}-\frac{1}{2\delta^2}\hat s^2\ge - U^{n_1-1}_{j_1,k_1,l_1}-\frac{\sigma (\hat s-\tau)}{4T}-\frac{1}{2\delta^2}(\hat s-\tau)^2.
\end{align*}
Thus $
\frac{1}{\delta^2}(\hat s^2-(\hat s -\tau)^2)\leq  U^{n_1-1}_{j_1,k_1,l_1}- U^{n_1}_{j_1,k_1,l_1}-\frac{\sigma\tau}{4T}\leq K\tau. 
$ 
Thus we obtain $\hat s \leq C(\delta^2+\tau),$ which yields that \begin{align}\label{result3}
\sigma\leq C(\delta^2+\tau)\leq C(\tau^{\frac12}+h^{\frac12}).
\end{align}

Combining estimates \eqref{result1}, \eqref{result2} and \eqref{result3} finishes the proof of the convergence order.  
\end{proof}

\section{Numerical experiment}\label{sec_6}
In this section, we present numerical experiments to illustrate the convergence of the proposed numerical methods.     
We consider the HJE on the Wasserstein space on graphs with noise intensity $\lambda_1>0$  
\begin{align}\label{lap_eq}
\partial_tu(t,\xi)+\mathcal H(\xi,\nabla_{\mathcal W} u(t,\xi))+\mathcal F(\xi)=\lambda_1\Delta_{\mathrm{ind}}u(t,\xi),
\quad u(0,\cdot)=\mathcal U_0.
\end{align}  
We set $d=3,$ 
$\omega_{1,2}=\omega_{1,3}=\omega_{2,3}=1,$ $g(r,s)=\frac{r+s}{2}.$ Define $\mathcal I_{\theta}(\xi):=\sum_{i=1}^3\xi_i^{-\theta}$. For the implicit method, we restrict the maximum number of iterations to $10$ and set the tolerance error to $10^{-6}.$ In the numerical experiments, the specified boundary extrapolations \eqref{boundary1} and \eqref{boundary}  are implemented via MATLAB’s \texttt{scatteredInterpolant} function, with  the \texttt{ExtrapolationMethod} being set to   `\texttt{nearest}' and `\texttt{linear}', respectively.

We take the initial value $\mathcal U_0(\xi)=\|\xi\|^2_{l^2},$ the potential function $\mathcal F\equiv 0$, and the coefficient function $\mathfrak a(\xi)=\mathcal I^{-2}(\xi)$ with $\mathcal I(\xi)=\sum_{i=1}^3\frac{1}{\xi_i}.$ The Hamiltonian is formulated as $\mathcal H(\xi,p)=\mathfrak a(\xi)\|p\|_{\xi}^{2}.$ The parameters are set as follows: $\epsilon=0.01,$ $\frac{\tau}{h}=0.05$, $\lambda_1=0.5$
 and $T=0.4.$ 
Table \ref{table1} presents an investigation of the temporal convergence order for proposed explicit and implicit methods.   
The reference solution is generated using the same numerical method with a refined spatial mesh ($N=512$) to ensure accuracy. The numerical results demonstrate that proposed methods can achieve first-order temporal convergence in both $L^{\infty}$ and $L^1$ norms. This observation 
 is consistent with theoretical guarantees for classical HJEs on Euclidean space (see e.g. \cite{splitting,firstorder}), where the rigorous error bounds can be established by requiring structural assumptions such as semi-concavity conditions or utilizing some techniques such as operator splitting. 
 We will study the analysis of the first-order accuracy of the proposed schemes in the future.

\begin{table}[tbhp]
\begin{tabular}{ccccc|ccccc}
\hline
\multicolumn{5}{c|}{Explicit method}&\multicolumn{4}{c}{Implicit method}\\
\hline
N & $L^{\infty}$-error  & Order & $L^1$-error  & Order & $L^{\infty}$-error  & Order & $L^1$-error  & Order\\
\hline
 16& 0.0121 & -- &0.0098 &--&0.0137& -- & 0.0107&--\\
\hline
 32&0.0052& 1.2234& 0.0041& 1.2406& 0.0059&1.2149&0.0046 & 1.2218\\
\hline
 64&0.0021 &1.3059 &0.0017 &1.2816 &0.0023&1.3435 & 0.0019&1.2636\\
\hline
 128&8.7759e-04 &1.2587 & 6.9843e-04&1.2883&9.7202e-04&1.2556&7.6839e-04&1.3207 \\
\hline
256&3.5721e-04 &1.2968 & 2.5380e-04& 1.4604&3.5537e-04&1.4517&2.7784e-04&1.4676\\
\hline
\end{tabular}
\vspace{0.1in}
\caption{{Approximation error}}
\label{table1}
\end{table}
We present some numerical experiments to illustrate the influences of the graph individual noise and the coefficients of Hamiltonians in Appendix \ref{sec_app2}.

\bibliographystyle{plain}
\bibliography{references.bib}

\begin{appendix} 
\section{Some Verifications and Proofs}\label{app_vari}
\subsection{Verifications of examples}

\begin{proof}[\textbf{Verification of Example \ref{exam2}}]
  It can be easily checked that Assumption \ref{ass_G} (\romannumeral1) is satisfied. Based on the fact that $(p_{k,l}^-)^2+(p^+_{k,l})^2=p_{k,l}^2$ and $\mathbf 1_{\{\xi_k\leq \xi_l\}}+\mathbf 1_{\{\xi_k>\xi_l\}}\equiv 1$  for $1\leq k<l\leq d$, 
 the function $\mathcal G$ satisfies Assumption \ref{ass_G} (\romannumeral2) for the Hamiltonian function $\mathcal H$ given in Example \ref{exam1}. As for Assumption \ref{ass_G} (\romannumeral3), noting that $|b_1^--b^-_2|\leq |b_1- b_2|$ and $|b_1^+-b^+_2|\leq |b_1- b_2|$ for any $b_1,b_2\in\mathbb R,$ the local Lipschitz property of $\mathcal G_{\mathcal H}$ can be verified. When $g$ satisfies \eqref{extension}, one can also obtain the local Lipschitz property of $\mathcal G_{\mathcal O}.$   \end{proof}

\begin{proof}[\textbf{Verification of Example \ref{ex_boundH}}] When $d=3,$ we have $\epsilon\in(0,\frac13).$ 
Recalling the definition of inner product $(\cdot,\cdot)_{\xi}$ in Hilbert space $\mathbb H_{\xi}$ and the definition of the graph individual noise (see subsections \ref{notation} and \ref{eq_assp}), we have   
 \begin{align*}
\mathcal A_0:&= -2(\nabla_G(\xi-\eta),\nabla_{\mathcal W}\mathfrak b(\xi))_{\xi}+\mathcal O_{\xi}(\nabla_{\mathcal W}\mathfrak b(\xi))\\
&=(\xi_3-\epsilon)e^{\top}_{1,2}(2\xi-2\eta+\log\xi)(e^{\top}_{1,2}\xi)\omega_{1,2}g_{1,2}(\xi) +(\xi_2-\epsilon)e^{\top}_{1,3}(2\xi-2\eta+\log\xi)(e^{\top}_{1,3}\xi)\omega_{1,3}g_{1,3}(\xi)  \\
&\quad +(\xi_1-\epsilon)e^{\top}_{2,3}(2\xi-2\eta+\log\xi)(e^{\top}_{2,3}\xi)\omega_{2,3}g_{2,3}(\xi).
\end{align*} 
To verify that Assumption \ref{ass_H2} holds, by the fact that $\|\nabla_G(\xi-\eta)-\nabla_{\mathcal W}\mathfrak b(\xi)\|^2_{\xi}-\mathcal O_{\xi}(\nabla_G(\xi-\eta)-\nabla_{\mathcal W}\mathfrak b(\xi))\ge  \|\nabla_G(\xi-\eta)\|^2_{\xi}-\mathcal O_{\xi}(\nabla_G(\xi-\eta))+\mathcal A_0$
 it suffices to show that $\mathcal A_0$ is non-negative when $\xi$ approaches $\partial\mathcal P_{\epsilon}(G)$. 
We take the case $\xi_2\to\epsilon$ as an example. In this case, $\mathcal A_0$ converges to 
\begin{align*}
\mathcal A_1:&= \omega_{1,2}g(\xi_1,\epsilon)(\xi_3-\epsilon)(\xi_1-\epsilon)\Big(\log\frac{\xi_1}{\epsilon}+2(\xi_1-\eta_1)-2(\epsilon-\eta_2)\Big)\\
&\quad +\omega_{2,3}g(\xi_3,\epsilon)(\xi_1-\epsilon)(\epsilon-\xi_3)\Big(\log\frac{\epsilon}{\xi_3}+2(\epsilon-\eta_2)-2(\xi_3-\eta_3)\Big). 
\end{align*} 
If $\xi_1\leq \xi_3,$ then it holds that $\xi_1\in[\epsilon,\frac{1-\epsilon}{2}],$  and that
\begin{align*}
&\mathcal A_1=\mathcal A_{1,1}+\mathcal A_{1,2}:=g(\xi_1,\epsilon)(\xi_1-\epsilon)(1-2\epsilon-\xi_1)\Big(6(\eta_2-\epsilon)+\log\frac{\xi_1(1-\epsilon-\xi_1)}{\epsilon^2}\Big)\\
& +(g(\xi_3,\epsilon)-g(\xi_1,\epsilon))(\xi_1-\epsilon)(1-2\epsilon-\xi_1)\Big(2(\eta_2-\eta_3)+2(1-3\epsilon)+2(\epsilon-\xi_1)+\log\frac{1-\epsilon-\xi_1}{\epsilon}\Big). 
\end{align*}
It is clear that $\mathcal A_{1,1}\ge 0$ for $\xi_1\in[\epsilon,\frac{1-2\epsilon}{2}]$ due to the non-negativity of $g$.  For the term $\mathcal A_{1,2},$ we have $g(\xi_3,\epsilon)-g(\xi_1,\epsilon)\ge 0$ due to the non-decreasing property of $g(\cdot,\epsilon)$, and 
$
2(\eta_2-\eta_3)+2(1-3\epsilon)\ge 0 
$ 
due to $\eta\in\mathcal P_{\epsilon}(G).$  Moreover,  $\log\frac{1-\epsilon-\xi_1}{\epsilon}\ge 2(\xi_1-\epsilon)$ when $1-\xi_1\ge \epsilon(1+ e^{2(\xi_1-\epsilon)})$, which can be garanteed when $\epsilon\leq\frac{1}{2(1+e)}.$ Similarly, if $\xi_1\ge \xi_3,$ we have that $\xi_3\in[\epsilon,\frac{1-\epsilon}{2}],$ and thus
\begin{align*}
\mathcal A_1\ge 0
\end{align*}
 when $\epsilon\leq \frac{1}{2(1+e)}$. 
\end{proof}

\subsection{Proofs of some lemmas, proposition, and theorem}
\begin{proof}[Proof of Lemma \ref{lemma1}]
By the definition of the bijective mapping $\Pi,$ we obtain  
\begin{align*}\nabla^{e_{j,k}}f(\xi)&=\lim_{h\to0}\frac1h(f(\xi+he_{j,k})-f(\xi))=\lim_{h\to0}\frac1h(\tilde f(\Pi(\xi+he_{j,k}))-\tilde f(\Pi(\xi)))\\
&=(\partial_{x_j}+\cdots+\partial_{x_{k-1}})\tilde f(\Pi(\xi)),
\end{align*}
due to that $\Pi(\xi+he_{j,k}))=\Pi(\xi)+h(m_1,\ldots,m_{d-1})$ with $m_l=1$ for $l\in\{j,\ldots,k-1\}$ and $m_l=0$ for $l\notin\{j,\ldots,k-1\}.$ 
The proof is finished. 
\end{proof}

\begin{proof}[Proof of Lemma \ref{lem1}]
For (\romannumeral1), the non-decreasing property of the mapping $\vec{\mathcal G}$ with respect to the argument $\widetilde U^n(s_{\vec i}+rh\vec m_{j,k}),j<k$ for $r=1$ or $r=-1$ is implied by the monotonicity of $\mathcal G$ in Assumption \ref{ass_G} (\romannumeral1). 
It suffices to prove the non-decreasing property of the mapping $\widetilde U^n_{\vec i}\mapsto \widetilde U^{n}_{\vec i}-\tau\widetilde{\mathcal F}(s_{\vec i})-\tau\widetilde{\mathcal G}\big({s_{\vec i}}, [D^+\widetilde U^n_{\vec i}],[D^-\widetilde U^n_{\vec i}]\big)$ for $\vec i\in\mathcal N^{\circ}$. This can be verified when $\frac{\tau}{h}$ satisfy \eqref{cond_CLF}, due to the fact that $\|[D^{\pm}\widetilde U^n_{\vec i}]\|_{l^2}\leq d\|[D^{\pm}\widetilde U^n_{\vec i}]\|_{\infty}\leq dR$ and the local Lipschitz property of $\mathcal G$ in Assumption \ref{ass_G} (\romannumeral3) with $R$ being replaced by $dR$.   

To prove (\romannumeral2), from \eqref{explicit} we  see that  $\vec{\mathcal G}(\widetilde U+\lambda)=\vec{\mathcal G}(\widetilde U)+\lambda,\,\lambda\in\mathbb R$. For $\widetilde U,\widetilde V\in\mathcal M_{0}$, we have $\widetilde U\leq \widetilde V+\lambda$ with $\lambda=\|(\widetilde U-\widetilde V)^+\|_{\infty}$ and $\widetilde V+\lambda\in\mathcal M_{0}.$ By the non-decreasing property of $\vec{\mathcal G}$ (see Definition \ref{def_monotone}), we have 
$\vec{\mathcal G}(\widetilde U)\leq \vec{\mathcal G}(\widetilde V+\lambda)=\vec{\mathcal G}(\widetilde V)+\lambda.$ Hence we have $\|(\vec{\mathcal G}(\widetilde U)-\vec{\mathcal G}(\widetilde V))^+\|_{\infty}\leq \|(\widetilde U-\widetilde V)^{+}\|_{\infty}$ and similarly, $\|(\vec{\mathcal G}(\widetilde U)-\vec{\mathcal G}(\widetilde V))^-\|_{\infty}=\|(\vec{\mathcal G}(\widetilde V)-\vec{\mathcal G}(\widetilde U))^+\|_{\infty}\leq \|(\widetilde V-\widetilde U)^{+}\|_{\infty}=\|(\widetilde U-\widetilde V)^{-}\|_{\infty}$. 
Thus we derive  $\|\vec{\mathcal G}(\widetilde U)-\vec{\mathcal G}(\widetilde V)\|_{\infty}\leq \|\widetilde U-\widetilde V\|_{\infty}$.

To prove (\romannumeral3), for each fixed multi-index $m_{j,k}$, we define the transition 	 $\tau_{\vec m_{j,k}}$ by $(\tau_{\vec m_{j,k}}\widetilde U)_{\vec i}:=\widetilde U_{\vec m_{j,k}+\vec i}$ for $\vec i\in\mathcal N^{\circ},$ then $\tau_{\vec m_{j,k}}\vec{\mathcal G}(\widetilde U)=\vec{\mathcal G}(\tau_{\vec m_{j,k}}\widetilde U)$. It follows from (\romannumeral2) that when $\tau_{\vec m_{j,k}}\vec{\mathcal G}(\widetilde U)\in \widetilde{\mathcal P}^h(G)\backslash \partial\widetilde{ \mathcal P}(G),$
\begin{align*}
\|D^+_{\vec m_{j,k}}\vec{\mathcal G}(\widetilde U)\|_{\infty}&=\|\tau_{\vec m_{j,k}}\vec{\mathcal G}(\widetilde U)-\vec{\mathcal G}(\widetilde U)\|_{\infty}
=\|\vec{\mathcal G}(\tau_{\vec m_{j,k}}\widetilde U)-\vec{\mathcal G}(\widetilde U)\|_{\infty}\\
&\leq \|\tau_{\vec m_{j,k}}\widetilde U-\widetilde U\|_{\infty}= \|D^+_{\vec m_{j,k}}\widetilde U\|_{\infty}\leq Rh. 
\end{align*}
When $\tau_{\vec m_{j,k}}\vec{\mathcal G}(\widetilde U)\in\widetilde{\mathcal P}^h(G)\cap\partial \widetilde{\mathcal P}(G),$ by the specified value on the boundary (see \eqref{boundary1} or \eqref{boundary}), 
$\|D^+_{\vec m_{j,k}}\vec{\mathcal G}(\widetilde U)\|_{\infty}\leq Rh$ still holds. Similar proof can lead to $\|D^-_{\vec m_{j,k}}\vec {\mathcal G}(\widetilde U)\|_{\infty}\leq Rh.$
 
For (\romannumeral4), using (\romannumeral2), we have 
\begin{align*}
\|\vec{\mathcal G}^{(n+j)}(\widetilde U)-\vec{\mathcal G}^{(n)}(\widetilde U)\|_{\infty}&\leq \|\vec{\mathcal G}^{(j)}(\widetilde U)-\widetilde U\|_{\infty}\leq \sum_{l=0}^{j-1}\|\vec{\mathcal G}^{(j-l)}(\widetilde U)-\vec{\mathcal G}^{(j-l-1)}(\widetilde U)\|_{\infty}\leq Kj\tau.
\end{align*}
The proof is finished. 
\end{proof}

\begin{proof}[Proof of \eqref{local_max}]
 For $r>0,$ define the set $A(r):=[\hat t-r,\hat t+r]\times (\mathcal P_{\epsilon}(G)\cap B(\hat\xi,r))$, where $B(\hat\xi,r)$ is the closed ball with center $\hat\xi$ and radius $r>0.$ Without loss of generality, we may assume that the maximum $(\hat t,\hat \xi)$ is strict, namely, there exists some small $r_0>0$ such that 
 \begin{align}\label{strict}
 u(t,\xi)-\varphi(t,\xi)<u(\hat t,\hat\xi)-\varphi(\hat t,\hat\xi),\quad \forall (t,\xi)\in A(r_0)\backslash(\hat t,\hat\xi),
 \end{align}
 where we let $\varphi(t,\xi)=\frac{\sigma t}{4T}-\frac{1}{2\delta^2}(|t-\hat s|^2+\|\xi-\hat\eta\|^2_{l^2}).$
Recall that $\Phi_{a,\delta}:=u-\varphi-\frac{a}{\mathfrak b}$ with the distance function $\mathfrak b$ satisying Assumption \ref{ass_H2}. For such $r_0>0,$ there exists local maximizer $(t_a,\xi_a)\in A(r_0)$ for $\Psi_{a,\delta}$ when $a$ is fixed, namely, $\Psi_{a,\delta}(t_a,\xi_a)=\sup_{(t,\xi)\in A(r_0)}\Psi_{a,\delta}(t,\xi).$ We aim to show that $\lim_{a\to 0}(t_a,\xi_a)=(\hat t,\hat\xi)$ as $a\to0.$ For every $r\in(0,r_0),$ the continuity of $u-\varphi$ and \eqref{strict} allow us to find $(t_r,\xi_r)\in A(r)\cap ((0,T]\times \mathcal P^{\circ}_{\epsilon}(G))$ such that $u(t,\xi)-\varphi(t,\xi)<u(t_r,\xi_r)-\varphi(t_r,\xi_r)$ for all $(t,\xi)\in A(r_0)\backslash A(r)^{\circ}$. Moreover, for every $r\in(0,r_0),$ there exists $a_r>0,$ such that $\Phi_{a,\delta}(t,\xi)<\Phi_{a,\delta}(t_r,\xi_r)$ for all $(t,\xi)\in A(r_0)\backslash A(r)^{\circ}$ when $a<a_r,$ which implies $(t_a,\xi_a)\in A(r)$ for $a<a_r.$ As a result, we obtain $\lim_{a\to 0}(t_a,\xi_a)=(\hat t,\hat\xi)$ as $a\to0.$  
 \end{proof}

Then we show existence and uniqueness as well as some \textit{a priori} estimates for the solution of the fully-implicit scheme \eqref{implicit1}.  Since numerical methods \eqref{implicit1} and \eqref{implicit} are equivalent based on a coordinate transformation, it suffices to prove the estimates for the solution of \eqref{implicit}. 
For each given $\tau\in(0,1)$ and $\widetilde V\in L^{\infty}(\mathbb R^{|\mathcal N^{\circ }|}),$ define the operator $\mathcal T(\tau,\widetilde V):\mathbb R^{|\mathcal N^{\circ }|}\to\mathbb R^{|\mathcal N^{\circ}|},\,\widetilde U\mapsto \mathcal T(\tau,\widetilde V)\widetilde U$ by 
\begin{align}\label{tildeT}
(\mathcal T(\tau,\widetilde V)\widetilde U)_{\vec i}&=\widetilde V_{\vec i}-\tau\widetilde{\mathcal F}(s_{\vec i})-\tau\widetilde{\mathcal G}\big({s_{\vec i}},[D^+\widetilde U_{\vec i}],[D^-\widetilde U_{\vec i}]\big),\quad \vec i\in\mathcal N^{\circ}.
\end{align}  
We first give the existence of the fixed point for the operator $\mathcal T(\tau,\widetilde V)$.  
\begin{lemma}\label{lem2}
Let Assumptions \ref{assumption_g}--\ref{ass_G} hold. For given $\tau\in(0,1)$ and $\widetilde V\in L^{\infty}(\mathbb R^{|\mathcal N^{\circ}|}),$ there exists a unique fixed point $\widetilde U\in L^{\infty}(\mathbb R^{|\mathcal N^{\circ}|})$ for the mapping $\mathcal T(\tau,\widetilde V)$, i.e., $\mathcal T(\tau,\widetilde V)\widetilde U=\widetilde U.$ Moreover, it holds that $\|\widetilde U\|_{\infty}\leq \|\tau \widetilde{\mathcal G}(\cdot,0,0)+\tau\widetilde{\mathcal F}-\widetilde V\|_{\infty}$. \end{lemma}
\begin{proof}
\textit{Existence.} 
 Let $r:=\max_{\vec i\in\mathcal N^{\circ}}|\tau\widetilde{\mathcal G}({s_{\vec i}},0)+\tau\widetilde{\mathcal F}(s_{\vec i})-\widetilde V_{\vec i}|.$ The operator $\mathcal T(\tau,\widetilde V)$ is continuous from $B_r:=\{\|\widetilde U\|_{\infty}\leq r\}$ to $\mathbb R^{|\mathcal N^{\circ}|}$ due to that $\mathcal G$ is locally Lipschitz continuous. 
When $\widetilde U\in\partial B_r,$ i.e., $\|\widetilde U\|_{\infty}=r>0,$ we have $\max_{\vec i\in\mathcal N^{\circ}}|\widetilde U_{\vec i}|=r>0.$ Hence, there exists $\vec i_0\in\mathcal N^{\circ}$ such that $\widetilde U_{\vec i_0}=r$ or $\widetilde U_{\vec i_0}=-r.$ 

\textit{Case 1: $\widetilde U_{\vec i_0}=r.$} In this case, we have $\widetilde U_{\vec i}\leq \widetilde U_{\vec i_0}$ for each $\vec i\in\mathcal N^{\circ}.$ 
By Assumption \ref{ass_G}, we derive that $\widetilde{\mathcal G}({s_{\vec i_0}},[D^{+}\widetilde U]_{\vec i_0},[D^{-}\widetilde U]_{\vec i_0})\ge \widetilde{\mathcal G}({s_{\vec i_0}},0,0).$ Hence, we obtain
\begin{align*}
(\mathcal T(\tau,\widetilde V)\widetilde U)_{\vec i_0}\leq \widetilde V_{\vec i_0}-\tau\widetilde {\mathcal F}(s_{\vec i_0})-\tau\widetilde{\mathcal G}({s_{\vec i_0}},0,0)\leq r=\widetilde U_{\vec i_0},
\end{align*} 
which leads to that $(\mathcal T(\tau,\widetilde V)\widetilde U)_{\vec i_0}\neq \lambda \widetilde U_{\vec i_0}$ for any $\lambda>1.$

\textit{Case 2: $\widetilde U_{\vec i_0}=-r.$} In this case, we have $\widetilde U_{\vec i}\ge \widetilde U_{\vec i_0}$ for any $\vec i\in\mathcal N^{\circ}.$ By virtue of the monotonicity of $\mathcal G,$ we can similarly derive that $\widetilde{\mathcal G}({s_{\vec i_0}},[D^{\pm}\widetilde U]_{\vec i_0})\leq \widetilde{\mathcal G}({s_{\vec i_0}},0),$ which yields  
\begin{align*}
(\mathcal T(\tau,\widetilde V)\widetilde U)_{\vec i_0}\ge \widetilde V_{\vec i_0}-\tau\widetilde{\mathcal G}({s_{\vec i_0}},0)-\tau\widetilde{\mathcal F}(s_{\vec i_0})\ge -r=\widetilde U_{\vec i_0}.
\end{align*}
This implies that $
(\mathcal T(\tau,\widetilde V)\widetilde U)_{\vec i_0}\neq \lambda\widetilde U_{\vec i_0}$ for any $\lambda>1.$ 

Combining \textit{Cases 1--2} gives that $\mathcal T(\tau,\widetilde V)\widetilde U\neq\lambda\widetilde U$ for any $\lambda>1,$ when $\widetilde U\in\partial B_r.$
According to the Leray--Schauder fixed point theorem (see e.g. \cite[Lemma 1]{Meanfield}), we obtain that there exists a fixed point $\widetilde U$ in $B_r,$ i.e., $\|\widetilde U\|_{\infty}\leq r.$ 

\textit{Uniqueness.}  Let $U$ and $\bar U$ satisfy, respectively,
\begin{align}\label{barU}
U_{\vec i}+\tau \widetilde{\mathcal G}({s_{\vec i}},[D^{\pm}U]_{\vec i})+\tau\widetilde{\mathcal F}_{\vec i}=W_{\vec i},\quad \bar U_{\vec i}+\tau \widetilde{\mathcal G}({s_{\vec i}},[D^{\pm}\bar U]_{\vec i})+\tau\widetilde{\mathcal F}_{\vec i}=\bar W_{\vec i}.
\end{align}
We aim to show that $0<\sup\limits_{\vec i\in\mathcal N^{\circ}}(U_{\vec i}-\bar U_{\vec i})^+\leq \|W-\bar W\|_{\infty}.$ Define $\Gamma_{\vec i}:=(U_{\vec i}-\bar U_{\vec i})^+$ for $\vec i\in\mathcal N^{\circ}.$ By \eqref{barU}, when $W\neq\bar W,$ we have $U\neq \bar U,$ due to that $U= \bar U$ implies $W=\bar W$. Hence we assume that $\|\Gamma\|_{\infty}=\sup_{\vec i\in\mathcal N^{\circ}}\Gamma_{\vec i}>0$ without loss of generality. Since $\Gamma$ is bounded, for any $\delta>0,$ there exists $\vec i_1\in\mathcal N^{\circ}$ such that $\Gamma_{\vec i_1}>\|\Gamma\|_{\infty}-\delta$.  Take $\zeta\in\mathcal C_0^2(\widetilde{\mathcal P}_{\epsilon}(G))$ such that $0\leq \zeta\leq 1,\;\|D\zeta\|_{L^{\infty}}\leq 1,\;\zeta(s_{\vec i_1})=1$ and define $\widehat \Gamma:=\Gamma+2\delta \zeta.$ Then we have $\widehat \Gamma_{\vec i_1}=\Gamma_{\vec i_1}+2\delta>\|\Gamma\|_{\infty}+\delta.$ Besides, there is an index $\vec i_0\in\mathcal N^{\circ}$ such that \begin{align}\label{maxPhi}\widehat\Gamma_{\vec i_0}\ge \widehat\Gamma_{\vec i}\quad \forall \,\vec i\in\mathcal N^{\circ}.
\end{align} This leads to that $\widehat\Gamma_{\vec i_0}\ge \widehat\Gamma_{\vec i_1}\ge \|\Gamma\|_{\infty}+\delta$. Hence, we have 
\begin{align*}
&\|\Gamma\|_{\infty}\leq \widehat{\Gamma}_{\vec i_0}-\delta\leq \Gamma_{\vec i_0}+\delta=U_{\vec i_0}-\bar {U}_{\vec i_0}+\delta. 
\end{align*}
By the definition of $\Gamma$ and expressions in \eqref{barU}, we derive that 
\begin{align}\label{supdiffU}
\sup_{\vec i\in\mathcal N^{\circ}}(U_{\vec i}-\bar U_{\vec i})^+\leq \|W_{\vec i}-\bar W_{\vec i}\|_{\infty}+\tau\Big(\widetilde{\mathcal G}({s_{\vec i_0}},[D^{\pm}\bar U]_{\vec i_0})-\widetilde{\mathcal G}({s_{\vec i_0}},[D^{\pm}U]_{\vec i_0})\Big)+\delta.
\end{align} 
It follows from \eqref{maxPhi} that $U_{\vec i}-\bar U_{\vec i}+2\delta\zeta(s_{\vec i})\leq (U_{\vec i}-\bar U_{\vec i})^++2\delta\zeta(s_{\vec i})\leq U_{\vec i_0}-\bar U_{\vec i_0}+2\delta\zeta(s_{\vec i_0})$ for any $\vec i\in\mathcal N^{\circ},$ and thus it holds that 
\begin{align*}
&\quad D^+_{\vec m_{j,k}}U_{\vec i_0}=U_{\vec i_0+\vec m_{j,k}}-U_{\vec i_0}\leq \bar U_{\vec i_0+\vec m_{j,k}}-\bar U_{\vec i_0}+2\delta(\zeta(s_{\vec i_0})-\zeta(s_{\vec i_0}+h\vec m_{j,k})),\\
&\quad D^-_{\vec m_{j,k}}U_{\vec i_0}=U_{\vec i_0}-U_{\vec i_0-\vec m_{j,k}}\ge \bar U_{\vec i_0}-\bar U_{\vec i_0-\vec m_{j,k}}+2\delta(\zeta(s_{\vec i_0}-h\vec m_{j,k})-\zeta(s_{\vec i_0})).
\end{align*}
This, together with the monotonicity of $\mathcal G$ in Assumption \ref{ass_G} (\romannumeral1) gives that 
\begin{align*}
\tau\widetilde{\mathcal G}({s_{\vec i_0}},[D^{\pm}\bar U]_{\vec i_0})&\leq \tau\widetilde{\mathcal G}\Big({s_{\vec i_0}},\frac {\sqrt{\omega_{j,k}}}{h}\Big( D^+_{\vec m_{j,k}}U_{\vec i_0} -2\delta \big(\zeta(s_{\vec i_0})-\zeta(s_{\vec i_0}+h\vec m_{j,k})\big)\Big),\\
&\quad \frac{\sqrt{\omega_{j,k}}}{h}\Big(D^-_{\vec m_{j,k}}U_{\vec i_0}-2\delta \big(\zeta(s_{\vec i_0}-h\vec m_{j,k})-\zeta(s_{\vec i_0})\big)\Big)\Big). 
\end{align*}
Letting $\delta\to0,$ and taking \eqref{supdiffU} into consideration, we obtain $\sup_{\vec i\in\mathcal N^{\circ}}(U_{\vec i}-\bar U_{\vec i})^+\leq \|W_{\vec i}-\bar W_{\vec i}\|_{\infty}.$ Similarly, one can derive $\sup_{\vec i\in\mathcal N^{\circ}}(U_{\vec i}-\bar U_{\vec i})^-\leq \|W_{\vec i}-\bar W_{\vec i}\|_{\infty},$ which implies 
\begin{align}\label{roman1}
\|U_{\vec i}-\bar U_{\vec i}\|_{\infty}\leq \|W_{\vec i}-\bar W_{\vec i}\|_{\infty}.
\end{align}  This completes the proof. 
\end{proof}
Define the operator $\widetilde {\mathcal T}: \mathbb R^{|\mathcal N^{\circ}|}\to \mathbb R^{|\mathcal N^{\circ}|}, W\mapsto U$ by $ U=\mathcal T(\tau, W) U,$ which represents the onestep mapping for the implicit method. 
Recall that $\mathcal M_{0}$ is given by \eqref{setM}. 
\begin{lemma}\label{lem3}
Let Assumptions \ref{assumption_g}--\ref{ass_G} hold.  Then the following properties hold. 
\begin{itemize} 
\item[(\romannumeral1)] $\widetilde {\mathcal T}(\mathcal M_{0})\subset \mathcal M_{0}$;
\item[(\romannumeral2)] For $n,j\ge 0$ and $W\in\mathcal M_{0},$ we have $\|\widetilde {\mathcal T}^{(n+j)}(W)-\widetilde {\mathcal T}^{(n)}(W)\|_{\infty}\leq Kj\tau$, where $K$ is given in Lemma \ref{lem1} (\romannumeral4).
\end{itemize} 
\end{lemma} 
\begin{proof} 
(\romannumeral1)  
Recall that $\tau_{\vec m_{j,k}}$ is the transition mapping defined in Lemma \ref{lem1}. When $\tau_{\vec m_{j,k}}\widetilde {\mathcal T}\widetilde V\in \widetilde{\mathcal P}^h(G)\backslash \partial\widetilde{ \mathcal P}(G) ,$  we see that 
$ \|D^+_{\vec m_{j,k}}\widetilde {\mathcal T}\widetilde V\|_{\infty}=\|\tau_{\vec m_{j,k}}\widetilde {\mathcal T}\widetilde V-\widetilde {\mathcal T}\widetilde V\|_{\infty}
=\|\widetilde {\mathcal T} \tau_{\vec m_{j,k}} \widetilde V-\widetilde {\mathcal T}\widetilde V\|_{\infty}\leq \|D^+_{\vec m_{j,k}}\widetilde V\|_{\infty}\leq Rh, $
where we use \eqref{roman1}. When $\tau_{\vec m_{j,k}}\widetilde {\mathcal T}\widetilde V\in \widetilde{\mathcal P}^h(G)\cap \partial\widetilde{ \mathcal P}(G),$ by \eqref{boundary1} or \eqref{boundary}, we have that 
$\|D^+_{\vec m_{j,k}}\widetilde {\mathcal T}\widetilde V\|_{\infty}\leq Rh$ still holds.  
The proof of (\romannumeral2) is similar to that of Lemma \ref{lem1} (\romannumeral4) by considering $\widetilde {\mathcal T}$ instead of $\vec{\mathcal G}$. 
\end{proof}

Below we show the existence, uniqueness and boundedness of the numerical solution for the implicit method.

\begin{proposition}\label{prop_EU}
Let Assumptions \ref{assumption_g}--\ref{ass_G} hold. Then there exists a unique numerical solution $\{\widetilde U^n\}_{0\leq n\leq N_T}$ for the implicit method  \eqref{implicit}. 
Moreover, the numerical solution satisfies that 
$\|\widetilde U^n\|_{\infty}\leq C(t_n+1)$ with some $C>0.$ 
\end{proposition}
\begin{proof}
The existence and uniqueness can be derived from Lemma \ref{lem2}. 
Recalling the definition of the operator $\widetilde T,$ by Lemma \ref{lem2}, we have $
\|\widetilde T \widetilde V\|_{\infty}\leq \|\tau\widetilde{\mathcal G}(\cdot,0,0)+\tau\widetilde{\mathcal F}-\widetilde V\|_{\infty},
$  which yields that $
\|\widetilde T(0)\|_{\infty}\leq \tau\|\widetilde{\mathcal G}(\cdot,0,0)+\widetilde{\mathcal F}\|_{\infty}.
$
Iterating twice gives 
\begin{align*}
\|\widetilde T^{(2)}(0)\|_{\infty}\leq \tau\|\widetilde{\mathcal G}(\cdot,0,0)+\widetilde{\mathcal F}\|_{\infty}+\|\widetilde T(0)\|_{\infty}\leq2\tau  \|\widetilde{\mathcal G}(\cdot,0,0)+\widetilde{\mathcal F}\|_{\infty}. 
\end{align*}
By iteration, we obtain
$\|\widetilde T^{(n)}(0)\|_{\infty}\leq n\tau\|\widetilde{\mathcal G}(\cdot,0,0)+\widetilde{\mathcal F}\|_{\infty}.$ 
It follows from Lemma \ref{lem3} that 
\begin{align*}
\|\widetilde T^{(n)}(\widetilde U^0)\|_{\infty}&\leq \|\widetilde T^{(n)}(\widetilde U^0)-\widetilde T^{(n)}(0)\|_{\infty}+\|\widetilde T^{(n)}(0) \|_{\infty}\leq \|\widetilde U^0\|_{\infty}+
Ct_n.
\end{align*}
The proof is finished. 
\end{proof}

\begin{proof}[Proof of Theorem \ref{thm_implicit}]
The proof is similar to that of Theorem \ref{thm1}, the main difference lies in the estimate of \eqref{U_eq1}. In the right-hand side of \eqref{U_eq1}, the arguments in $\mathcal G$ are replaced by the differences of numerical solution in the $n_0$th time step for the implicit method. Thus we omit further details. 
\end{proof}

\section{Some definitions}\label{sec_app}
We introduce the definition of the viscosity solution for HJE \eqref{HJeq} on the Wasserstein space on graphs; see \cite{MCC}. Define  $\mathrm{USC}([0,T)\times\mathcal P^{\circ}(G))$ (resp. $\mathrm{LSC}([0,T)\times\mathcal P^{\circ}(G))$)  the space of all upper (resp. lower) semi-continuous functions on $[0,T)\times\mathcal P^{\circ}(G).$ Denote by $\mathcal C^1((0,T)\times \mathcal P^{\circ}(G),l^2)$ the set of real-valued functions $\varphi:(t,\xi)\mapsto \mathbb \varphi(t,\xi)$ on $(0,T)\times \mathcal P^{\circ}(G)$ which have a continuous Fr\'echet derivative w.r.t. $\xi\in\mathcal P^{\circ}(G)$ and a continuous derivative w.r.t.  $t\in(0,T)$.

\begin{definition}\label{def_viscosity} \cite[Definition 4.2]{MCC}
(\romannumeral1) A function $u\in\mathrm{USC}([0,T)\times\mathcal P^{\circ}(G))$ is a viscosity subsolution to \eqref{HJeq} if $u(0,\cdot)\leq \mathcal U_0$ and for every $(t_0,\xi_0)\in(0,T)\times\mathcal P^{\circ}(G)$ and every $\varphi\in\mathcal C^1((0,T)\times\mathcal P^{\circ}(G),l^2)$ such that $u-\varphi$ has a local maximum at $(t_0,\xi_0),$ we have 
\begin{align*}
\partial_t\varphi(t_0,\xi_0)+\mathcal H(\xi_0,\nabla_{\mathcal W}\varphi(t_0,\xi_0))+\mathcal F(\xi_0)\leq \mathcal O_{\xi_0}(\nabla_{\mathcal W}\varphi(t_0,\xi_0)).
\end{align*}

(\romannumeral2) A function $u\in\mathrm{LSC}([0,T)\times\mathcal P^{\circ}(G))$ is a viscosity supersolution to \eqref{HJeq} if $u(0,\cdot)\ge  \mathcal U_0$ and for every $(t_0,\xi_0)\in(0,T)\times\mathcal P^{\circ}(G)$ and every $\varphi\in\mathcal C^1((0,T)\times\mathcal P^{\circ}(G),l^2)$ such that $u-\varphi$ has a local minimum at $(t_0,\xi_0),$ we have 
\begin{align*}
\partial_t\varphi(t_0,\xi_0)+\mathcal H(\xi_0,\nabla_{\mathcal W}\varphi(t_0,\xi_0))+\mathcal F(\xi_0)\ge  \mathcal O_{\xi_0}(\nabla_{\mathcal W}\varphi(t_0,\xi_0)).
\end{align*}

(\romannumeral3) A function $u$ is a viscosity solution of \eqref{HJeq} if it is both a viscosity subsolution and a viscosity supersolution.
\end{definition} 

Below we introduce the definition of a monotone method; see e.g. \cite{CrandallLions}.

\begin{definition}\label{def_monotone}
A numerical method  $$ U^{n+1}_{\vec i}=\vec{\mathcal G}\Big( U^n(\xi^{s_{\vec i}}+rhe_{j,k}):r=0,1,\text{ or }-1,\text{ and }1\leq j<k\leq d\Big),\;\vec i\in\mathcal N^{\circ}$$
 is called \textit{monotone} on $[-R,R]$ if the mapping $\vec{\mathcal G}$ is a non-decreasing function of each argument as long as $\frac1h|D^+_{e_{j,k}} U^n_{\vec i}|,\,\frac1h|D^-_{e_{j,k}} U^n_{\vec i}|\leq R$ with some constant $R>0$. 
\end{definition}

\section{More numerical experiments}\label{sec_app2}
\subsection{Meshspaces} 
We plot the meshspaces under different coordinates, i.e., ${\mathcal P}^h_{\epsilon}(G)$ and $\widetilde {\mathcal P}^h_{\epsilon}(G)$ for $d=3$ in Figure \ref{mesh1} (A) and (B), respectively.  
\begin{figure}[h]
\centering

\subfloat[]{
\begin{minipage}[t]{0.4\linewidth}
\centering
\includegraphics[height=4.4cm,width=5.8cm]{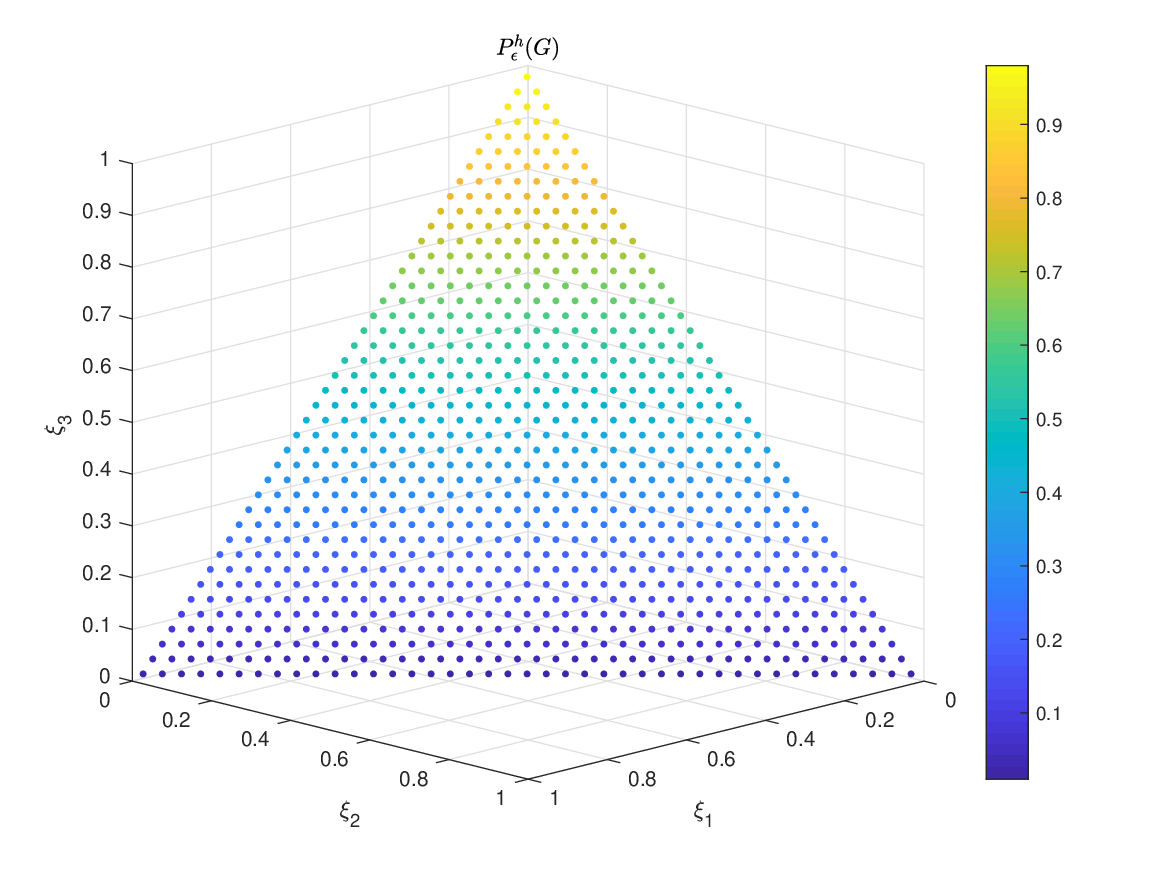}
\end{minipage}
}   \begin{tikzpicture}[overlay]
    \draw[->,  double, double distance=0.3pt] (0.5, 2.6) -- (1.5, 2.6); % 从第一张图到第二张图的箭头
    \node at (1, 2.9) {$\Pi$}; % 在箭头中间写"Mapping"
\end{tikzpicture}
    \hspace{1.3cm} 
\subfloat[]{
\begin{minipage}[t]{0.4\linewidth}
\centering
\includegraphics[height=4.4cm,width=5.8cm]{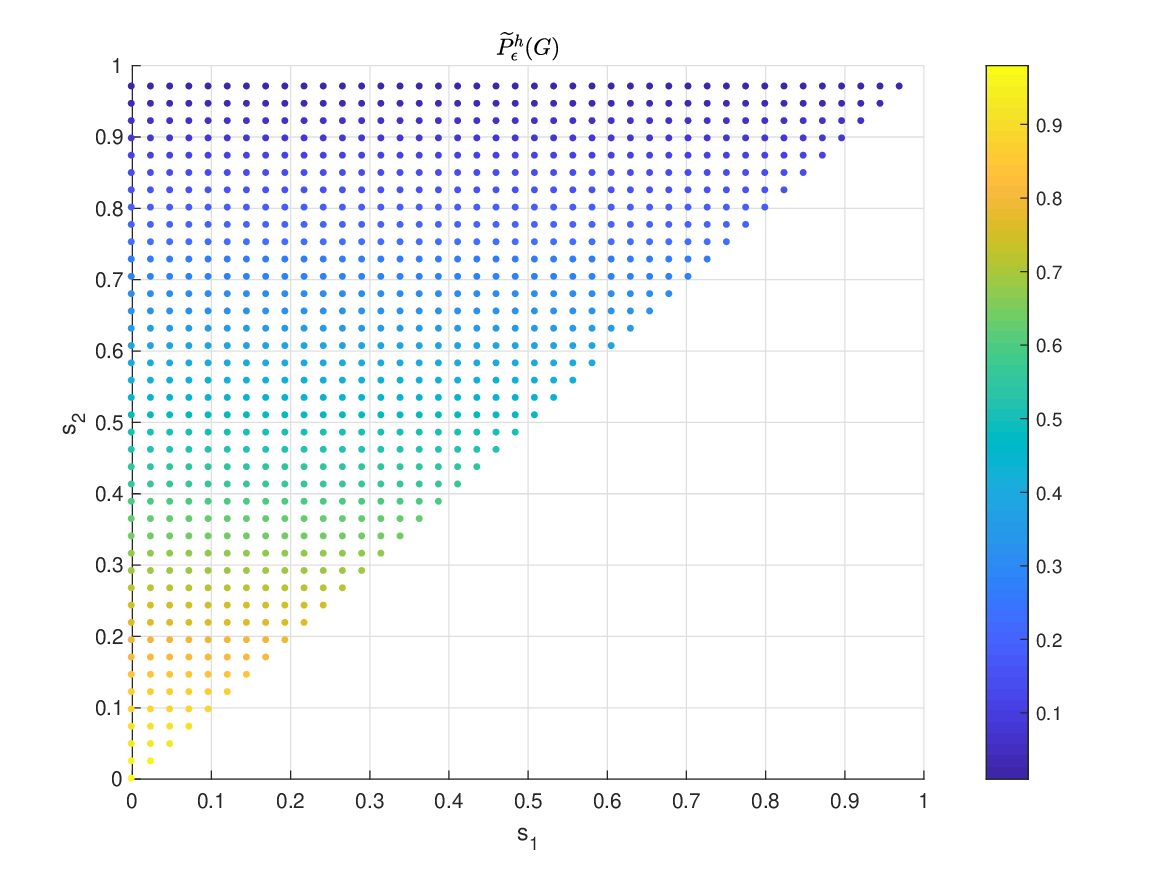}
\end{minipage}
}  

\centering
\caption{(A): ${\mathcal P}^h_{\epsilon}(G)$ and (B): $\widetilde{\mathcal P}_{\epsilon}^h(G)$,\,$\epsilon=0.01,h=0.025.$}
\label{mesh1}
\end{figure}
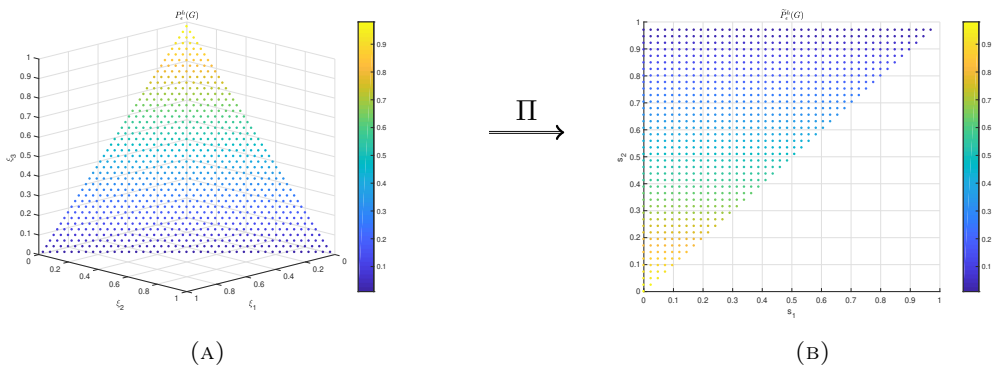

\subsection{Discussion on the boundary condition} 
In this part, we employ the monotone method \eqref{explicit1} 
to discretize \eqref{lap_eq}.    
We consider initial value $\mathcal U_0(\xi)=\min\{\xi_1,\xi_2,\xi_3\}\cos(\|\xi\|^2_{l^2}),$ which vanishes at the boundary.  
Recall the function $\mathcal I_{\theta}(\xi)=\sum_{i=1}^3\xi_i^{-\theta}$ given in Section \ref{sec_6}. 
 In Figure \ref{pic2}, we plot the numerical solution for the Hamiltonian $\mathcal H(\xi,p)=\mathcal I^{-2}_{0.1}(\xi)\|p\|^2_{\xi}$ subject to two kinds of boundary conditions:
Zero Dirichlet boundary condition 
in (A), and  
constant extrapolation boundary condition in 
(B).  
As seen in subfigure (A), the numerical solution is poorly simulated at the boundaries due to the formation of boundary layers, thus indicating that the zero Dirichlet boundary condition is not suitable for the considered equation \eqref{lap_eq} when $\lambda_1\neq0$. By contrast, subfigure (B) demonstrates that the constant extrapolation method effectively eliminates boundary layers and produces a well-behaved numerical solution. The solution with the constant extrapolation method deviates from the zero boundary and concentrates in the interior region, consistent with theoretical predictions (see Section \ref{sec_3}). 

\begin{figure}[h]
\centering

\subfloat[$\lambda_1=0.1,T=1$, Dirichlet]{
\begin{minipage}[t]{0.35\linewidth}
\centering
\includegraphics[height=4.4cm,width=5.8cm]{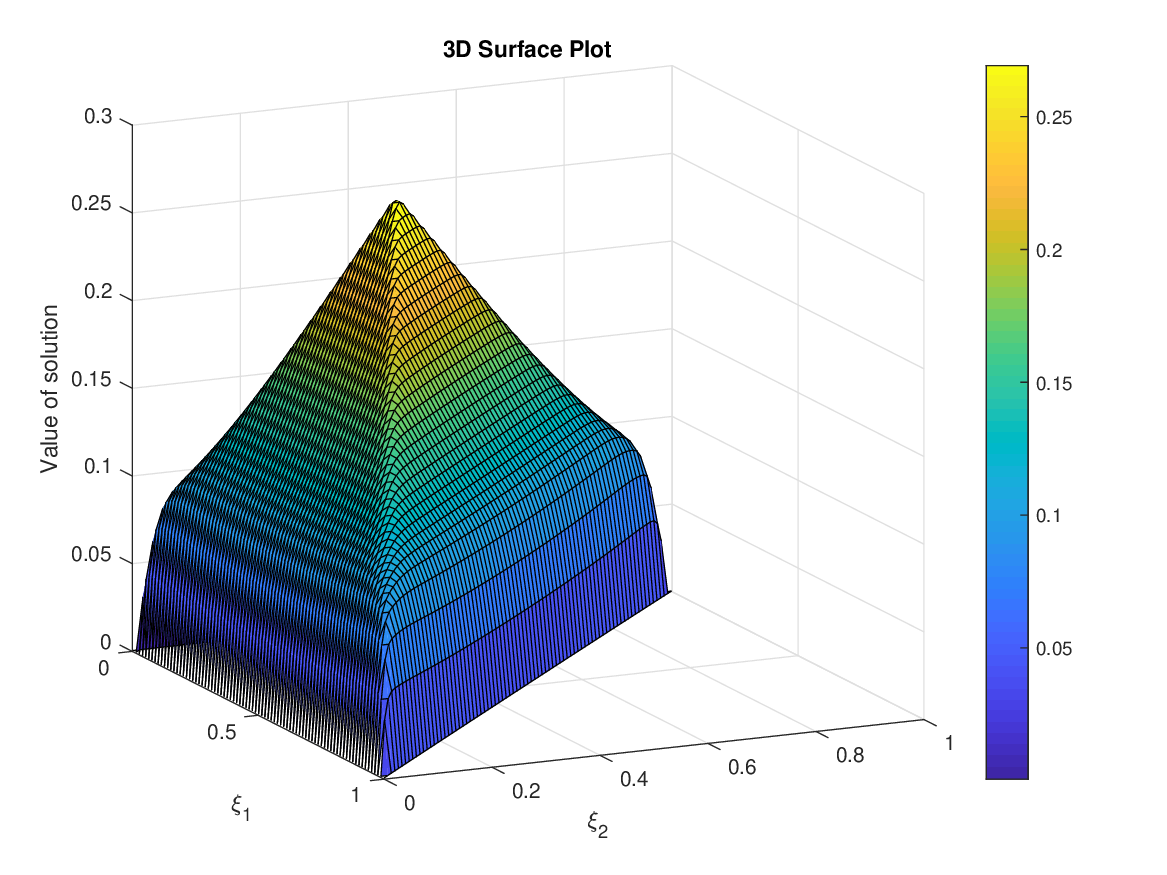}
\end{minipage}
}   
\subfloat[$\lambda_1=0.1,T=1$, Extrapolation]{
\begin{minipage}[t]{0.35\linewidth}
\centering
\includegraphics[height=4.4cm,width=5.8cm]{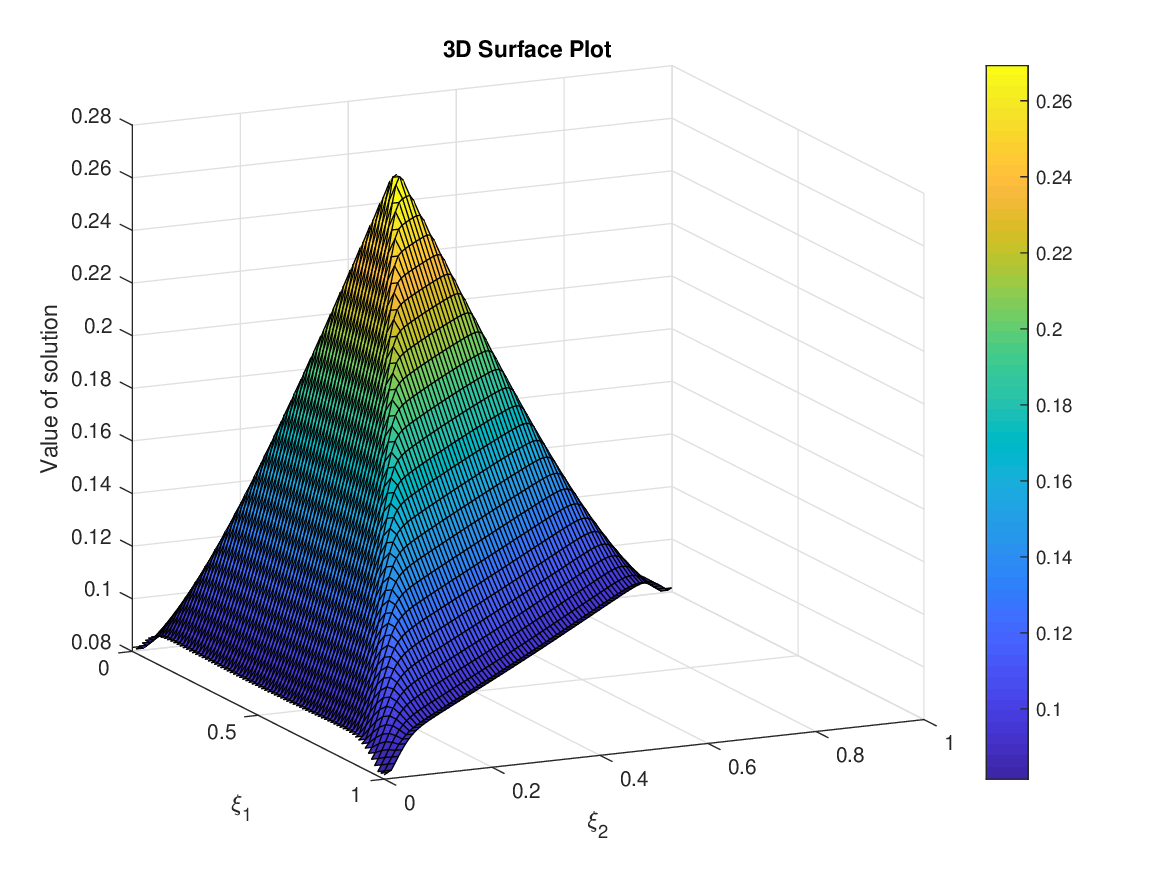}
\end{minipage}
} 
\centering
\caption{Numerical solution for $\mathcal H(\xi,p)=\mathcal I_{0.1}^{-2}(\xi)\|p\|_{\xi}^{2},\; \mathcal U_0(\xi)=\min\{\xi_1,\xi_2,\xi_3\}\cos(\|\xi\|^2_{l^2}), \tau/h=0.1,h=0.1\times 2^{-3},\epsilon=0.001$.} \label{pic2}
\end{figure}

\subsection{Effects of graph individual noise and Hamiltonian in solution dynamics}
We first investigate the impact of the graph individual noise term $\lambda_1\Delta_{\mathrm{ind}} u$ in \eqref{lap_eq} with the initial value $\mathcal U_0(\xi)=-\min\{\xi_1,\xi_2,\xi_3\}\cos(\sum_{i=1}^3\xi^2_i)$ and a fixed Hamiltonian $\mathcal H(\xi,p)=\mathcal I_1^{-2}(\xi)\|p\|^2_{\xi}.$
From Figure \ref{pic4} (B)(C)(D), we observe that increasing the noise intensity $\lambda_1$ induces pronounced deviations of the numerical solutions from the initial state (see subfigure (A)) near the boundaries, meanwhile the global minimum value stays invariant. This potentially indicates that the graph individual noise term has the effect of letting the solution keep away from boundary values and approach the interior values, with amplification proportional to noise intensity 
$\lambda_1$. From the contour plots (E)(F), we see that the graph individual noise term has a regularizing effect on smoothing the solution. The smoothing ability correlates directly with $\lambda_1$, evidenced by the transition from sharp features at small $\lambda_1$ to homogenized profiles at larger values. 

\begin{figure}[h]
\centering

\subfloat[$T=0$]{
\begin{minipage}[t]{0.35\linewidth}
\centering
\includegraphics[height=4.4cm,width=5.8cm]{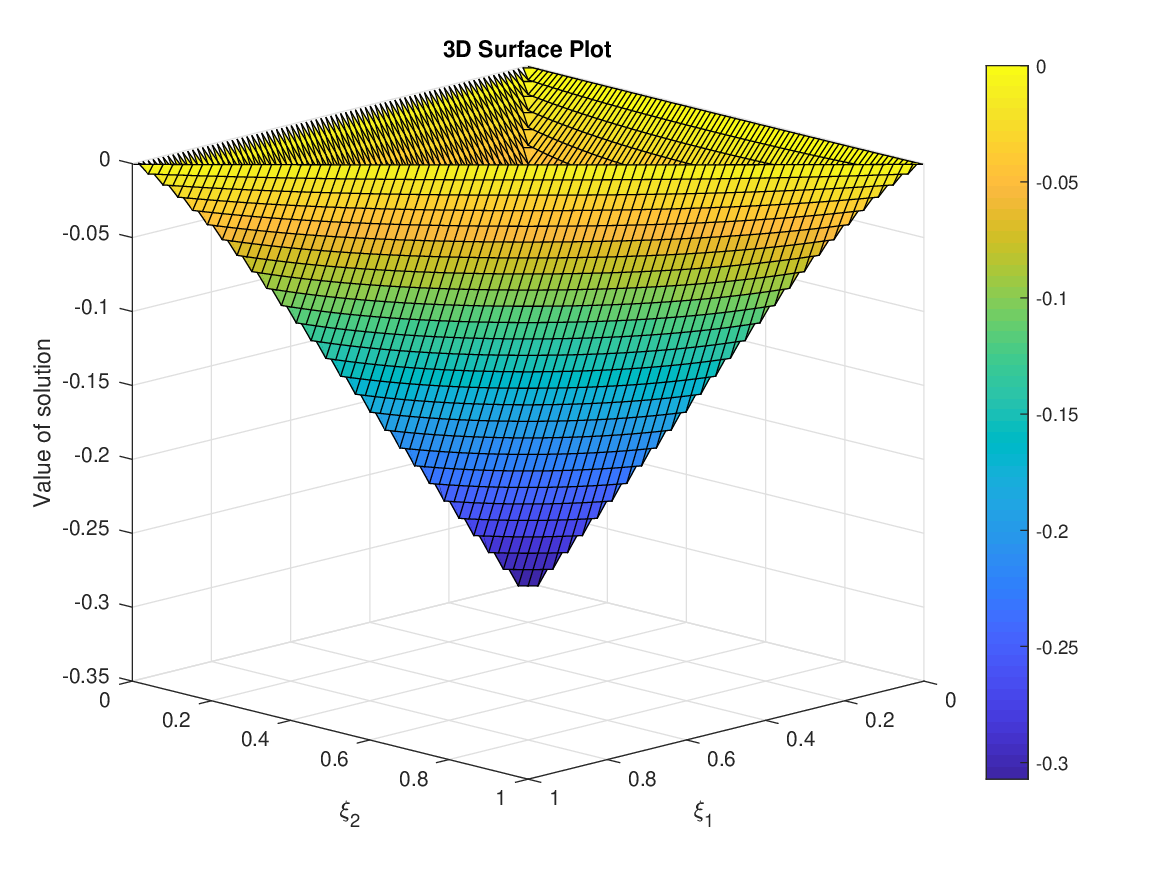}
\end{minipage}
} 
\subfloat[$\mathfrak a=\mathcal I^{-2}_1,\lambda_1=0.01$]{
\begin{minipage}[t]{0.35\linewidth}
\centering
\includegraphics[height=4.4cm,width=5.8cm]{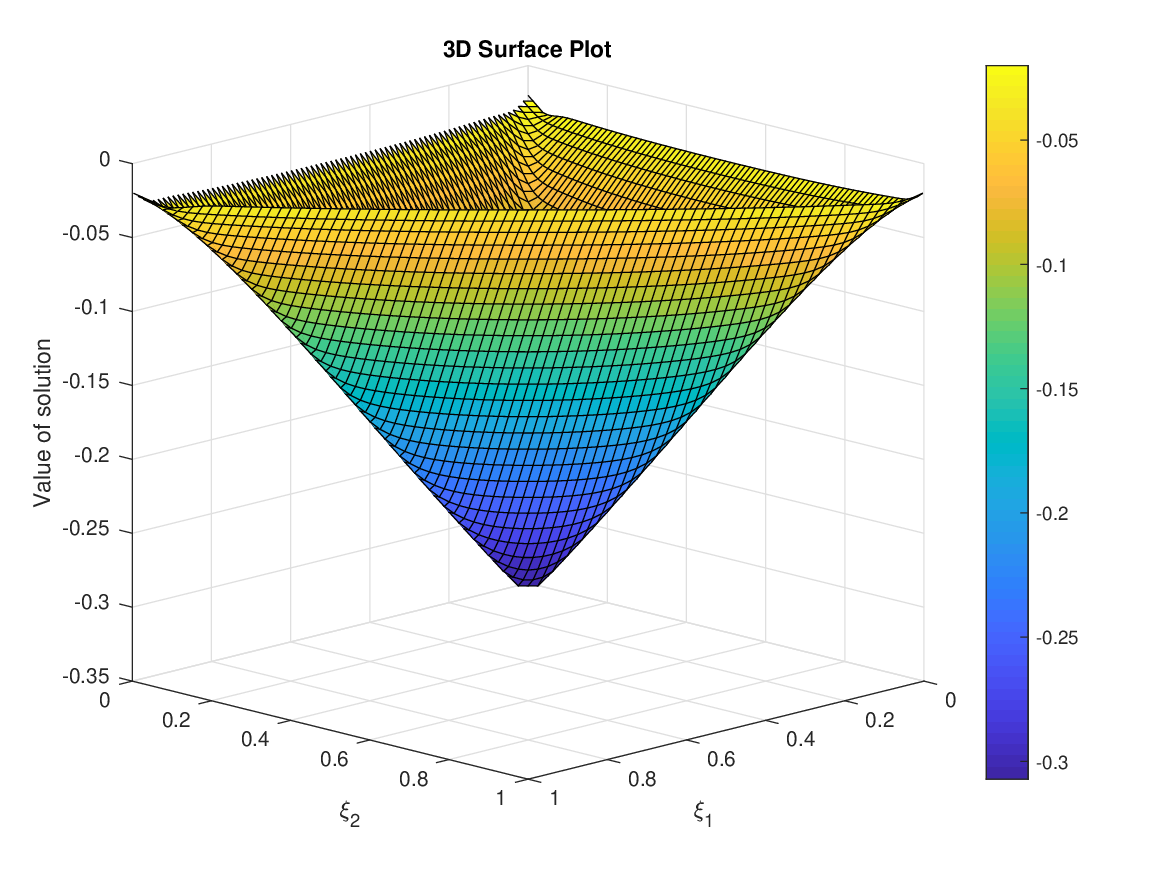}
\end{minipage}
}

\subfloat[$\mathfrak a=\mathcal I^{-2}_1,\lambda_1=0.1$]{
\begin{minipage}[t]{0.35\linewidth}
\centering
\includegraphics[height=4.4cm,width=5.8cm]{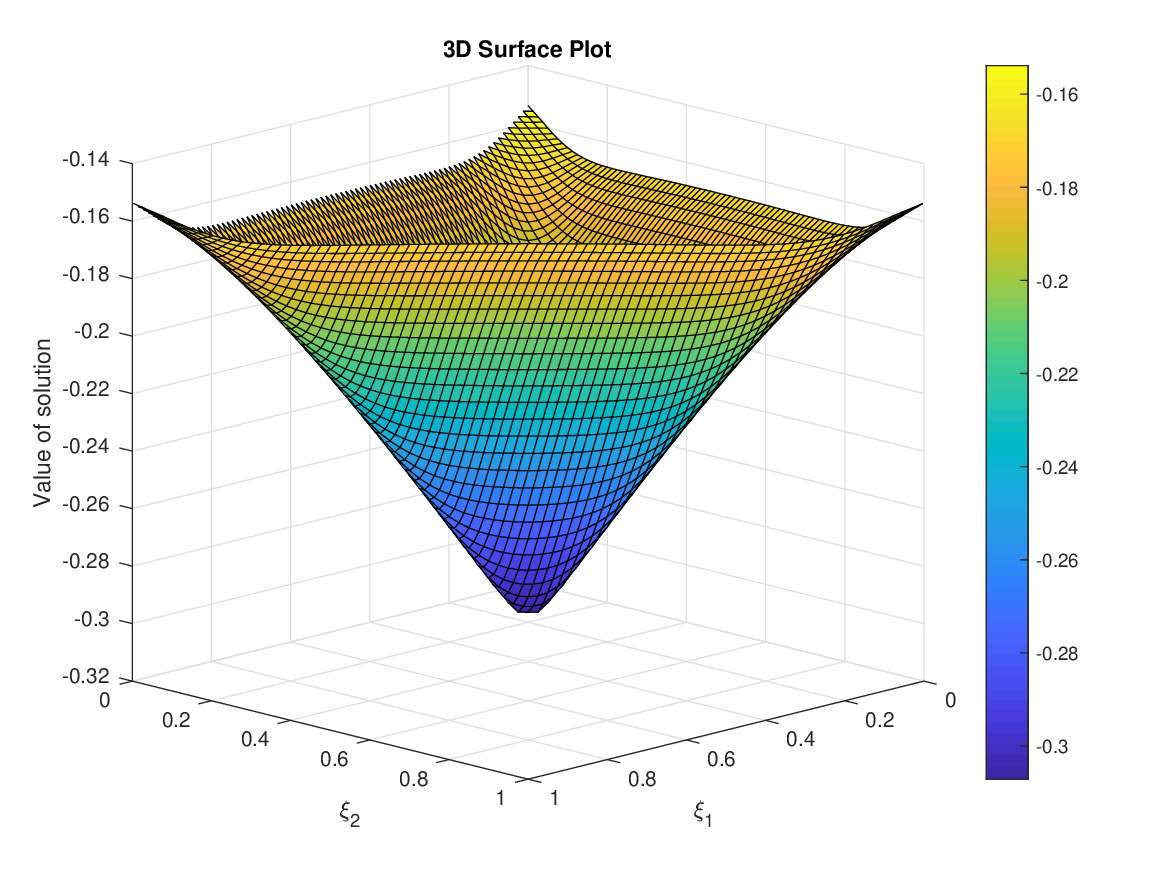}
\end{minipage}
} 
\subfloat[$\mathfrak a=\mathcal I^{-2}_1,\lambda_1=1$]{
\begin{minipage}[t]{0.35\linewidth}
\centering
\includegraphics[height=4.4cm,width=5.8cm]{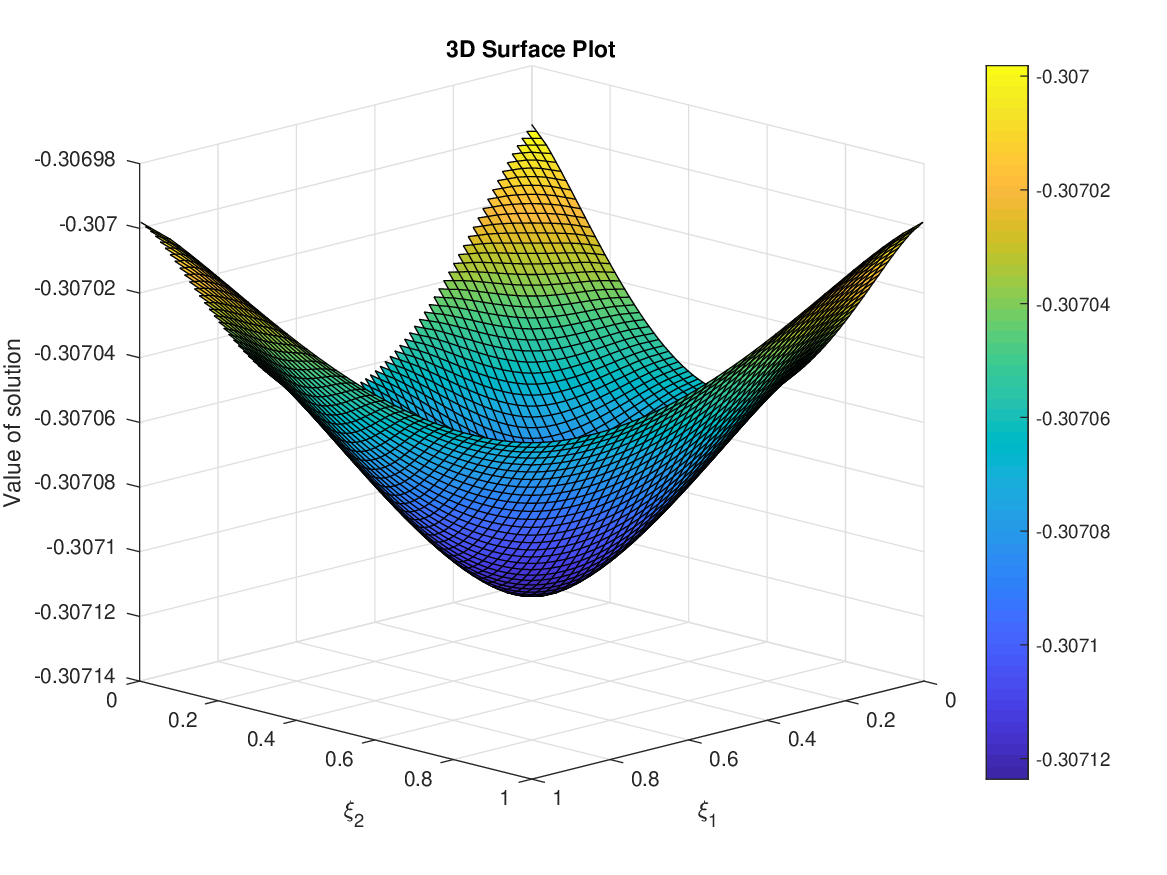}
\end{minipage}
}

\subfloat[$\mathfrak a=\mathcal I^{-2}_1,\lambda_1=0.1$]{
\begin{minipage}[t]{0.35\linewidth}
\centering
\includegraphics[height=4.4cm,width=5.8cm]{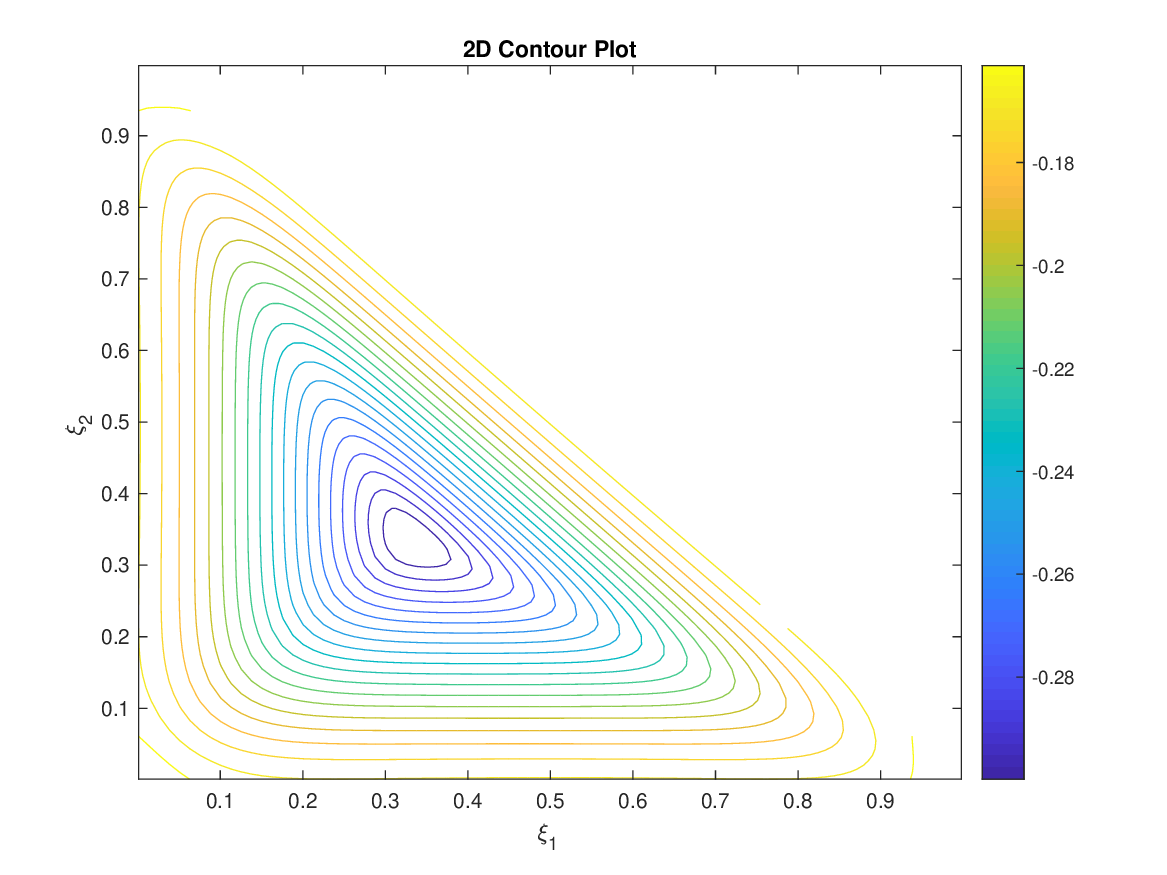}
\end{minipage}
}
\subfloat[$\mathfrak a=\mathcal I^{-2}_1,\lambda_1=1$]{
\begin{minipage}[t]{0.35\linewidth}
\centering
\includegraphics[height=4.4cm,width=5.8cm]{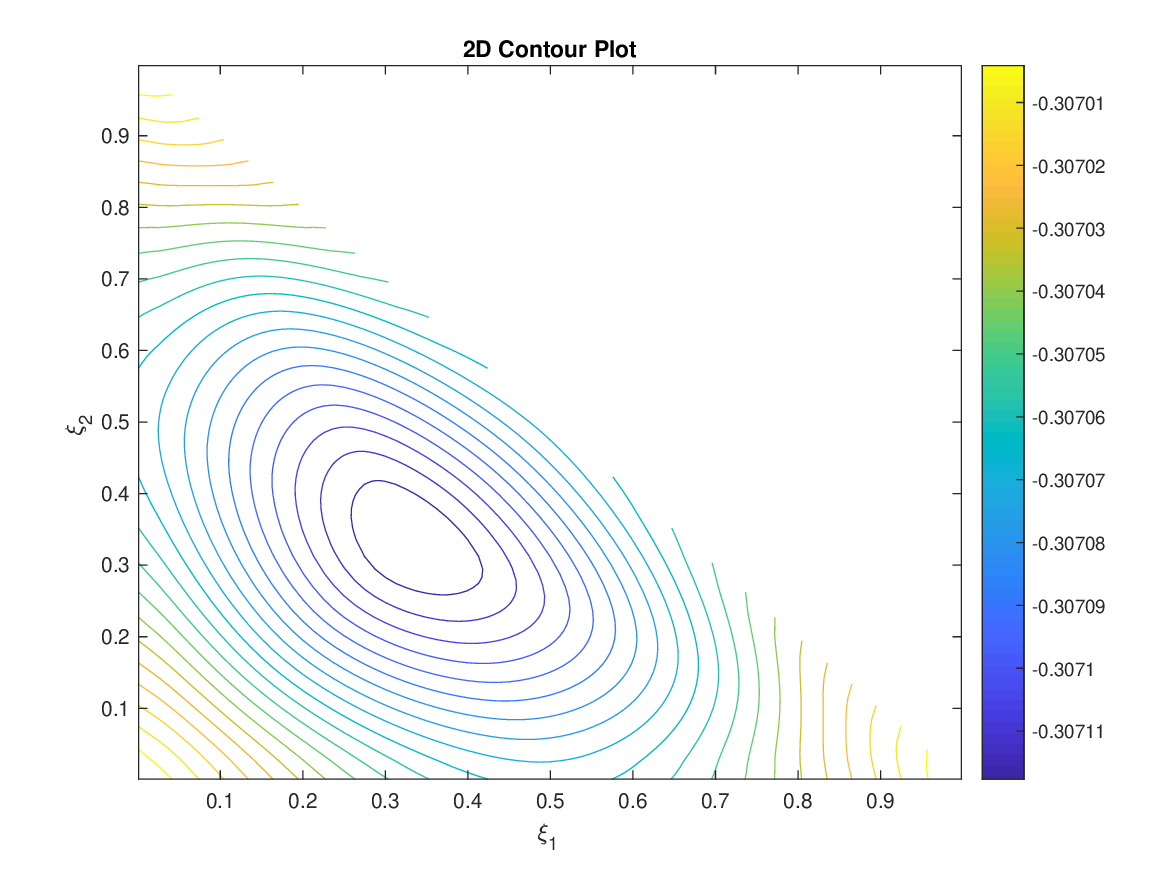}
\end{minipage}
}

\centering
\caption{Numerical solution for $\mathcal H(\xi,p)= \mathcal I_{1}^{-2}(\xi)\|p\|^2_{\xi},\;\mathcal U_0(\xi)=-\min(\xi_i)\cos(\sum_{i=1}^3\xi^2_i),h=0.0125,\epsilon=0.001,\frac{\tau}{h}=0.05,T=2.$}\label{pic4}
\end{figure}

Next, we investigate the behavior of numerical solutions of \eqref{lap_eq} with $\lambda_1=0$, $\mathcal H(\xi,p)=\mathfrak a(\xi)\|p\|^2_{\xi},$ and  $\mathcal U_0(\xi)=\|\xi\|^2_{l^2},$ using different choices of the function $\mathfrak a$. The results are shown in Figure  \ref{pic5}. Subfigures (B)(C) illustrate the numerical solutions for $\mathfrak a(\xi)=(\sum_{i=1}^3\xi^{-\theta}_i)^{-2}$ with $\theta=0.5$ and $\theta=0.3$, respectively. 
As $\theta$ decreases, the numerical solution deviates increasingly from the initial value (subfigure (A)).  
This indicates that smaller values of $\theta$ exhibit a similar effect to an increase in noise intensity. However, regarding the scale of variation in the numerical solution, altering 
$\theta$ has a less pronounced impact compared to changes in noise intensity.
Subfigure (D) shows the numerical solution for the logarithmic coefficient function $\mathfrak a(\xi)=(\sum_{i=1}^3\log(\xi_i))^{-2}$, whose behavior is similar to the case with $\theta=0.3$ in subfigure (C).  
This may indicate the well-posedness of viscosity solutions for equations whose Hamiltonians have logarithmic coefficients.

\begin{figure}[h]
\centering

\subfloat[$T=0$]{
\begin{minipage}[t]{0.35\linewidth}
\centering
\includegraphics[height=4.4cm,width=5.8cm]{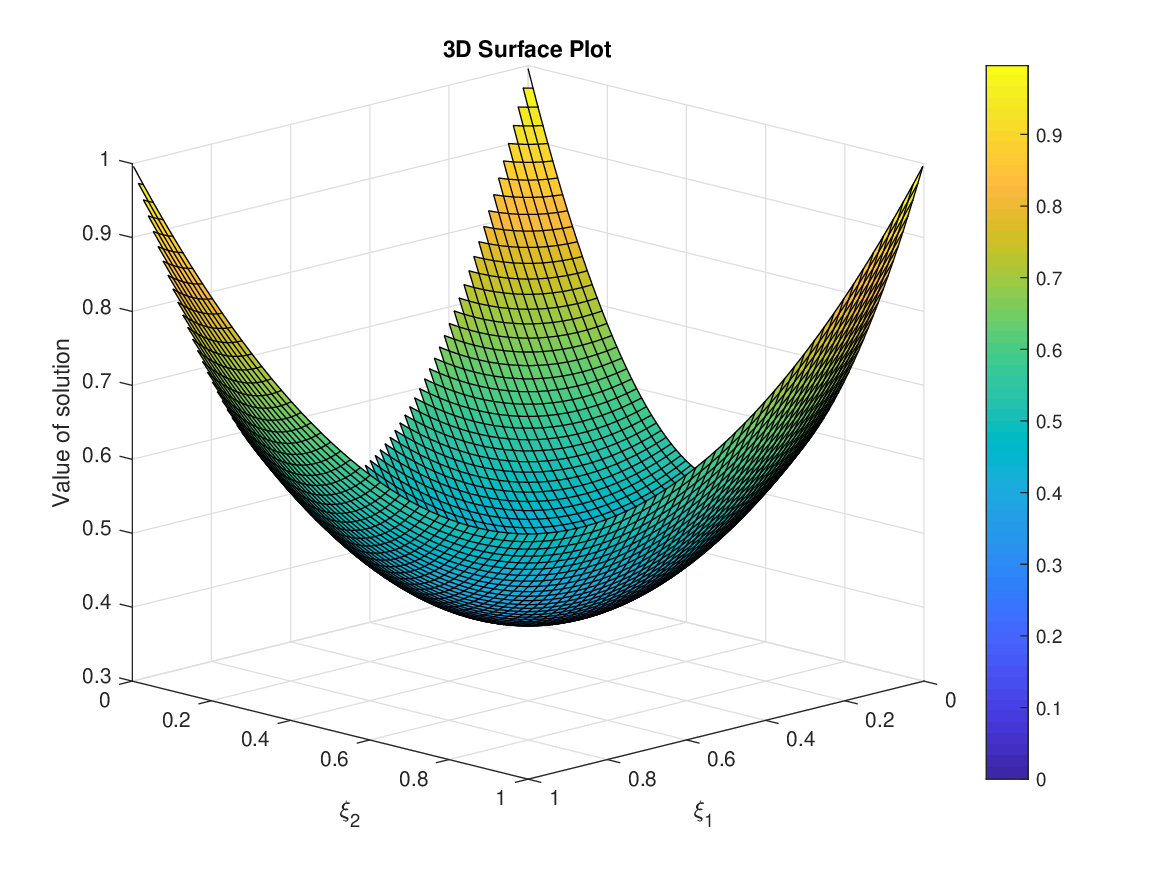}
\end{minipage}
}   
\subfloat[$\mathfrak a=(\sum_{i=1}^3\xi_i^{-\theta})^{-2},\theta=0.5$]{
\begin{minipage}[t]{0.35\linewidth}
\centering
\includegraphics[height=4.4cm,width=5.8cm]{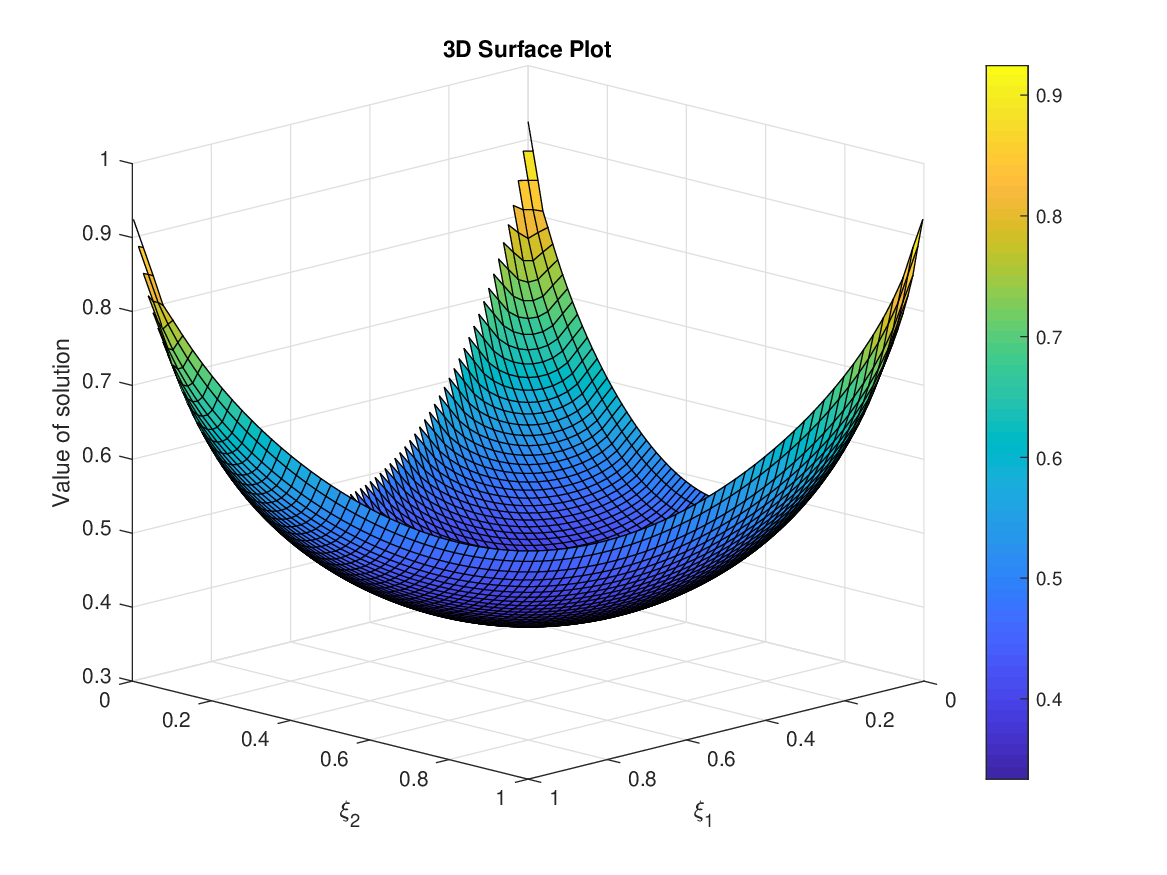}
\end{minipage}
} 

\subfloat[$\mathfrak a=(\sum_{i=1}^3\xi_i^{-\theta})^{-2},\theta=0.3$]{
\begin{minipage}[t]{0.35\linewidth}
\centering
\includegraphics[height=4.4cm,width=5.8cm]{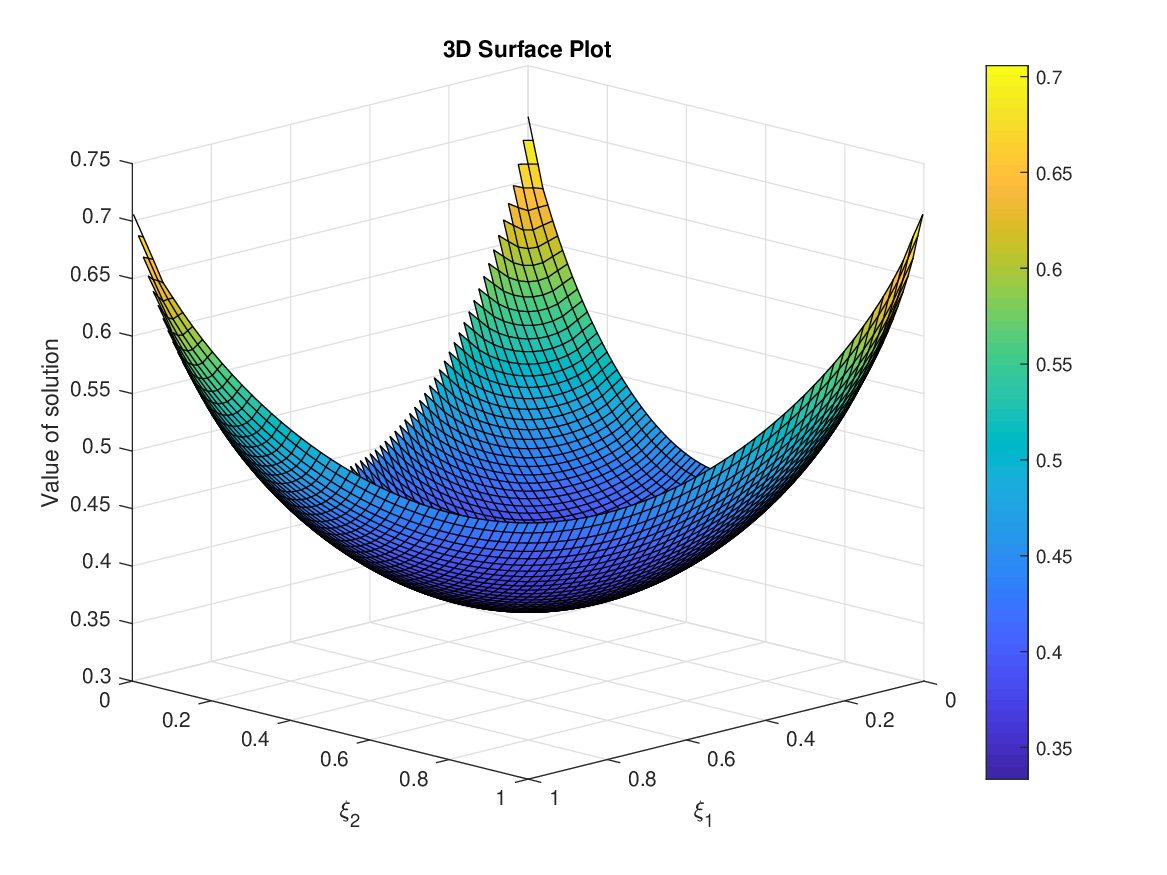}
\end{minipage}
}   
\subfloat[$\mathfrak a=(\sum_{i=1}^3\log(\xi_i))^{-2}$]{
\begin{minipage}[t]{0.35\linewidth}
\centering
\includegraphics[height=4.4cm,width=5.8cm]{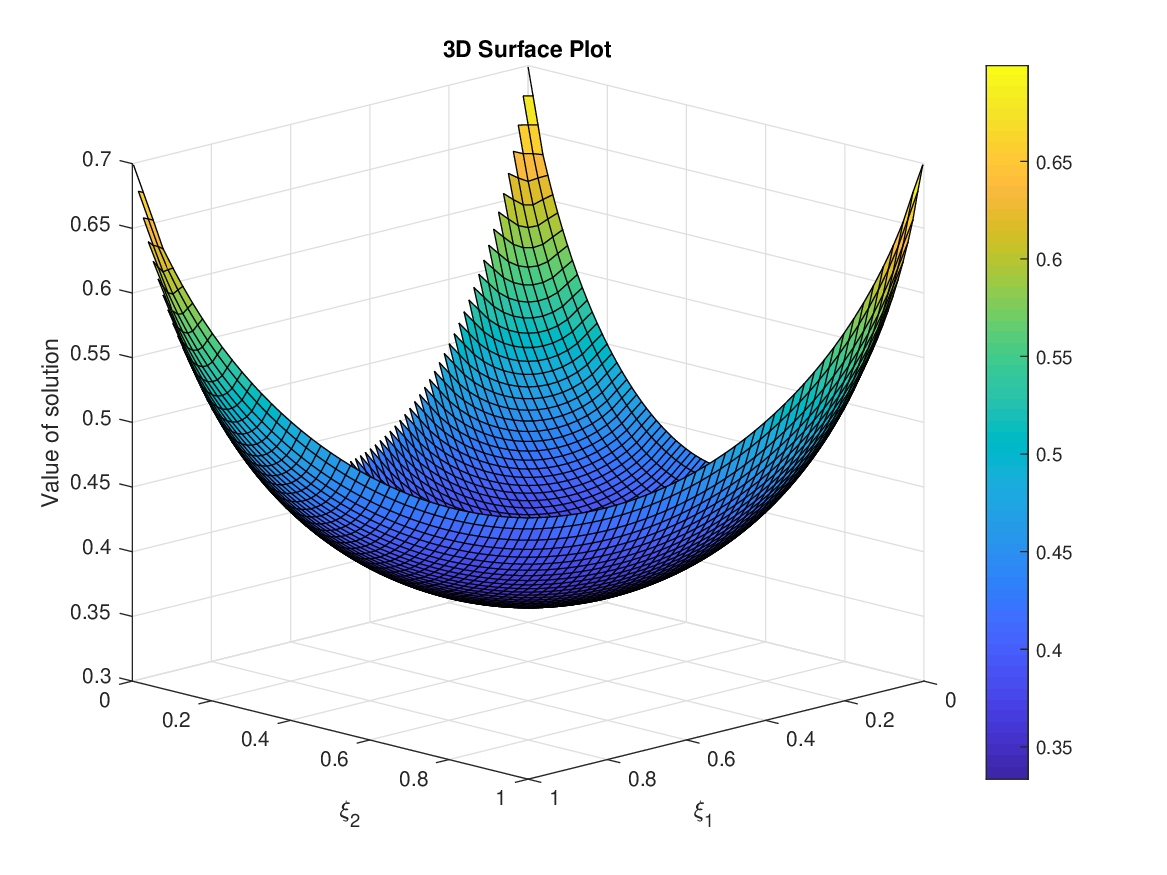}
\end{minipage}
}

\centering
\caption{Numerical solution for $\mathcal H(\xi,p)=\mathfrak a(\xi)\|p\|_{\xi}^2$, $\frac{\tau}{h}=0.1,h=0.0125,\epsilon=0.001,$ (A): initial value $\mathcal U_0(\xi)=\|\xi\|^2_{l^2}$, and (B)--(D): $T=5$.}\label{pic5}
\end{figure}

\end{appendix}

\end{document}